\documentclass{amsart}
\usepackage[utf8]{inputenc}

\numberwithin{equation}{section}

\usepackage{amssymb}
\usepackage{enumerate}
\usepackage{pgfplots} 
\pgfplotsset{compat=1.14}
\usepackage{etoolbox}
\usepackage[colorlinks]{hyperref}
\usepackage{bbm}    %more blackboard characters
\usepackage[normalem]{ulem}

\newtheorem{thm}{Theorem}[section]
\newtheorem{lem}[thm]{Lemma}

\newtheorem{cor}[thm]{Corollary}
\newtheorem{prop}[thm]{Proposition}

\newtheorem{rem}[thm]{Remark}

\newcommand{\cA}{{\mathcal A}}
\newcommand{\cB}{{\mathcal B}}
\newcommand{\cC}{{\mathcal C}} % Smooth functions
\newcommand{\cD}{{\mathcal D}}
\newcommand{\cE}{{\mathcal E}}

\newcommand{\cH}{{\mathcal H}} % Set of eigenvalues
\newcommand{\cI}{{\mathcal I}} % Union of partition elements
\newcommand{\cJ}{{\mathcal J}} % Map to polynomials
\newcommand{\cK}{{\mathcal K}} 
\newcommand{\cL}{{\mathcal L}} % Transfer operator
\newcommand{\cO}{{\mathcal O}}
\newcommand{\cP}{{\mathcal P}}
\newcommand{\cQ}{{\mathcal Q}}
\newcommand{\cM}{{\mathcal M}}  % Signed measures 
\newcommand{\cR}{{\mathcal R}}
\newcommand{\cT}{{\mathcal T}} % Matrix version of transfer operator 
\newcommand{\cU}{{\mathcal U}} 

\newcommand{\bC}{{\mathbb C}} % Complex numbers

\newcommand{\bN}{{\mathbb N}} % Natural numbers
\newcommand{\bM}{{\mathbb M}}

\newcommand{\bR}{{\mathbb R}} % Real numbers
\newcommand{\bT}{{\mathbb T}} % Torus
\newcommand{\bV}{{\mathbb V}}

\newcommand{\bZ}{{\mathbb Z}} % Integers

\newcommand\ve{\varepsilon}

\newcommand\vf{\varphi}

\newcommand{\Id}{\mathbbm{1}} % Identity
\newcommand{\bbh}{\mathbbm{h}}
\newcommand{\SBV}{{\textrm{SBV}}}
\DeclareMathOperator{\sign}{sign}

\newcommand{\dist}{\operatorname{dist}}
\newcommand{\Const}{C_{\sharp}}
\newcommand{\supp}{\operatorname{supp}} % Support
\newcommand{\srb}{\mu_{\scriptscriptstyle\operatorname{SRB}}} % Sinai-Ruelle-Bowen measure
\newcommand{\mme}{\mu_{\scriptscriptstyle\operatorname{BM}}} % Bowen-Margulis measure
\newcommand{\htop}{h_{\textrm{top}}}  % Topological entropy
\newcommand{\abs}[1]{\left|#1\right|} % Absolute value
\newcommand{\indicator}[1]{\mathbf{1}_{#1}} % Characteristic function
\newcommand{\poly}[1]{\mathfrak{P}_{#1}} % Polynomials
\newcommand{\lin}{\operatorname{lin}}

\begin{document}

\title{Locating Ruelle-Pollicott resonances}
\author{Oliver Butterley}
\author{Niloofar Kiamari}
\author{Carlangelo Liverani}
\address{Oliver Butterley\\
Dipartimento di Matematica\\
II Universit\`{a} di Roma (Tor Vergata)\\
Via della Ricerca Scientifica, 00133 Roma, Italy.}
\email{{\tt butterley@mat.uniroma2.it}}
\address{Niloofar Kiamari\\
Dipartimento di Matematica e Fisica\\
III Universit\`{a} di Roma\\
Largo San Leonardo Murialdo, 1 -- 00146 Roma, Italy.}
\email{{\tt niloofar.kiamari@uniroma3.it}}
\address{Carlangelo Liverani\\
Dipartimento di Matematica\\
II Universit\`{a} di Roma (Tor Vergata)\\
Via della Ricerca Scientifica -- 00133 Roma, Italy.}
\email{{\tt liverani@mat.uniroma2.it}}

\begin{abstract}
    We study the spectrum of transfer operators associated to various dynamical systems.
    Our aim is to obtain precise information on the discrete spectrum. To this end we propose a unitary approach.
    We consider various settings where new information can be obtained following different branches along the proposed path.
    These settings include affine expanding Markov maps, uniformly expanding Markov maps, non-uniformly expanding or simply monotone maps, hyperbolic diffeomorphisms. We believe this approach could be greatly generalized.
\end{abstract}

%\keywords{...}
%\subjclass{...}
\thanks{This work was partially supported by the PRIN Grant ``Regular and stochastic behaviour in dynamical systems" (PRIN 2017S35EHN).
    L.C. and O.B. acknowledge the MIUR Excellence Department Project awarded to the Department of Mathematics, University of Rome Tor Vergata, CUP E83C18000100006.
    We thank Péter Bálint, Viviane Baladi, Oscar Bandtlow, Giovanni Forni, Stefano Galatolo, Daniele Galli, Thomas Gilbert, Marco Lenci, Domokos Szász and Imre Péter Tóth for useful remarks.
    We are grateful to the anonymous referees whose comments and suggestions greatly improved this text.}
\maketitle

\setcounter{tocdepth}{1}
\tableofcontents
\section{Introduction}
Transfer operators are used widely in  Dynamical Systems. Their first manifestations going back, at least, to the Koopman  operator, and its use by John von Neumann to prove the mean ergodic theorem. Next, the Russian school developed the spectral theory for the Koopman operator acting on $L^2$ and its relation to the statistical properties of the system (such as ergodicity, mixing, ...), \cite{CFS82}. Later attention concentrated on the adjoint of the Koopman operator, now called the Frobenius-Perron or the Ruelle-Frobenius-Perron transfer operator. First such an operator appeared after coding the system \cite{Bowen75}. Subsequently, starting with \cite{LY73, Ruelle89}, the direct study of the transfer operator acting on functions and, more recently, starting with \cite{BKL02}, acting on spaces of distributions, acquired progressively more importance.\footnote{But see  \cite{RS75} for precursors of this point of view.}
This is the current focus.

Historically research was mostly focussed on the study of the peripheral spectrum (which encodes sharp, quantitative, information on ergodicity and mixing), on establishing a spectral gap (which yields the rate of mixing) (e.g, \cite{Liverani95,Baladi00}), and on estimates of the essential spectrum and relations with the Ruelle zeta function (which encodes information on the spectrum of periodic orbits), see \cite{Baladi18} for a recent review.
For flows or systems with a neutral direction the study is often more involved but there is still the possibility of some type of spectral gap, e.g., \cite{Dolgopyat98,Butterley16,BW20,BE17,TZ20,CL20}.

However, recently it has become apparent the need of a much deeper and detailed understanding of the point spectrum \cite{FT17, DZ16, DZ17, GF18, GL19, FGL19, BS20, KW20, GGHW20}.
Possibilities include identifying the point spectrum by understanding the connection to the action of the dynamics on cohomology \cite{FGL19}
or obtaining results related to bands of spectrum for transfer operators associated to systems with a neutral direction \cite{Faure07,FT13,FT15,FT17, BS20}.
Additionally various works investigated the possibility of an explicit description of the spectrum for analytic expanding or hyperbolic maps \cite{SBJ13, BJS17, SBJ17} (using Blaschke products), or perturbative and generic results \cite{KR04, Naud12, Adam17, BN19}. Clearly, a more explicit description of the spectrum is important also in applications as it provides precise quantitative information on the invariant measures, entropy, decay of correlation, variance in the CLT and so on.

Unfortunately, no general theory  exists to address this issue. One exception being the Hilbert metric technique, see \cite{Liverani95}, however such an approach can yield results only for the spectral gap and they are often far from optimal (see Remarks \ref{rem:hilbert0}, \ref{rem:hilbert} and \ref{rem:hilbert2}), hence the need for an alternative approach. The special cases in which some results have been obtained seem to point to a general philosophy:  to study the commutator between some type of differentiation and the transfer operator (e.g. see \cite{DFG15, GF18, FGL19}).  Although this idea is rather vague, we believe it can give rise to a general theory. In fact, it is surprising that this approach has not been explored in any systematic way, in spite of the vast literature devoted to transfer operators.  Hence, the first step to substantiate our claim is an investigation of several concrete examples. This is the task of the present work.

We apply the above philosophy to different classes of dynamical system, starting from very simple ones and increasing progressively in complexity. For each example we obtain non trivial results that illustrate the power of this approach to the problem. Although our results fall short of a general theory, we believe they suffice to argue that walking further along this path is likely to yield interesting results in a much more general setting.
Let us describe our results more in detail.

In section \ref{sec:piecewise} we discuss the point spectrum for a family of transfer operators of Markov piecewise affine expanding maps. This is the simplest possible, non trivial, example. Yet, it goes a long way in illustrating our strategy.

In section \ref{sec:smoothexp} we address similar questions in the case of full branch piecewise smooth expanding maps of the interval. The simplest non-linear case. This is a class of maps which has been extensively studied and for which one could expect that all has been said already. Yet, we are able to obtain new interesting information. In particular, we concentrate on two transfer operators. The one associated to the SRB measure for which we obtain effective bounds on the spectral gap and fine informations about the spectra. The other is the operator associated to the measure of maximal entropy for which we establish a spectral gap of size at least $e^{h_{\operatorname{top}}}-1$. This illustrates the fact, seen also later in other examples, that the transfer operator associated to the measure of maximal entropy enjoys  surprisingly large gaps.

In section \ref{sec:non-unifexp} we study the spectral gap for the operator associated to the measure of maximal entropy for full branch monotone maps. This includes the case of maps with attracting periodic orbits. We show that the measures of maximal entropy are exponentially mixing with a rate, at least, $\htop$. We are not aware of similar results. Apart from the case of intermittent maps (when only neutral fixed points are present) for which it is known to exist a unique measure of maximal entropy which is exponentially mixing. However, even in this special case, nothing quantitatively precise was known on the speed of mixing.

Finally, in section \ref{sec:hyp}, we study hyperbolic maps. We start, as an illustration, with automorphisms of the torus. This sheds some light on the difficulties involved in extending the approach to the general hyperbolic case. Next, we propose a possible solution to such difficulties: to study the spectrum of the action of the pushforward, for hyperbolic maps, on forms. This allows, for example, to study, again, the measure of maximal entropy. Once more we obtain a large gap. In particular, our approach provides alternative proofs, and a slight strengthening, of recent results by Baladi \cite[Theorem 2.1]{Baladi20} and Forni \cite{Forni20a}, moreover we establish a topological interpretation of the point spectrum which should hold in more generality.
%%%%%%%%%%%%%%%

\section{Affine Expanding Markov maps}
\label{sec:piecewise}
In this section we discuss the simplest possible case: one dimensional piecewise affine Markov maps. This allows us to show our approach in the simplest possible form and it is presented mainly for didactical purposes.
Indeed, the results of this section are essentially not new (for prior results see \cite{MSO81,BYOM84}, \cite[Section 2 \& Appendix]{SBJ13b} and also Remark~\ref{rem:use-sbj}).

In this setting the invariant densities can be computed easily since the Frobenius-Perron operator can be represented by a finite-dimensional matrix (see \cite[Chapter 9]{BG97} for full details).
Here we go beyond the peripheral spectrum and show that studying a particular family of matrices yields the full Ruelle-Pollicott spectrum.
To this end, the smoothness of the observables is relevant.
This will be a leitmotiv in the following and it is essential since it is known that  even the point spectrum of the transfer operator may change drastically if one considers a class of observables that allow discontinuities (e.g., see \cite{BSTV97, BKL02} and also Remark \ref{rem:disc_spect}).

Let $I:=[0,1]$ and let $f : I \rightarrow I$ be a \emph{piecewise  affine expanding Markov map} in the following sense:
there exists a collection of disjoint open intervals $\{I_j\}_{j=1}^N =\{(p_{j}, p_{j+1})\}_{j=1}^N$ which form a partition of a full measure subset of \(I\)
and,
for all $i,j$,
\[
    \text{either}
    \quad
    f(I_i) \cap I_j  = \emptyset,
    \quad \text{or} \quad
    I_j \subseteq f(I_i).
\]
Moreover, we suppose that \(f'\) is constant on each \(I_i\).
Finally we suppose that there exists \(\lambda > 1\) such that\footnote{ We consider only the transformations \(f\) which are orientation preserving. See Remark~\ref{rem:signed-derviative} concerning the general case $ \abs{f'}\geq \lambda$.} $ f'(x)\geq \lambda$ for all \(x\in \cup_i I_i\).

The partition \(\{I_i\}_{i=1}^N\) is called a \emph{Markov partition}.
Let $\cI= \cup_{i=1}^N I_i$ be the disjoint union of the partition elements.
The $N\times N$ matrix $A$ defined by
\begin{center}
    $A[i,j] = 1$, if $I_j \subseteq f(I_i)$, and $A[i,j] = 0$, if $f(I_i) \cap I_j = \emptyset$,
\end{center}
is called the {\em adjacency matrix}\footnote{ It is also called the \emph{incidence} matrix~\cite{BG97}.} of the Markov map $f$.
For convenience let \(\lambda_j := {f'|}_{I_j}\)
and \(\lambda = \min_j \lambda_j\).
For any\footnote{We use throughout the convention \(\bN := \{1,2,\ldots\} \) and \(\bN_0 := \{0,1,2,\ldots\} \).} \(k\in \bN_0\), let  $B_k$ be the $N\times N$ matrix defined by
\begin{equation}
    \label{eq:defB}
    B_k[i,j]:=\lambda_j^{-k} {A[j,i]}.
\end{equation}
If partition elements are equally sized then \(B_1\) is left stochastic, i.e., \(\sum_{i} B_1[i,j] =1\) for each \(j\).
In general there exists a diagonal matrix \(D\) such that \(D^{-1}B_{1}D\) is left stochastic~\cite[\S 9.3]{BG97}.

For simplicity, in the following theorem we additionally suppose that \(f\) is topologically transitive.
This means that there exists\footnote{ The existence of these invariant measures is well known in this context and also follows from the results later in this section.} a unique \(f\)-invariant probability measure which is absolutely continuous with respect to Lebesgue (denoted \(\srb\))
and a unique measure of maximal entropy (also known as the Bowen-Margulis measure) (denoted \(\mme\)). We let \(\htop\) denote the topological entropy.

Also, we use \(\cC^\infty(\cI)\) denote the set of functions on \(I\) which are \(\cC^\infty\) when restricted to each \(I_j\).

Finally, we use \(\sigma\) to denote the spectrum of a matrix and \(\sigma_{B}(\cL)\) for the spectrum of an operator \(\cL\) acting on a Banach space \(B\).

We can now state a result concerning Ruelle-Pollicott resonances.

\begin{thm}
    \label{thm:pa-ruelle}
    There exists a set of complex numbers \(\Xi_1=\{\xi_1, \xi_2, \cdots\}\) and, for each \(\xi_i \in \Xi_1\), an associated integer\footnote{ The integer \(m_{i}\) is the Jordan block dimension. A given \(\xi_i\) might be repeated in \(\Xi_{1}\) according to the geometric multiplicity.} \(m_i\) such that,
    for any \(\phi, \varphi \in \cC^\infty(\cI)\) and \(\epsilon >0\) there is an asymptotic expansion
    \[
        \int_{I} \phi \cdot \varphi \circ f^n \ d\srb
        =
        \sum_{ \xi_i \in \Xi_1 : |\xi_i| \geq \epsilon } \sum_{k=0}^{ m_i -1} \xi_i^n n^k C_{i,k}(\phi,\varphi)
        + o(\epsilon^n)
    \]
    where \(C_{i,k}(\phi,\varphi)\) are finite rank and non-zero bilinear functions of \(\phi,\varphi\).\\
    The set \(\Xi_1\) is equal (as a subset of \(\bC\)) to \({\bigcup_{l=1}^{\infty} \sigma(B_l)}\)
    and the equality holds also for the total multiplicity of each eigenvalue.\footnote{ More can be said about the multiplicities and Jordan blocks, see Theorem~\ref{thm:pa-spectrum} and Remark~\ref{rem:multiplicity}.}\\
    Similarly there exists a set of complex numbers \(\Xi_0=\{\xi_1, \xi_2, \cdots\}\) and for each \(\xi_i \in \Xi_0\) an associated integer \(m_i\) such that,
    for any \(\phi, \varphi \in \cC^\infty(\cI)\) and \(\epsilon >0\) there is an asymptotic expansion
    \[
        \int_{I} \phi \cdot \varphi \circ f^n \ d\mme
        =
        e^{-n \htop}
        \sum_{  \xi_i \in \Xi_0 : |\xi_i| \geq \epsilon  } \sum_{k=0}^{ m_i -1} \xi_i^n n^k C_{i,k}'(\phi,\varphi)
        + o(\epsilon^n)
    \]
    where \(C_{i,k}'(\phi,\varphi)\) are finite rank and non-zero bilinear functions of \(\phi,\varphi\).
    The set \(\Xi_0\) is equal (as a subset of \(\bC\)) to \({\bigcup_{l=0}^{\infty} \sigma(B_l)}\)
    and the equality holds also for the total multiplicity of each eigenvalue.
\end{thm}

\noindent
The proof of the above Theorem is included towards the end of the section and follows from a significantly stronger result (Theorem~\ref{thm:pa-spectrum}), described in terms of transfer operators, that needs some further preliminaries to be properly stated.

\begin{rem}
    The assumption of topological transitivity means that \(B_{1}\) is irreducible.
    Since also \(D^{-1}B_{1}D\) is left stochastic for some diagonal matrix \(D\) it follows that \(1\) is the leading eigenvalue of \(B_1\) and this eigenvalue has multiplicity \(1\).
    Moreover \(C_{1,0}(\phi,\varphi) = \int \phi \ d\srb \int \varphi \ d\srb\).
\end{rem}

\begin{rem}\label{rem:disc_spect}
    In the case where \(f\) has the form \(x\mapsto \kappa x \mod 1\) for some \(\kappa \in \{2,3,\ldots\}\) we could consider \(f\) as a smooth map of the circle. In this case, restricting our attention to observables which are smooth on the circle, the set of Ruelle-Pollicott resonances would reduce\footnote{ This can be seen by considering the action of the dynamics on Fourier series.} to \(\{0\}\).
    However, studying these same systems for observables that are smooth on the interval, we see a much more interesting spectrum,  see Remark \ref{rem:full_branch}.
\end{rem}

Observe that for any $r \geq 0$, the Sobolev space \(W^{r,1}(I)\) is the set of all \(h\in L^1(I)\) such that $h$ and all of its weak derivatives up to the $r$'th belong to $L^1(I)$.
Consider, for any $r \geq 0$, the space \(W^{r,1}(\cI)\) which is the set of all \(h\in L^1(I)\) such that, for each \(i\), the restriction of \(h\) to \(I_i\) is in \(W^{r,1}(I_i)\).
For convenience we write  \(h'\) and $h^{(l)}$ to mean the weak derivative and $l$ weak derivative respectively of $h$ restricted to $\cI$.
For each \(r\in \bN_0\) the space \(W^{r,1}(\cI)\) is a Banach space equipped with the norm
\[
    {\|h\|}_{r,1} =\sum_{l=0}^r  \int_{\cI} |h^{(l)}(x)|dx.
\]
In the following, to simplify notation, we will write $W_r$ for $W^{r,1}(\cI)$ and
we will write \({\|\cdot\|}_{r}\) for \({\|\cdot\|}_{r,1}\).
Observe that \(W_0\) coincides with \(L^{1}(I)\).

Since, by assumption, ${f|}_{I_j}$ is invertible on its range, let us call $g_j$ its inverse
(\(g_j := {{f|}_{I_j}}^{-1} \)).
The domain of \(g_j\) is the interval \(f(I_j)\) which might not be equal to the unit interval.
If \(f(I_j) = (0,1)\) for all \(j\) then \(f\) is said to be a \emph{full branch} map.
We can now define our main objects of investigation: the transfer operators.
For all $k \in \bN_0$, \(h\in L^1(I)\) and $x\in I_i$ we define\footnote{The second sum here is understood in the sense that, when the summands are defined on a subset of the full integral, they are extended to the full interval by taking the value zero where not defined.}
\[
    \cL_k h(x)
    :=\sum_{y\in f^{-1}(x)}\frac{h(y)}{{[f'(y)]}^k}
    =\sum_{j} B_k[i,j] \ h\circ g_j(x).
\]
Since \(f\) preserves the Markov partition, composition with an affine transformation preserves Sobolev space and the sum consists of a finite number of terms it follows that these operators are well defined as operators \(\cL_k : W_r \to W_r \).
Similarly they are well defined, by this same formula, on \(\cC^r(\cI) = \bigoplus_{i} \cC^{r}(I_i)\).

Observe that \(\cL_1\) coincides with the usual transfer operator: the dual of the Koopman operator.

We define $\poly{r}(\cI)$ to be the set of  polynomial functions\footnote{ Studying the action on polynomials was also used for the vertical direction in the pseudo Anosov case \cite{FGL19}.} of degree $r$ on each interval $I_j$.
Since \(f\) is piecewise affine, the space $\poly{r}(\cI)$ is invariant under \(\cL_k\) for each \(r,k\in \bN_0\).
Thus, it is natural to consider the finite rank operator \({\cL_k|}_{\poly{r}(\cI)}\).
\begin{thm}
    \label{thm:pa-spectrum}
    Let \(k\in \bN_0\), \(r\in \bN\).
    There exists a projector \(\Pi_{k,r}: W_r \to \poly{r}(\cI) \subset W_r \) such that the spectral radius of
    \(  \cL_k (\Id - \Pi_{k,r})\) on \(W_r\)  is not greater than \( \lambda^{-(k+r-1)}\).
    Moreover
    \[
        \sigma\big({\cL_k|}_{\poly{r}(\cI)}\big)
        =
        \bigcup_{l=0}^{r}\sigma(B_{k+l})
    \]
    and the multiplicity of each eigenvalue \(\xi \in \sigma\big({\cL_k|}_{\poly{r}(\cI)}\big)\) is equal to the sum of the multiplicities of \(\xi\) as eigenvalues of \(B_{k+l}\), \(l \in \{0,\ldots,r\}\).
\end{thm}

\noindent
The remainder of this section is devoted to the proofs of the two above theorems.

\begin{rem}
    This result tells nothing about the spectrum of \(\cL_k\) within the disk \(\{\abs{z}\leq  \lambda^{-(k+r-1)}\}\) more than the fact that the essential spectrum is contained somewhere within.
    Indeed a full disk of essential spectrum is expected (see \cite{CI91}).
\end{rem}

\begin{rem}
    \label{rem:use-sbj}
    As mentioned earlier, the results of this section are essential not new although our strategy differs substantially.
    Indeed Theorem~\ref{thm:pa-spectrum} can be deduced \cite[Proposition A.5]{SBJ13b} by the following argument.\footnote{We are grateful to the anonymous referee who brought the relevant references to our attention and who suggested this argument.}
    First observe that the space \(\cH\) in  \cite{SBJ13b} (which is a space of holomorphic functions) is densely and continuously embedded in \(W_r\). Consequently the part of the spectrum of \(\cL_k\) on \(W_r\) with modulus greater that the essential spectral radius coincides with the spectrum of \(\cL_k\) on \(\cH\) with modulus greater that the essential spectral radius.
    Note also that the result in the reference \cite[Proposition A.5]{SBJ13b} doesn't require \(f\) to be orientation preserving.
\end{rem}

The next equality is our key observation. Albeit very simple, the rest of the paper relies on it and variants thereof.

\begin{lem}
    \label{lem:derivatives}
    For all $k,r\in \bN_0$, $h\in W_r$ and $l \in \{0,\dots, r\} $,
    \begin{equation*}
        (\cL_k h)^{(l)}
        =\cL_{k+l} h^{(l)}.
    \end{equation*}
\end{lem}
\begin{proof}
    Fix $k, r\in \bN_0$.
    The claimed equality holds trivially in the case \(l=0\).
    Observe that, by chain rule, for all \(x\in \cI_{i}\), \(h\in \cC^\infty(\cI)\),
    \[
        (\cL_k h)'(x)
        =\sum_j \lambda_j \ B_{k}[i,j] h'\circ g_j(x)
        =\sum_j B_{k+1}[i,j] h'\circ g_j(x)
        =\cL_{k+1} h'(x).
    \]
    If we assume that, for some $l\geq 0$, the claimed equality holds, i.e., for all \(h\in \cC^\infty(\cI)\),
    \[
        (\cL_k h)^{(l)}(x)=\cL_{k+l} h^{(l)}(x),
    \]
    then, using the previous observation,
    \[
        (\cL_k h)^{(l+1)}(x)
        =(\cL_{k+l} h^{(l)})'(x)
        =(\cL_{k+l+1} h^{(l+1)})(x).
    \]
    The equality for all \(l\) follows by induction.
    Using the density of \(\cC^\infty(\cI)\) in \(W_r\) we obtain the result for \(h\in W_r\).
\end{proof}

\begin{rem}
    \label{rem:signed-derviative}
    In general we could allow the \(\lambda_j\) to be positive or negative.
    If \(\cL_1\) coincides with the operator associated to the SRB measure the derivative \(\lambda_j\) occurs with absolute value in the formula.
    However, as is clear from the proof of the above lemma, when the derivative occurs as a result of differentiating the sign of the derivative remains.
    This means that, if we are interested in \(\mme\) then we should consider \( B_k[i,j]:=\lambda_j^{-k} {A[j,i]}\)
    but, if we are interested in \(\srb\), then we should consider \( B_k[i,j]:=\lambda_j^{-(k-1)}\abs{\lambda_j}^{-1} {A[j,i]}\).
\end{rem}

To proceed we now prove a set of Lasota-Yorke inequalities for the operators $\cL_k : W_r \to W_r$.

Let \(\Gamma_0 :={\|f'\|}_{L^\infty}\)
and, for all \(k\in \bN\), let \(\Gamma_k := \lambda^{-(k-1)}\).

\begin{lem}
    \label{lem:LY}
    Let $k\in \bN_0$, $r\in \bN$.
    For all
    \(h\in W_r\),
    \begin{equation*}
        \begin{split}
            &{\|\cL_k h\|}_{r}\leq \Gamma_k {\| h\|}_{r}
            \\
            &{\|\cL_k h\|}_{r}\leq
            \lambda ^{-(k+r-1)}   {\| h\|}_{r} +  \Gamma_k  {\| h\|}_{r-1}.
        \end{split}
    \end{equation*}
    The first inequality also holds in the case \(r=0\).
\end{lem}
\begin{proof}
    We start by considering the case \(k\in\bN\).
    Let \(h \in W_r\).
    By definition of ${\Vert \cdot \Vert}_r$ and Lemma~\ref{lem:derivatives},
    \begin{equation*}
        \begin{split}
            {\|\cL_k h\|}_r
            &=\sum_{l=0}^r  \int_{\cI} |(\cL_k h)^{(l)}(x)|dx
            = \sum_{l=0}^r\int_{\cI} |\cL_{k+l} h^{(l)}(x)|dx \\
            &\leq \sum_{l=0}^r  \lambda ^{-(k+l-1)} \int_{\cI}\cL_1|h^{(l)}(x)|dx.
        \end{split}
    \end{equation*}
    Since, by the obvious change of variables,
    \( \int_{\cI}\cL_1|h^{(l)}(x)|dx = \int_{\cI}| h^{(l)}(x)|dx  \)
    the above implies that, for all \(r\in \bN_0\),
    \begin{equation}
        \label{eq:working1}
        {\|\cL_k h\|}_r
        \leq \sum_{l=0}^r  \lambda ^{-(k+l-1)} \int_{\cI}|h^{(l)}(x)|dx.
    \end{equation}
    That is,
    \(
    {\|\cL_k h\|}_r
    \leq  \lambda^{-(k-1)}{\| h\|}_r
    \)
    as required to prove the first inequality.
    Moreover, when \(r\geq 1\), the above \eqref{eq:working1} implies that (here we separate the term \(l=r\) from the rest of the sum)
    \[
        \begin{split}
            {\|\cL_k h\|}_{r}
            &\leq \lambda ^{-(k+r-1)} \int_{\cI} |h^{(r)}(x)|dx
            +  \lambda ^{-(k-1)} \sum_{l=0}^{r-1}\int_{\cI}|h^{(l)}(x)|dx \\
            &\leq \lambda^{-(k+r-1)}{\| h\|}_{r}+{ \lambda^{-(k-1)}}{\| h\|}_{r-1}
        \end{split}
    \]
    as required by the second estimate.

    To conclude we must consider the case  \(k=0\).
    First observe that, for any \(h\in\cC^\infty(\cI)\),
    \[
        \int_{\cI}|\cL_0h|(x)\ dx = \int_{\cI}|\cL_1 f' h|(x)\leq \int_{\cI}\left[\cL_1 |f'|\,| h|\right](x)\leq {\|f'\|}_{L^\infty} \int_{\cI}|h|(x)\ dx.
    \]
    Similar to the proof in the case \(k\in \bN\),
    by definition of the norm and Lemma~\ref{lem:derivatives},
    \begin{equation*}
        {\|\cL_0 h\|}_r
        =\sum_{l=0}^r  \int_{\cI} |(\cL_0 h)^{(l)}(x)|dx
        = \sum_{l=0}^r\int_{\cI} |\cL_{l} h^{(l)}(x)|dx.
    \end{equation*}
    This means that, for all \(r\in \bN_0\) (recall that \(\Gamma_0 ={\|f'\|}_{L^\infty}\)),
    \begin{equation*}
        {\|\cL_0 h\|}_r
        \leq \sum_{l=0}^r \lambda^{-l} \int_{\cI}  |\cL_{0} h^{(l)}(x)|dx
        \leq \Gamma_0 \sum_{l=0}^r \int_{\cI}  | h^{(l)}(x)|dx
    \end{equation*}
    and so proves the first inequality.
    On the other hand, now assuming that \(r\geq 1\),
    \begin{equation*}
        \begin{split}
            {\|\cL_0 h\|}_r
            &= \int_{\cI} |\cL_{r} h^{(r)}(x)|dx
            + \sum_{l=0}^{r-1} \int_{\cI} |\cL_{l} h^{(l)}(x)|dx  \\
            &\leq \lambda^{-(r-1)}\int_{\cI} |\cL_{1} h^{(r)}(x)|dx
            + \sum_{l=0}^{r-1} \int_{\cI} |\cL_{0} h^{(l)}(x)|dx.
        \end{split}
    \end{equation*}
    Consequently
    \begin{equation*}
        {\|\cL_0 h\|}_r
        \leq \lambda^{-(r-1)}\int_{\cI} | h^{(r)}(x)|dx
        + \Gamma_0 \sum_{l=0}^{r-1} \int_{\cI} |h^{(l)}(x)|dx.
    \end{equation*}
    Thus,
    \({\|\cL_0 h\|}_r
    \leq \lambda^{-(r-1)} {\|h\|}_{r}
    + \Gamma_0 \sum_{l=0}^{r-1} {\|h\|}_{r-1},\)
    as required.
\end{proof}

\begin{lem}
    \label{lem:spectrum0}
    Let \(k \in \bN_0\), \(r\in \bN\).
    The operator $\cL_k$ acting on $W_{r}$ has spectral radius bounded by $\Gamma_k$ and essential spectral radius bounded by $\lambda ^{-(k+r-1)}$.
\end{lem}
\begin{proof}
    The first inequality of Lemma~\ref{lem:LY} implies that the spectral radius is bounded by  $\Gamma_k$.
    For all \(r\in \bN\), \(W_{r}\) is compactly embedded in $W_{r-1}$.
    This means the Lasota-Yorke inequalities of Lemma~\ref{lem:LY} imply,
    by the argument of Hennion~\cite{Hennion93}, based on the formula of Nussbaum~\cite{Nussbaum70} (see \cite[Appendix B]{DKL} for a pedagogical illustration of the Hennion--Nussbaum theory), that the essential spectral radius of $ \cL_k $ is bounded by $\lambda ^{-(k+r-1)}$.
\end{proof}

For convenience we use the notation \(\cD : h \mapsto h'\).
For any \(k\in\bN\), \(\nu\in\bC\), let \(E_k(\nu)\) denote the generalised eigenspace for \(\cL_k\) associated to the eigenvalue \(\nu\).
I.e.,  \(E_k(\nu)\) is the set of \(h\) such that \((\cL_k-\nu)^m h = 0\) for some \(m\in\bN\).
An immediate consequence of Lemma~\ref{lem:derivatives} is the following commutation relation:
For any \(l,k,m\in \bN_0\), \(\nu\in\bC\), \(h\in W_r\)
\[
    \cD^l \circ (\cL_k - \nu)^m h
    =
    (\cL_{k+l} - \nu)^m \circ \cD^l h.
\]
This in turn means that
\begin{equation}
    \label{eq:pa-diff-map}
    \cD^l E_k(\nu) \subset E_{k+l}(\nu).
\end{equation}

\begin{proof}[{\bfseries Proof of the first statement of Theorem~\ref{thm:pa-spectrum}}]
    Let \(k\in \bN_{0}\), \(r\in\bN\).
    According to Lemma~\ref{lem:spectrum0} the essential spectral radius of \(\cL_k :  W_{r} \to W_{r}\) is not greater that $\lambda ^{-(k+r-1)}$.
    Fix some arbitrarily small $\epsilon>0$ and define
    \(\cH_{k,r}:=\{\nu\in \sigma_{W_{r}}(\cL_k), |\nu|> \lambda^{-(k+r-1)}+\epsilon\}\).
    For each \(\nu\in\cH_{k,r}\) let \(P_{\nu}\) denote the associated spectral projector and hence let \(\Pi_{k,r} := \sum_{\nu\in \cH_{k,r}}P_{\nu}\).
    Consequently
    \(  \cL_k - \cL_k\circ\Pi_{k,r}   : W_r \to W_r  \)
    has spectral radius not greater than \(\lambda^{-(k+r-1)}+\epsilon\).
    For any \(l\in\bN\) Lemma~\ref{lem:spectrum0} gives an upper bound of \(\lambda^{-(l-1)}\) for the spectral radius of \(\cL_l:W_1 \to W_1\) and so \(E_l(\nu)= \{0\}\) whenever \(|\nu| > \lambda^{-(l-1)}\).
    As observed above \eqref{eq:pa-diff-map}, differentiating \(r\) times takes the generalised eigenspace \(E_k(\nu)\) to the generalised eigenspace \(E_{k+r}(\nu)\) of the operator \(\cL_{k+r}\).
    However \(E_{k+r}(\nu) = \{0\}\) since $|\nu|> \lambda^{-(k+r-1)}$.
    This means that \(E_k(\nu) \subset \poly{r}(\cI)\) whenever \(\nu\in \cH_{k,r}\)
    and so we have shown that the image of \( \Pi_{k,r} \) is contained in \(  \cup_{\nu\in \cH_{k,r}} E_{k}(\nu) \subset \poly{r}(\cI)\).
    The claim follows by the arbitrariness of $\epsilon$.
\end{proof}

We can identify \(\bR^N\) with \(\poly{0}(\cI)\), the set of functions that are constant on each partition element, in the sense that we associate the function \(\sum_i a_i \indicator{I_i}\in \poly{0}(\cI)\) to each \(a = (a_i) \in \bR^N\) ( \(\indicator{A}\) denotes the characteristic function of the set \(A\)).

Let \(r \in \bN\).
The space \((\bR^{N})^{(r+1)}\) is identified with \(\poly{r}(\cI)\) as follows.
We use the notation \((a^0,a^1,\ldots , a^r) \in (\bR^{N})^{(r+1)}\) where  \(a^{j} = (a_{1}^{j},a_{2}^{j},\ldots,a_{N}^{j})\) for each \(j\).
Let\footnote{ Abusing notation we will often write the same symbol for \(a\in \bR^{N}\) and the corresponding \(a \in \poly{0}(\cI)\) with the interpretation given by context.}
\(\cJ : \bR^{N(r+1)} \to \poly{r}(\cI)\),
\[
    \cJ(a^0,\ldots , a^r): x
    \mapsto
    \sum_{l=0}^{r}  x^{l} \sum_{j=1}^{N}  a^l_j \indicator{I_j}(x).
\]
Observe that \(\cJ : \bR^{N(r+1)} \to \poly{r}(\cI)\) is onto and invertible.

For any \(k\in \bN_0\), \(r \in \bN\) we define the \(N(r+1) \times N(r+1) \) matrix
\[
    {\cT}_{k,r} := \cJ^{-1} \circ {\cL_k|}_{\poly{r}(\cI)} \circ   \cJ.
\]
In order to understand the spectrum of  \({\cL_k|}_{\poly{r}(\cI)}\) it suffices to study the spectrum of the matrix \({\cT}_{k,r}\).

\begin{lem}
    \label{lem:triangle}
    Let \(k\in \bN_0\), \(r \in \bN\).
    Then \(\cT_{k,r}\) has lower block triangular form
    \[
        {\cT}_{k,r} = \begin{pmatrix}
            B_{k}   &         &        & 0       \\
            F_{1,0} & B_{k+1} &        &         \\
            \vdots  & \vdots  & \ddots &         \\
            F_{r,0} & F_{r,1} & \ldots & B_{k+r}
        \end{pmatrix},
    \]
    where the matrices on the diagonal are the ones previously introduced in \eqref{eq:defB}.
\end{lem}
\begin{proof}
    Fix \(k\in \bN_0\), \(r \in \bN\).
    For any \(l\in \{0,\ldots,r\}\)
    we consider \((a^0,a^1,\ldots, a^r) \in \bR^{N(r+1)} \) and suppose that \(a^j =0\) for all \(j\neq l\).
    This means that \(\cJ (a^0,a^1,\ldots, a^r) = \mathbf{x}^{l} a\).
    We wish to compute \(\cJ^{-1} \circ \cL_k \circ   \cJ (a^0,a^1,\ldots, a^r) = \cJ^{-1} \circ \cL_k (\mathbf{x}^{l} a)\).
    For each \(j\) let \(q_j := f(p_j^+)\) (i.e., \(\lim_{\epsilon\to 0}f(p_j + \epsilon)\)).
    Observe that, for all \(x\in I_j\),
    \(  f(x) = \lambda_j(x-{ p_{j}}) + q_j \).
    Consequently, for all \(x\in f(I_j)\),
    \begin{equation}
        \label{eq:inverse}
        { g_j(x)} = (x - q_j)\lambda_j^{-1} + { p_{j}}.
    \end{equation}
    Using this inverse,
    \begin{equation}
        \label{eq:Lkxa}
        \begin{aligned}
            \cL_{k} (\mathbf{x}^{l}a)(x)
             & =\sum_{i,j} \indicator{I_i}(x) B_k [i,j]a_j \left(x\lambda_j^{-1} -q_j{\lambda_j^{-1}}+{ p_{j}}\right)^{l}
            \\
             & =\sum_{i,j} \indicator{I_i}(x) B_{l+k} [i,j]a_j x^{l} + \rho(x),
        \end{aligned}
    \end{equation}
    where \(\rho \in \poly{l-1}(\cI)\).
    I.e.,
    \(\cL_{k} (\mathbf{x}^{l}a^l) =  \mathbf{x}^{l} B_{k+l} a^l + \rho\),
    where $ \rho \in \poly{l-1}(\cI)$.
    This proves that \({\cT}_{k,r}\) has lower diagonal block form and that the diagonal elements of the block matrix are the \(B_{k+l}\).
    The exact form of the matrices \(F_{i,j}\) which appear below the diagonal are superfluous to our present argument and we won't identify them further.
\end{proof}

\begin{proof}[{\bfseries Proof of Theorem~\ref{thm:pa-spectrum}}]
    The first statement of the theorem was proven above, it remains to prove the second statement.
    Recall the lower triangular block form of \({\cT}_{k,r}\) as shown in Lemma~\ref{lem:triangle}.
    We can assume without loss of generality that each \(B_{k+l}\) is in lower triangular form.
    If a matrix is in triangular form then the values on the diagonal are the eigenvalues repeated according to multiplicity.
    That each \(B_{k+l}\) is in triangular form means that the \(N(r+1) \times N(r+1)\) matrix \({\cT}_{k,r}\) is in triangular form.
    Moreover the diagonal is the union of the diagonals of the \(B_{k+l}\).
    This implies the claimed correspondence of the eigenvalues of \({\cT}_{k,r}\) and the union of the set of eigenvalues of the \({\{B_{k+l}\}_{l=0}^{r}}\), including correspondence in multiplicity.
\end{proof}

\begin{rem}
    \label{rem:multiplicity}
    The lower triangular block form shown in Lemma~\ref{lem:triangle} and the argument of the above proof further implies that, if some   \(B_{k+l}\) has a Jordan block of dimension \(m\in \bN\), then  \({\cT}_{k,r}\) has a corresponding Jordan block of dimension \(m\) or greater.
    Indeed \({\cT}_{k,r}\) has the possibility to have a Jordan block of greater dimension if a given eigenvalue appears in more than one of the \(B_{k+l}\).
\end{rem}

\begin{proof}[{\bfseries Proof of Theorem~\ref{thm:pa-ruelle}}]
    Fix \(k\in \bN_0\).
    For each $r \in \bN$ consider the finite set of eigenvalues
    \[
        {\{\xi_{j}\}}_{j=0}^{K_r}
        =
        \sigma_{W_r} (\cL_k)
        \setminus
        \{|z|\leq \lambda ^{-(k+r-1)}\}
    \]
    described by Theorem~\ref{thm:pa-spectrum}.
    We define as usual the corresponding eigen projectors \({\{\Pi_{j}:W_r \to W_r\}}_{j=0}^{K_r}\) and eigen nilpotents \({\{Q_{j}:W_r \to W_r\}}_{j=0}^{K_r}\)
    which satisfy the commutation relations: \(\Pi_{j}\Pi_{k} = \delta_{j,k}\), \(\Pi_{j}Q_{k} =Q_{k}\Pi_{j} = \delta_{j,k}Q_{k}\).
    Let \(S_r : = \Id - (\Pi_1+\Pi_2+\cdots + \Pi_{K_r})\) and observe that \( \cL_k S_r\) has spectral radius not greater than \(\lambda ^{-(k+r-1)}\).
    This means that the operator \(  \cL_k:  W_r \to  W_r\) satisfies the decomposition
    \begin{equation}
        \label{eq:Decomposition}
        \cL_k
        =
        \sum_{j=1}^{K_r} \left(\xi_{j}\Pi_{j}  + Q_{j}\right)
        +  \cL_k S_{r}.
    \end{equation}
    Further observe that each operator we define remains defined by the same formula on \(W_r\) for any \(r\) sufficiently large.

    Now let us recall the connection between the transfer operators and invariant measures (see \cite{LSV98} or \cite{Baladi00} for full details).
    For each \(k\in \{0,1\}\) there exists \(h_k\in W_r\) (\emph{the invariant density}), a probability measure \(\nu_k\) (\emph{the conformal measure}), \(\gamma_k>0\) (\emph{the spectral radius}) and a probability measure \(\mu_k\) defined as \(\mu_k(\varphi):=\nu_k(h_k \varphi)\) (\emph{the invariant measure}).
    Moreover
    \(\nu_k(\cL_k \varphi) = \gamma_k \nu_k(\varphi)\)
    and
    \(\mu_k(\varphi\circ f) = \mu_k(\varphi)\).

    In our present setting \(\mu_0\) is the measure of maximal entropy \(\mme\) and \(\mu_1\) is the SRB measure \(\srb\).
    Furthermore \(\ln \gamma_0\) is equal to the topological entropy, \(\gamma_1 = 1\) and \(\nu_1\) coincides with Lebesgue measure.

    Continuing for  \(k\in \{0,1\}\) we observe that
    \[
        \begin{aligned}
            \int_{I} \phi \cdot \varphi \circ f^n \ d\mu_k
             & = \int_{I} (\phi \cdot h_k)(x) \cdot \varphi \circ f^n(x) \ d\nu_k(x)             \\
             & = \gamma_k^{-n} \int_{I} \cL_k^n(\phi \cdot h_k)(x) \cdot \varphi (x) \ d\nu_k(x)
        \end{aligned}
    \]
    We then combine this formula with the spectral decomposition above \eqref{eq:Decomposition} to produce the asymptotic expansion required.
\end{proof}

\begin{rem}\label{rem:full_branch}
    If \(f\) is full branch, the matrices $B_{k+l}$ are such that all the entries in any column $j$ is equal to ${\lambda_j^{-(k+l)}}$.
    The spectrum of this type of matrix is the union of zero and the sum of entries on different columns.
    Consequently Theorem~\ref{thm:pa-spectrum} implies that, outside of the disk
    \(\lbrace\vert \nu\vert \leq \lambda ^{-(k+r-1)}\rbrace\),
    the spectrum of \(\cL_k :W_r \to W_r \)
    is equal to \(\lbrace \xi_0, \ldots, \xi_{r -1} \rbrace\)
    where $ \xi_l:={\sum}_{j=1}^N \lambda_j^{-(k+l)} $.
    \begin{itemize}
        \item In the case \(k=0\) we obtain \(\xi_0 = \sum_{j=1}^{N} \lambda_j^{0} = N\);
        \item In the case \(k=1\) we see that \(\xi_0 = \sum_{j=1}^{N} \lambda_j^{-1} =  \sum_{j=1}^{N} |I_j| = 1\).
    \end{itemize}
    A comprehensive investigation of the resonances of \(x \mapsto 2 x \mod 1\) can be found in \cite[\S3]{Driebe99}).
    The eigenfunctions for this map are the Bernoulli polynomials.
\end{rem}

\begin{rem}
    Observe that \(B_0\) is the transpose of \(A\) and   that, for any Markov map as considered in the present section, the logarithm of the spectral radius of \(B_0\) is equal~\cite[\S2.1]{Buzzi10} to the topological entropy.
\end{rem}

\begin{rem}
    In this section we used Sobolev spaces \(W_r\) but, with a slightly more complex argument, we could equally well have worked with \(\cC^r(\cI) = \bigoplus_{i} \cC^{r}(I_i)\).
\end{rem}

\subsection{A Jordan block example}\ \\
In the following we construct an example of a Markov expanding map such that \(B_1\) has a Jordan block of dimension two.
Previously Driebe~\cite{Driebe97} presented an example of a linear Markov expanding map such that, for each \(n\in \bN\), the eigenvalue \(3^{-2n}\) has a Jordan block of dimension two. However the example isn't orientation-preserving and the mechanism giving rise to the Jordan blocks appears to depend on this.
On the contrary, the example constructed here is orientation-preserving.

Let \(I_1 = (0,\frac{1}{4})\), \(I_2 = (\frac{1}{4},\frac{1}{2})\), \(I_3 = (\frac{1}{2},\frac{3}{4})\), \(I_4 = (\frac{3}{4},1)\)
and let \(f:I\to I\) be as shown in Figure~\ref{fig:jordan}, defined by
\[
    f(x):=
    \begin{cases}
        3 x + \frac{1}{4}   & \text{if \(x \in I_1\)}  \\
        3 (x - \frac{1}{4}) & \text{if \(x \in I_2\)}  \\
        2 (x - \frac{1}{2}) & \text{if \(x \in I_3\)}  \\
        3 (x - \frac{3}{4}) & \text{if \(x \in I_4\)}.
    \end{cases}
\]
% https://www.wolframalpha.com/input/?i=jordan+normal+form+calculator&assumption=%7B%22F%22%2C+%22JordanDecompositionCalculator%22%2C+%22theMatrix%22%7D+-%3E%22%7B%7B0%2C1%2F3%2C1%2F2%2C1%2F3%7D%2C%7B1%2F3%2C1%2F3%2C1%2F2%2C1%2F3%7D%2C%7B1%2F3%2C1%2F3%2C0%2C1%2F3%7D%2C%7B1%2F3%2C0%2C0%2C0%7D%7D%22
This means that
\[
    A = \begin{pmatrix}
        0 & 1 & 1 & 1 \\
        1 & 1 & 1 & 0 \\
        1 & 1 & 0 & 0 \\
        1 & 1 & 1 & 0
    \end{pmatrix},
    \quad
    B_1 = \begin{pmatrix}
        0           & \frac{1}{3} & \frac{1}{2} & \frac{1}{3} \\[3pt]
        \frac{1}{3} & \frac{1}{3} & \frac{1}{2} & \frac{1}{3} \\[3pt]
        \frac{1}{3} & \frac{1}{3} & 0           & \frac{1}{3} \\[3pt]
        \frac{1}{3} & 0           & 0           & 0
    \end{pmatrix}.
\]
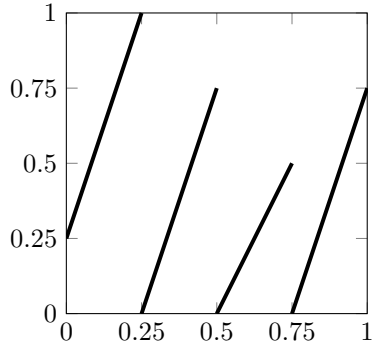
\begin{figure}[htbp]
    \begin{center}
        \begin{tikzpicture}
            \begin{axis}[
                    xmin=0, xmax=1, ymin=0, ymax=1,
                    xtick distance=.25, ytick distance=.25,
                    x=4cm, y=4cm]
                \addplot[line width=1.5, domain=0:.25, samples=21]{3*x + 0.25};
                \addplot[line width=1.5, domain=.25:.5, samples=21]{3*x - 3/4};
                \addplot[line width=1.5, domain=.5:.75, samples=21]{2*x -1};
                \addplot[line width=1.5, domain=.75:1, samples=21]{3*x -9/4};
            \end{axis}
        \end{tikzpicture}
    \end{center}
    \caption{An orientation-preserving Markov map with Jordan block}
    \label{fig:jordan}
\end{figure}
The matrix \(B_1\) has the eigenvalues \(\{-\frac{1}{3},0,1\}\) and the eigenvalue \(-\frac{1}{3}\) has a Jordan block of dimension two.
If we let
\[
    a_1 := \left(\begin{smallmatrix}
            -1\\
            0\\
            0\\
            1
        \end{smallmatrix}\right), \quad
    a_2 := \left(\begin{smallmatrix}
            3\\
            3\\
            -6\\
            0
        \end{smallmatrix}\right), \quad
    a_3 := \left(\begin{smallmatrix}
            0\\
            -1\\
            0\\
            1
        \end{smallmatrix}\right), \quad
    a_4 := \left(\begin{smallmatrix}
            9\\
            12\\
            8\\
            3
        \end{smallmatrix}\right),
\]
then \(B_1 a_1 = -\frac{1}{3} a_1\),
\((B_1 + \frac{1}{3}\Id)a_2 = a_1\),
\(B_1 a_3 = 0\)
and \(B_1 a_4 = a_4\).
In particular \(\{a_1,a_2\}\) span the generalised eigenspace associated to the eigenvalue \(-\frac{1}{3}\).
Since the essential spectral radius of $\cL_1$, when acting on $W_r$, $r\geq 2$, is smaller than $1/4$, then $\cL_1$, on such spaces, has a Jordan block in the point spectrum.

\begin{rem}
    Another interesting example of an affine expanding Markov map is the Baladi map studied in \cite{CE04}. This is a system which exhibits non-trivial complex resonances. In the reference the connection between resonances and decay of correlation was considered and the outer set of resonances identified. Our results give a description of the full set of resonances for this system and the connection to the decay of correlations.
\end{rem}

%%%%%%%%%%%%%%%%%%%%%%%%%%%%%%%%%%%%%%%%%
\section{Piecewise smooth full branch expanding maps}
\label{sec:smoothexp}
In this section we discuss the simplest non-linear case: full branch maps. For such maps there exists already some general quantitative results on the spectral gap, e.g., \cite[Section 2]{Liverani95}, however they are not optimal, we will comment about the comparison case by case.

Let $f\in \cC^r([0,1], [0,1])$, $r\geq 2$, be a full branched piecewise expanding map, $f'> \lambda>1$. For $k\in\bN_0$ let us consider the transfer operator
\begin{equation}
    \label{eq:trans-op}
    \cL_k h(x)=\sum_{y\in f^{-1}(x)} \frac{h(y)}{[f'(y)]^k}.
\end{equation}
Observe that $\cL_0$ is the operator associated to the measure of maximal entropy while $\cL_1$ is the operator associated to the SRB measure \cite{Baladi00}.\footnote{ For the measure of maximal entropy see also the beginning of section \ref{sec:non-unifexp} where it is explained in a  more general setting.}

For convenience, throughout this section, we denote the \emph{distortion} by\footnote{ Here, as in the previous section, the derivative is taken only at the smoothness points of $f$.}
\[D_{f}=\left(\frac{1}{f'}\right)'.\]
The key fact we wish to leverage on, in analogy with Lemma~\ref{lem:derivatives}, is the  formula
\begin{equation}\label{eq:derivL}
    (\cL_k h )' =\cL_{k+1} h'+k\cL_k(h \cdot D_{f})
\end{equation}
which is obtained simply by differentiating \eqref{eq:trans-op}.
Notice that the distortion $D_f$ is 0 when the map is piecewise affne and this is the reason why the
treatment in Section 3 is more complicated than in Section 2 (See Lemma \ref{lem:derivatives}). Amazingly, the above formula yields several non trivial facts. To illustrate its power we start discussing the operator $\cL_0$.
\subsection{Measure of maximal entropy}
\begin{lem}\label{lem:gapone} If $f$ is a $N$ covering, then $\sigma_{\cC^1}(\cL_0)\subset \{ N\}\cup\{z\in\bC\;:\; |z|\leq 1\}$.
    Moreover, $\sigma_{\cC^2}(\cL_0)\cap\{z\in\cC\;:\; |z|\geq \lambda^{-1}\}=\{ N\}\cup(\sigma_{\cC^1}(\cL_1)\cap\{z\in\bC\;:\; |z|\geq \lambda^{-1}\})$
\end{lem}
\begin{proof}
    Note that $\cL_0 1= N$, so $ N\in\sigma(\cL_0)$.
    Since
    \[
        \left(\frac{1}{(f^n)'}\right)'=\sum_{k=0}^{n-1}\frac{D_{f}\circ f^k}{(f^{n-k-1})'\circ f^{k+1}},
    \]
    $\Big| \Big(\frac{1}{(f^n)'}\Big)'\Big|\leq  C = {{\|D_{f}\|}_\infty}({1-\lambda^{-1}})^{-1}$. In particular we may use \eqref{eq:derivL} for $f^n$, rather than for $f$ and \({\|\cL_1\|}_{\infty}\leq C\).
    Using these estimates, \eqref{eq:derivL} and a direct computation
    \[
        \begin{split}
            {\|\cL_0 h\|}_\infty
            &\leq N {\|h\|}_\infty,\\
            {\|(\cL_0^n h)'\|}_\infty
            &\leq {\|\cL_1^nh'\|}_\infty\leq C{\|h'\|}_\infty.
        \end{split}
    \]
    By the usual arguments the above inequalities imply that the essential spectral radius of $\cL_0$, acting on $\cC^1$, is at most one. On the other hand if $\cL_0h=\nu h$, with $|\nu|> 1$ we have, for all $n\in\bN$,
    \[
        \nu^nh'=(\cL_0^nh)'=\cL_1^n h'
    \]
    which, since \({\|\cL_1^n\|}_{L^1}=1\), implies $|h'|=0$, that is, $h$ must be constant.

    To conclude observe that on the one hand, if $\cL_0 H=\nu H$, $H\in\cC^2$, then $\cL_1 H'=\nu H'$. On the other hand,  if $\cL_1 h=\nu h$, $h\in\cC^1$, let $H_0(x)=\int_0^x h(y) dy$ and observe that $H_0\in \cC^2$ and
    \[
        \left((\nu-\cL_0)H_0\right)' =(\nu-\cL_1)h=0.
    \]
    Thus $(\nu-\cL_0)H_0$ is a constant function. Let $\alpha$ denote this constant value and let $c=-(\nu- N)^{-1}\alpha$.
    Now let $H_c(x)=\int_0^x h(y) dy + c$ and observe that
    \[
        (\nu-\cL_0)H_c=(\nu-\cL_0)H_0+(\nu- \cL_0)c =\alpha+(\nu- N)c = 0.
    \]
    The result follows since the essential spectrum of $\cL_1$, when acting on $\cC^1$, is bounded by $\lambda^{-1}$.

\end{proof}
\begin{rem}\label{rem:hilbert0}
    Note that the proof of Lemma \ref{lem:gapone} implies that $\cL_0 h=\int h \ d\mme +Qh$, where $\mme$ is the measure of maximal entropy, and ${\|Q^n\|}_{\cC^1}\leq C$. That is, $\cL_0$ has a spectral gap $N-1$ while the Hilbert metric technique  can yield, at the very best, a spectral gap $N -\lambda$, see \cite{Liverani95}.
\end{rem}

The above shows that the spectrum of $\cL_0$ is largely determined by the spectrum of $\cL_1$. Hence, before continuing our investigation of the spectrum of $\cL_0$, it is necessary to undertake an investigation of the spectrum of $\cL_1$.
%%%%%%%%%%%%%%%%%%%%%%%%%%%%%%%
\subsection{ The SRB measure}\ \\
Note that the vector space $\bV=\{h\in\cC^1\;:\; \int_0^1 h=0\}$ is invariant under $\cL_1$, we can thus restrict $\cL_1$ to $\bV$.
If we define
\begin{equation}\label{eq:phi}
    \phi(g)(x)=\int_0^x g(y) dy-\int_0^1 (1-y)g(y)dy=\int_0^x yg(y)dy+\int_x^1(y-1)g(y)dy,
\end{equation}
then $\phi:\cC^0\to\bV$ and \(\phi(h')=h\) for all \(h\in \bV\).
Thus, for each $h\in\bV$,
\begin{equation}\label{eq:structure}
    (\cL_1 h )'    =\cL_2 h'+\cL_1(\phi(h') D_{f})=:\cL_\star(h').
\end{equation}
The relevance of the operator $\cL_\star$ rests in the next Lemma.
\begin{lem}
    \label{lem:Lstar}
    If $f\in\cC^2 ([0,1],[0,1])$, then the spectrum of $\cL_1:\cC^1\to\cC^1$ satisfies
    \[
        \sigma_{\cC^1}(\cL_1)\cap\{z\in\bC\;:\;|z|>\lambda^{-1}\}
        =\{1\}\cup\sigma_{\cC^0}(\cL_\star)
        \cap\{z\in\bC\;:\;|z|>\lambda^{-1}\}.
    \]
\end{lem}
\begin{proof}
    It is well known that the essential spectral radius of $\cL_1$, when acting on $\cC^1$ is bounded by $\lambda^{-1}$, hence we can restrict ourselves to the point spectrum.

    Since $\int_{0}^{1}\vf\cL_1 h=\int_{0}^{1}(\vf\circ f) \cdot h$, it follows that the Lebesgue measure is an eigenvector, with eigenvalue one, of the dual operator, and hence $1\in \sigma_{\cC^1}(\cL_1)$. In addition, $\bV=\{h\in\cC^1\;:\; \int_0^1 h=0\}$ is invariant under $\cL_1$. It follows that if $\cL_1 h =\nu h$, $|\nu|>\lambda^{-1}$ and $h\in \cC^1$, then $h'\in \cC^0$ and  \eqref{eq:structure} implies $\cL_\star h'=\nu h'$. On the other hand if $g\in \cC^0$ and $\cL_\star g =\nu g$, $|\nu|> \lambda^{-1}$, then $h=\phi(g)\in \bV$ and
    \[
        ( \cL_1h-\nu h )' =  \cL_\star g-\nu g=0.
    \]
    Hence, there exists a constant $C$ such that $\cL_1h-\nu h=C$, but integrating we have $C=0$, thus $h$ is an eigenvector of $\cL_1$.
\end{proof}
\begin{rem} Note that the above Lemma holds verbatim with $W^{1,1}$ substituted to $\cC^1$. In the following, we find more convenient to consider the spectrum of $\cL_1$ when acting on $W^{1,1}$.
\end{rem}
Let
\[
    \tau=\lambda^{-1} +\int_0^1\left|\left(\frac{1}{f'(y)}\right)'\right|  dy  =\lambda^{-1}+ \|D_{f}\|_{L^1}.
\]
\begin{lem} \label{lem:Phi-est}
    The norm of $\cL_2$ on $L^1$ is bounded by $\lambda^{-1}$. The operator $\cL_c(g)=\cL_1(\phi(g) D_{f})$, acting on $L^1$, is a  compact operator.
    In addition, for all $g\in L^1$
    \[
        \begin{split}
            {\|\phi(g)\|}_{L^1}&\leq\frac 12 {\|g\|}_{L^1}\\
            {\|\phi(g)\|}_{L^\infty}&\leq {\|g\|}_{L^1}\\
            {\|\phi(g)\|}_{W^{1,1}}&\leq \frac 32{\|g\|}_{L^1}.
        \end{split}
    \]
    In particular \(\cL_{*} : L^{1} \to L^{1}\) is a well-defined operator and \({\|\cL_{*}\|}_{L^1} \leq \tau\).
\end{lem}
\begin{proof}
    Since
    \[
        {\|\cL_2 h\|}_{L^1}\leq \lambda^{-1}{\|\cL_1 h\|}_{L^1}\leq  \lambda^{-1}{\|h\|}_{L^1},
    \]
    the first statement follows. Moreover,
    \[
        \begin{split}
            \int_0^1|\phi(g)(x)|dx
            &\leq 2\int_0^1|g(y)|y(1-y) dy\leq \frac 12{\|g\|}_{L^1},\\
            |\phi(g)(x)|
            &\leq\int_0^1|g(y)| dy={\|g\|}_{L^1}.
        \end{split}
    \]
    Finally,  ${\|\phi(g)'\|}_{L^1}\leq {\|g\|}_{L^1}$ implies the last of the three inequalities and also that $\phi$ is compact, the compactness of $\cL_c$ follows.
\end{proof}
\begin{thm} \label{thm:simple-bound} Let us consider $\cL_1$ as an operator acting on $W^{1,1}$, then $\sigma_{\textrm{ess}}(\cL_1)\subset \{z\in\bC\;:\; |z|\le \lambda^{-1}\}$.
    Moreover $\sigma(\cL_1)\setminus \{1\}\subset \{z\in\bC\;:\; |z|\le \tau\}$.
\end{thm}
\begin{proof}
    If $\nu\in\bC$ is such that $|\nu|>\lambda^{-1}$ and $\cL_1 h=\nu h$, for some $h\in W^{1,1}$ with $\int h=0$, then $\cL_\star g=\nu g$, for $g=h'$.
    Then, recalling Lemma \ref{lem:Phi-est} ,
    \[
        \begin{split}
            |\nu|\|g\|_{L^1}&\leq\lambda^{-1}\|g\|_{L^1}+\int_0^1\left|\left(\frac 1{f'(y)}\right)'\right|  dy \ {\| \phi(g)\|}_{L^\infty}\\
            &\leq \left[\lambda^{-1}+\|D_{f}\|_{L^1} \right] {\|g\|}_{L^1}.
        \end{split}
    \]
    This proves the theorem since $h'=0$ implies $h=0$.
\end{proof}
The above lemma provides an upper bound for the spectral gap, but it is very unsatisfactory. First, such a bound is of interest only if $\tau<1$ (for example, in the counterexample of Keller, Rugh \cite{KR04} $\tau>1$).   Second, even if $\tau<1$, it is unclear if there exists other point spectrum outside $\{z\in\bC\;:\;|z|\leq \lambda^{-1}\}$.
\begin{rem}
    \label{rem:hilbert}
    For $\cL_1$ the Hilbert metric approach yields a bound of the spectral gap given by a rather cumbersome formula. However, if one considers the limit of large $\lambda$ and small $D_{f}$, then, using \cite[Lemma 2.3]{Liverani95}, one can check that the bound of the spectral gap cannot be better than $\lambda^{-1}(1+2\|D_{f}\|_\infty)+\|D_{f}\|_\infty$, which is worse than the one provided, in the same limit, by Theorem \ref{thm:simple-bound}. However, for large $D_{f}$ the bound of Theorem \ref{thm:simple-bound} is useless while  \cite[Lemma 2.3]{Liverani95} provides an explicit, although rather poor, bound.
\end{rem}
Very few results are known on the existence of point spectrum with the notable exception of cases when the map has been explicitly constructed  to exhibit point spectrum \cite{KR04} or when one restricts the map to the class of holomorphic maps, often of a special nature, as in \cite{BJS17, SBJ13, SBJ17}. No analytical technique is available to treat $\cC^2$ open classes of maps. On the contrary a lot of work exists on the side of numerical computation, mainly of the invariant measure but also, to some extent, of the spectrum, e.g., see \cite{Liverani01} and references therein. While most of the numerical work does not track round off errors and hence it is unsatisfactory from the rigorous point of view, some notable exceptions use interval arithmetic and hence have the status of a proof, e.g., \cite{BGNX16, GN16, Ippei11}.

Hence, it is interesting to note that the present approach offers an alternative, possibly much more convenient,  route to a numerical computation of the spectrum.
\begin{rem}
    \label{rem:numerical}
    We conclude the section with a remark on how the above discussion can provide a numerical scheme to locate eigenvalues.
    Let $\cK:= \cL_1(\phi(g)D_{f})$, $\phi$ being defined in \eqref{eq:phi}. Also, let $\alpha>0$ and $\{\vf_i\}_{i=1}^\infty$ be a Schauder base of $W^{1,1}$ such that, calling $\Pi_N$ the projection onto $\operatorname{span}\{\vf_i\}_{i=1}^N$ along $\operatorname{span}\{\vf_i\}_{i=N+1}^\infty$,  we have $\|\Pi_N\|_{L^{1}}\leq \Const $ and
    \[
        \|\Id-\Pi_N\|_{W^{1,1}\to L^1}\leq \Const N^{-\alpha}.
    \]
    Then, to study the spectrum of $\nu- \cL_\star$, $1>|\nu|> \lambda^{-1}$ when acting on $L^1$, write
    \[
        \begin{split}
            & \nu-\cL_\star
            =(\nu-\cL_2)\left[\Id-(\nu-\cL_2)^{-1}\cK \right]=(\nu-\cL_2)\left[\Id-\Pi_N(\nu-\cL_2)^{-1}\Pi_N\cK-\Delta_N \right]\\
            &=(\nu-\cL_2)\left[\Id-\Pi_N(\nu-\cL_2)^{-1}\cK\Pi_N\right] \left(\Id-\left[\Id-\Pi_N(\nu-\cL_2)^{-1}\cK\Pi_N\right]^{-1}\Delta_N\right),\\
            & \Delta_N= (\Id-\Pi_N)(\nu-\cL_2)^{-1}\cK +\Pi_N(\nu-\cL_2)^{-1}(\Id-\Pi_N)\cK.
        \end{split}
    \]
    Note that Lemma \ref{lem:Phi-est} implies $\|\cK\|_{L^1}+\|\cK\|_{L^1\to W^{1,1}}\leq \Const$, hence
    \[
        {\|\Delta_N\|}_{L^1}\leq \Const (\lambda^{-1}-|\nu|)^{-1} N^{-\alpha}.
    \]
    Thus $\nu$ belongs to the resolvent of $\cL_\star$ if
    \[
        \|\left[\Id-\Pi_N(\nu-\cL_2)^{-1}\cK\Pi_N\right]^{-1}\Delta_N\|_{L^1}\leq \Const \frac{{\|\left[\Id-\Pi_N(\nu-\cL_2)^{-1}\Pi_N\cK\right]^{-1}\|}_{L^1}}{(\lambda^{-1}-|\nu|) N^\alpha}\leq 1.
    \]
    Since $\Pi_N(\nu-\cL_2)^{-1}\Pi_N\cK$ is a finite rank operator, $\left[\Id-\Pi_N(\nu-\cL_2)^{-1}\Pi_N\cK\right]^{-1}$ its norm can be evaluated numerically. In fact, by Neumann series, we have, for $\zeta$ small and setting $\cR_N=\Pi_N(\nu-\cL_2)^{-1}\Pi_N$,
    \[
        \left[\Id-\zeta\cR_N\cK\right]^{-1}=\Id+\zeta\left[\Id-\zeta\cR_N\cK\Pi_N\right]^{-1}\cR_N\cK,
    \]
    thus, by analyticity, the same holds for $\zeta=1$.
    Hence, the spectrum of $\cL_\star$ is close to the values of $\nu$ for which $\cR_N\cK\Pi_N$ has eigenvalue one.
    Since  $\cR_N\cK\Pi_N$ is a $N\times N$ matrix, this provides a rather quick way to determine rigorously if $\cL_1$ has point spectrum outside the spectral radius of $\cL_2$, aside from one.
\end{rem}

\subsection{Point spectrum}\ \\
If we consider class of maps with some special features, it is possible use arguments like the ones put forward in Remark \ref{rem:numerical} to obtain relevant information about the point spectrum without any computer assisted method.

As an example, the next Theorem provides more precise information on the spectrum in a special class of maps. Note that the following approach can be generalised, here we present only the simplest application to illustrate the logic of the argument.
\begin{thm}
    \label{thm:large-gap1}
    Let $I:=[0,1]$ and $f:I\to I$. Consider the partition $\{(p_i,p_{i+1})\}_{i=1}^{N}$ to be a partition of a full-measure subset of $[0,1]$ such that for any $1\leq i\leq N$, $f([p_i,p_{i+1}])=[0,1]$, $f\in\cC^3([p_i,p_{i+1}],[0,1])$, and $f'(p_i^+)=f'(p_{i}^-)$, $i\in \{2,\dots, N\}$.\footnote{ By $g(p^+)$,  $g(p^-)$ we mean the right and left limit, respectively, of the function \(g\) at the point \(p\). Since $f'(p_i^+)=f'(p_{i}^-)$ there is no need to distinguish between $p_i^-$ and $p_i^+$, so we will not do it anymore.} Also assume that $D_{f}\not\equiv 0$ and $D_{f}\geq 0$. Then,\footnote{Note that the following provides a spectral gap if $f'(1)\geq 2$.} for $\cL_1:W^{2,1}([0,1])\to W^{2,1}([0,1])$
    \[
        \begin{split}
            &\sigma(\cL_1)\subset\left\{z\in\bC\;:\; |z|\leq  \min\big\{1, \tfrac 2{f'(1)}-\tfrac 1{f'(0)} \big\}\right\}\cup \left\{1\right\}\\
            &\sigma_{\textrm{ess}}(\cL_1)\subset\left\{z\in\bC\;:\; |z|\leq  \tfrac 1{f'(1)^2}\right\}.
        \end{split}
    \]
    Moreover, $\left\{1\right\}$ is a simple eigenvalue of $\cL_1$. In addition, there exists $\mu_2< \tfrac 1{f'(1)}$ such that, setting $\Delta=\frac1{f'(1)}-\frac1{f'(0)}$, $\mu_*= \frac1{f'(1)}$, $\Gamma=\left(1-\sum_{i=1}^N\frac1{f'(p_i)}\right)$ and
    \[
        \begin{split}
            &A_0=\left\{a\in\bR\;:\; \mu_2<a< 1\right\}\\
            &A_1=\left\{a+ib\in\bC\;:\; a>\mu_*, b^2<\frac{(a-\mu_*)\Delta}{2(1+a^{-1}\Gamma)}\left[\sqrt{1+4\frac{(1+a^{-1}\Gamma)^2(a-\mu_*)^2}{\Delta^2}}-1\right]\right\}\\
            &A_2=\left\{a+ib\in\bR\;:\; a<-\mu_*,\, b^2< \frac{(|a|-\mu_*)^2(a^2-\mu_*^2-\Delta|a|)}{a^2-\mu_*^2+\Delta|a|}\right\}\\
            &A_3=\left\{a+ib\in\bC\;:\;a\geq 0,\, |a+ib-\mu_*|^2>\mu_*^2+\mu_*\frac{\Delta+\sqrt{4\mu_*^2+\Delta^2}}2\right\}\\
            &A_4=\left\{a+ib\in\bC\;:\;a< 0,\, b^2>\mu_*\Delta+\mu_*^2+2|a|\mu_*-a^2\right\}
        \end{split}
    \]
    we have
    \[
        \left(\cup_{i=0}^3A_i\right)\cap [\sigma(\cL_1)\setminus\{1\}]=\emptyset.
    \]
\end{thm}
The proof of the above Theorem is a boring computation using the ideas illustrate in the previous section, so we postpone it to Appendix \ref{app:one}.
Here we provide an application of the Theorem to show what can be achieved with some, moderately lengthy, hand made computations. Much more could be obtained using the assistance of a computer as mentioned in Remark  \ref{rem:numerical}.

\begin{rem} Note that if $D_{f}\equiv 0$, then $f'(1)=N$ and $\cL_1$ has eigenvalue $N^{-1}$ with eigenfunction $g(x)=x-\frac{1}{2}$. Indeed,
    \[
        \begin{split}
            \cL_1 g(x)&=\sum_{i=1}^{N} N^{-1}g\left(\frac{x+i-1}N\right)=\sum_{i=1}^{N}\frac{\left(x+i-1\right)}{N^2}-\frac{1}{2N}\\
            &=N^{-1}\left( x+\frac{N-1}{2}-\frac{N}{2}\right)=Ng(x).
        \end{split}
    \]
    By perturbation theory, see \cite{KL99}, it follows that such an eigenvalue survives for small distortion. However, the above theorem implies that, for perturbations satisfying Theorem \ref{thm:large-gap1}, one cannot make it increase more than $\tfrac 2{f'(1)}-\tfrac 1{f'(0)}$.
\end{rem}

\begin{rem}
    As an example consider $f(x)=4x-x^2\mod 1$. In this case Theorem~\ref{thm:large-gap1} applies with $D_{f}=\frac{2}{(4-2x)^2}>0$, $f'(0)=4$, $f'(1)=2$, $\Delta=\frac{1}{4}$, $\mu_*= \frac{1}{2}$ and some  $\mu_2< \frac{1}{2}$.  Moreover $p_1=0$, $p_2=2-\sqrt 3$, $p_3=2-\sqrt 2$, hence
    $f'(p_2)=2\sqrt 3$,  $f'(p_3)=2\sqrt 2$ and $\Gamma=1-\frac 14-\frac1{2\sqrt 3}-\frac 1{2\sqrt 2}$.
    Consequently,
    for $\cL_1:W^{2,1}([0,1])\to W^{2,1}([0,1])$,
    \[
        \begin{split}
            &\sigma(\cL_1)\subset\left\{z\in\bC\;:\; |z|\leq  \tfrac 3{4}\right\}\cup \left\{1\right\}\\
            &\sigma_{\textrm{ess}}(\cL_1)\subset\left\{z\in\bC\;:\; |z|\leq  \tfrac 1{4}\right\}.
        \end{split}
    \]
    Moreover, the sets
    \[
        \begin{split}
            &A_0=\left\{a\in\bR\;:\; \mu_2<a< 1\right\}\\
            &A_1=\left\{a+ib\in\bC\;:\; a>\frac{1}{2}, \, b^2<\frac{(a-\frac{1}{2})}{8(1+a^{-1}\Gamma)}\left[\sqrt{1+64(1+a^{-1}\Gamma)^2(a-\frac{1}{2})^2}-1\right]\right\}\\
            &A_2=\left\{a+ib\in\bR\;:\; a<-\frac{1}{2},\, b^2< \frac{(|a|-\frac{1}{2})^2(4a^2-1-|a|)}{4 a^2-1 + {|a|}}\right\}\\
            &A_3=\left\{a+ib\in\bC\;:\;a\geq 0, \, \vert a+ib-\frac{1}{2}\vert^2 >\frac{5+\sqrt{17}}{16}\right\}\\
            &A_4=\left\{a+ib\in\bC\;:\;a< 0, \, b^2>\frac{3}{8} + |a| - a^2 \right\}
        \end{split}
    \]
    contain no spectrum of \(\cL_1\) except \(1\).
    These regions are illustrated in Figure~\ref{fig:A-regions}.
\end{rem}

\usetikzlibrary{math}
\tikzmath{ \gnum = 0.75 - 1/(2*sqrt(3)) - 1/(2*sqrt(2));}

\begin{figure}[tbp]
    \centering
    \noindent
    \begin{tikzpicture}[scale=4, domain=-1:1]
        \fill[fill=blue!10!white] (-1.1,-1.1) rectangle (1.1,1.1);
        \fill[fill=white] (0,0) circle (0.75 cm);
        % Circles with labels
        \draw[color=gray!30!white, thin] (0,0) circle (0.5 cm);
        \path (0.5,-0.11) node (l2) {\(\frac{1}{2}\)};
        \draw (0.5,-0.05) -- (0.5,0.05);
        \draw[color=blue, very thick] (0,0) circle (0.75 cm);
        \path (0.75,-0.11) node (l2) {\(\frac{3}{4}\)};
        \draw (0.75,-0.05) -- (0.75,0.05);
        \draw[color=gray!30!white, thin] (0,0) circle (1 cm);
        \path (1,-0.11) node (l1) {\(1\)};
        \draw (1,-0.05) -- (1,0.05);
        % Essential spectrum
        \fill[color=gray!30!white, thin] (0,0) circle (0.25 cm);
        \path (0.25,-0.11) node (l3) {\(\frac{1}{4}\)};
        \draw (0.25,-0.05) -- (0.25,0.05);
        % Leading eigenvalue
        \fill (0:1 cm) circle (0.02);
        % A0
        \draw[very thick, draw=blue] (0.45,0) node[above=2pt] {$A_0$} --  (3/4,0);
        % A1
        \fill[fill=blue!10!white] plot[domain=0.5:0.8]
        (\x,{sqrt( (\x-0.5)*(sqrt(1+4^3*(1+\x^(-1)*\gnum )^2*(\x-0.5)^2)-1)/(8*(1+\x^(-1)*\gnum ))})
        |- (0.8,0) ;
        \draw[draw=blue, very thick] plot[domain=0.5:0.75]
        (\x,{sqrt( (\x-0.5)*(sqrt(1+4^3*(1+\x^(-1)*\gnum )^2*(\x-0.5)^2)-1)/(8*(1+\x^(-1)*\gnum ))})
        node[below right] {$A_1$} ;
        \fill[fill=blue!10!white] plot[domain=0.5:0.8]
        ((\x,{-sqrt( (\x-0.5)*(sqrt(1+4^3*(1+\x^(-1)*\gnum )^2*(\x-0.5)^2)-1)/(8*(1+\x^(-1)*\gnum ))})  |- (0.8,0) ;
        \draw[very thick, blue] plot[domain=0.5:0.75]
        ((\x,{-sqrt( (\x-0.5)*(sqrt(1+4^3*(1+\x^(-1)*\gnum )^2*(\x-0.5)^2)-1)/(8*(1+\x^(-1)*\gnum ))})  ;
        % A2
        \fill[very thick, fill=blue!10!white] plot[domain=-0.8:-0.5] (\x,{sqrt(  (\x + 0.5)^2 * (4*\x^2 - 1 + \x  )/ (4*\x^2 - 1 - \x )  )})
        |- (-0.8,0) node[right,above=3pt] {$A_2$};
        \fill[very thick,  fill=blue!10!white] plot[domain=-0.8:-0.5] (\x,{-sqrt(  (\x + 0.5)^2 * (4*\x^2 - 1 + \x  )/ (4*\x^2 - 1 - \x )  )})  |- (-0.8,0);
        \draw[very thick, draw=blue] plot[domain=-0.8:-0.5]
        (\x,{sqrt(  (\x + 0.5)^2 * (4*\x^2 - 1 + \x  )/ (4*\x^2 - 1 - \x )  )});
        %  \draw[very thick,  draw=blue] plot[domain=-0.8:-0.5] (\x,{-sqrt( (abs(\x)-0.5)^2 * (\x^2 - 1/4-(1/4)*abs(\x)) /(\x^2 - 0.25 + abs(\x)/4 )   ))});
        \draw[very thick,  draw=blue] plot[domain=-0.8:-0.5]
        (\x,{-sqrt(  (\x + 0.5)^2 * (4*\x^2 - 1 + \x  )/ (4*\x^2 - 1 - \x )  )});
        % A3
        \fill[fill=blue!10!white]  plot[domain=0:0.3]
        (\x,{sqrt(\x - \x^2 + (1 + sqrt(17))/16)})
        |- (0,0.8)
        node[right] {$A_3$};
        \draw[very thick, blue] plot[domain=0.0:0.3]
        (\x,{sqrt(\x - \x^2 + (1 + sqrt(17))/16)});
        \fill[fill=blue!10!white]  plot[domain=0:0.3]
        (\x,{-sqrt(\x - \x^2 + (1 + sqrt(17))/16)})
        |- (0,-0.8)
        node[right] {$A_3$};
        \draw[very thick, blue] plot[domain=0.0:0.3]
        (\x,{-sqrt(\x - \x^2 + (1 + sqrt(17))/16)})  ;
        % A3 tick marks
        \draw (-0.05,0.56586) node[left] {\footnotesize \(\sim 0.56\)} -- (0.05,0.56586) ;
        \draw (-0.05,-0.56586) node[left] {\footnotesize \(\sim -0.56\)} -- (0.05,-0.56586) ;
        % A4
        \fill[fill=blue!10!white]  plot[domain=-0.2:0]
        (\x,{sqrt(3/8 - \x - \x^2)})
        |- (0,0.8)
        node[left] {$A_4$};
        \draw[very thick, blue] plot[domain=-0.2:0]
        (\x,{sqrt(3/8 - \x - \x^2)});
        \fill[fill=blue!10!white]  plot[domain=-0.2:0]
        (\x,{-sqrt(3/8 - \x - \x^2)})
        |- (0,-0.8)
        node[left] {$A_4$};
        \draw[very thick, blue] plot[domain=-0.2:0]
        (\x,{-sqrt(3/8 - \x - \x^2)});
        % A4 tick marks
        \draw (-0.05,0.6123) node[left] {\footnotesize \(\sim 0.61\)} -- (0.05,0.6123) ;
        \draw (-0.05,-0.6123) node[left] {\footnotesize \(\sim -0.61\)} -- (0.05,-0.6123) ;
        % Axis
        \draw (-1.1,0) -- (1.1,0);
        \draw (0,-1.1) -- (0,1.1);
        \path (225:1.2cm) node (v0) {\(\mathbb{C}\)};
    \end{tikzpicture}
    \caption{Spectrum of the example. Theorem~\ref{thm:large-gap1} implies that the disk of radius 1/4 (central shaded region) contains the essential spectrum while the eigenvalues, except for 1, cannot belong to the exterior shaded region.}
    \label{fig:A-regions}
\end{figure}
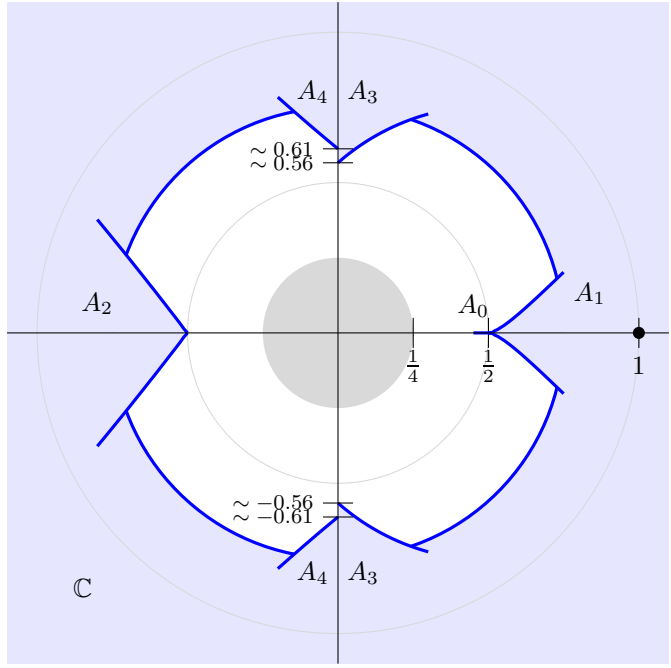

\noindent
%%%%%%%%%%%%%%%%%%%%%%%%%%%%%%%%%%%%%%%%%

\subsubsection{\bfseries Different operators}As a last comment on the present approach to the study of the spectrum of $\cL_1$, let us remark that it is possible to investigate the commutation relations with different operators. As an example, let us consider the operator $A(h)=h'+\alpha h$ for some function $\alpha$. Then
\begin{equation}\label{eq:altpf}
    A\cL_1 h=\cL_2 h'+ \cL_1 (D_{f}h) + \cL_1 ((\alpha\circ f) h)=\cL_2 (Ah)+\cL_1\left[(D_{f}-\frac{\alpha}{f'}+\alpha\circ f)h\right].
\end{equation}
In general, is not obvious what the best choice of $\alpha$ could be. To keep the discussion short let us consider only the special, well known, case in which $\ln f'$ is $\cC^1$ cohomologous to a constant.\footnote{ This happens if $f$ is $\cC^2$ conjugated to a map $f_\ell(x)=\ell x\mod 1$, $\ell\in\bZ$ with $|\ell|\geq 2$.} That is, there exists a $\cC^1$ function $B$ such that
\[
    \ln f'+B-B\circ f=c.
\]
Then we can choose $\alpha=B'$ and obtain
\[
    A\cL_1 h=\cL_2 (Ah).
\]
Accordingly, if there exists $h\in W^{1,1}$ such that $\cL_1 h=\nu h$, $|\nu|> \lambda^{-1}$, then $\cL_2(Ah)=\nu Ah$, thus $Ah=0$ (since the spectral radius of $\cL_2$ on $L^1$ is bounded by $\lambda^{-1}$). This implies that $h=e^{-B}a$, $a\in\bC$, hence
\[
    \cL_1e^{-B}=\cL_0e^{-c-B\circ f}=e^{-c-B}\cL_01=e^{-c}N e^{-B}.
\]
Integrating yields $e^{-c} N=1$, hence $\nu=1$. It follows that
\[
    \sigma_{W^{1,1}}(\cL_1)\subset\{1\}\cup \{z\in\bC\;:\; |z|\leq \lambda^{-1}\},
\]
hence, as expected, the existence of a large spectral gap.

In the general case, one could try to minimise $D_{f}-\frac{\alpha}{f'}+\alpha\circ f$ in order to produce estimates that improve  Theorem~\ref{thm:large-gap1}.

%%%%%%%%%%%%%%%%%%%%%%%%%%%%%%%%%%%%%%%%%%%%%%%%%
\section{Piecewise monotone maps}
\label{sec:non-unifexp}
Up to now we have considered uniformly expanding systems. However much of our arguments were rather general, it is then natural to ask if one can apply the present philosophy also to non-uniformly expanding maps or even maps that expand only in some part of the phase space. We believe the answer to be affirmative and to justify such a belief we discuss one of the simplest possibilities: one dimensional piecewise monotone map. Note however that we will develop the full theory only for one dimensional full branch monotone maps (See \cite{Milnor88} for full details on the related theory). Of course, for such more general systems one cannot expect to prove as many results as in the previous section. Yet, some interesting and novel results can be obtained. To illustrate such a fact we will discuss the operator associated to the measures of maximal entropy.

Let $\cP=\{I_1,\dots, I_N\}$, $N\in\bN$, be a partition of $[0,1]$ in the sense that the $I_i$ are open disjoint sets and $\cup_{i=1}^N \overline{I_i}=[0,1]$.
Let $f:[0,1]\to[0,1]$ be a map such that $f|_{I_i}$ is strictly monotone and $f|_{I_i}\in\cC^1$. Thus each point in $(0,1)$ has at most $N$ preimages. Suppose that $\Lambda={\|f'\|}_\infty<\infty$. We write $\bM_*$ for the set of maps satisfying the above properties.
\begin{rem} Note that maps in $\bM_*$ can have attracting fixed points or attracting periodic orbits and can be multimodal.
\end{rem}
\begin{rem} Note that we ask only $f|_{I_i}\in\cC^1$, rather than $f|_{I_i}\in\cC^{1+\alpha}$ as is necessary when studying the SRB measure.
\end{rem}

We start by studying the spectral and essential spectral radius of $\cL_0$. For each $h\in L^1$,
\begin{equation}
    \label{eq:deriv1}
    {\|\cL_0 h\|}_{L^1} \leq \int_0^1 \cL_1|f'h|=\int |f'| |h|\leq \Lambda {\|h\|}_{L^1}.
\end{equation}
On the other hand, for $h\in BV$ and $\vf\in\cC^1$ we have, calling $\phi_i$ the inverse of $f|_{I_i}$,
\[
    \begin{split}
        \int_0^1\vf'(x)\cL_0h(x) x&=\sum_{i=1}^N\int_0^1\vf'(x) \Id_{f(I_i)}(x)h(\phi_i(x)) dx\\
        &=\sum_{i=1}^N\int_{I_i}\vf'(f(x)) |f'(x)| h(x) dx\\
        &=\sum_{i=1}^N\int_{I_i}\frac{d}{dx}\left[\sign( f'(x))\cdot \vf\circ f(x))\right] h(x) dx.
    \end{split}
\]
We can then define the counterm, recall that $I_i=(p_i,p_{i+1})$,
\[
    \psi(x)=\sum_{i=1}^n\Id_{I_i}(x)\sign(f'(x)) \left\{\vf(f(p_{i}))+\frac{\vf(f(p_{i+1}))-\vf(f(p_{i}))}{p_{i+1}-p_i}(x-p_i)\right\}
\]
and $\vf_*(x)=\sign(f'(x))\vf(f(x)) -\psi(x)$. Note that, by construction, $\vf_*$ is Lipschitz, $\|\vf\|_{\cC^0}\leq 2\|\vf\|_{\cC^0}$  and
\begin{equation}\label{eq:weak_der0}
    \begin{split}
        \int_0^1\vf'(x)\cL_0h(x) x=&\int_0^1\vf_*'(x) h(x) dx\\
        &+\sum_{i=1}^N\int_{I_i}\sign( f'(x)) \frac{\vf(f(p_{i+1}))-\vf(f(p_{i}))}{p_{i+1}-p_i} h(x) dx.
    \end{split}
\end{equation}
Accordingly, there exists $C>0$ such that
\begin{equation}
    \label{eq:deriv2}
    \|\cL_0 h\|_{BV} \leq 2\|h\|_{BV}+C \| h\|_{L^1}.
\end{equation}
If we apply the above to the map $f^n$, rather than $f$, we have, since also $f\in\bM_*$, the Lasota-Yorke inequality
\begin{equation}
    \label{eq:LY}
    \begin{split}
        \|\cL_0 h\|_{L^1}&\leq  \Lambda {\|h\|}_{L^1},\\
        \|\cL_0^n h\|_{BV}&\leq 2\| h\|_{BV}+ C_n\|h\|_{L^1},
    \end{split}
\end{equation}
for each $n\in\bN$ and some constants $C_n>0$.

From \eqref{eq:LY}  and Hennion's Theorem \cite{Hennion93} (see also \cite[Appendix B]{DKL}) it follows that the spectral radius, on $BV$, of $\cL_0$ is bounded by $\Lambda$ while the essential spectral radius is bounded by one.\footnote{ In \eqref{eq:LY} choose $n_0$ such that $2^{1/n_0}\leq \Lambda$ and iterate with steps $n_0$.}

This establishes the first step of our strategy. Next we have to consider the commutation between the derivative and $\cL_0$. Since BV functions have weak derivatives that are measures, this makes sense, however such measures can be rather singular. To simplify matters it seems convenient to restrict our functions to SBV (special bounded variation) functions, that is functions for which the singular part
of the weak derivative is supported at most on countably many point, not a Cantor
set. Recall that SBV is a closed subspace of BV \cite[Corollary 4.3]{AFP}.
\begin{lem}\label{lem:closure} It holds true $\cL_0(\SBV)\subset \SBV$. The essential spectrum of $\cL_0$ is bounded by one. Moreover, the eigenvectors associated to eigenvalues of modulus strictly larger than one have zero absolutely continuous part. Hence, if a power of $f$ is a Markov map, then the eigenvectors associated to eigenvalues of modulus strictly larger than one are constant on the elements of the Markov partition.
\end{lem}
\begin{proof}
    Let $h\in \textrm{SBV}$, then we can write the weak derivative, seen as a measure, as $Dh=\sum_{a\in\cA}\alpha_a\delta_a+\bbh dx$, for some $\bbh\in W^{1,1}$,  a countable set $\cA$ and numbers $\alpha_a\in\bC$ such that $\sum_{a\in\cA}|\alpha_a|<\infty$, see \cite[Corollary 3.33]{AFP}. Then \eqref{eq:weak_der0}  implies, setting $\cA_0=\cA\setminus\{p_i\}_{i=1}^N$ and $\epsilon_i=\sign(f'(x))$ for $x\in I_i$,
    \[
        \begin{split}
            &\int_0^1\vf'(x)\cL_0h(x) dx=-\int_0^1\vf_*(x) Dh(dx)+\sum_{i=1}^N \epsilon_i\frac{\vf(f(p_{i+1}))-\vf(f(p_{i}))}{p_{i+1}-p_i} \int_{I_i} h(x) dx\\
            =&-\sum_{a\in\cA_0}\alpha_a[\sign(f'(a))\vf(f(a))-\psi(a)]-\int_0^1\vf\circ f(x)\sign(f'(x)) \bbh(x) dx\\
            &+\int_0^1\psi(x) \bbh(x) dx+\sum_{i=1}^N  \epsilon_i\frac{\vf(f(p_{i+1}))-\vf(f(p_{i}))}{p_{i+1}-p_i} \int_{I_i} h(x) dx .
        \end{split}
    \]
    Hence
    \[
        \begin{split}
            D\cL_0h=&\cL_1 \bbh(x) dx+\sum_{a\in\cA_0}\alpha_a\sign(f'(a))\delta_{f(a)}\\
            &-\sum_{i=1}^N\sum_{a\in\cA_0}\Id_{I_i}(a)\epsilon_i \left\{\frac{p_{i+1}-a}{p_{i+1}-p_i}\delta_{f(p_i)}+\frac{a-p_{i}}{p_{i+1}-p_i}\delta_{f(p_{i+1})}\right\}\\
            &-\sum_{i=1}^N \epsilon_i\left\{\delta_{f(p_i)}\int_{I_i}\frac{p_{i+1}-x}{p_{i+1}-p_i}\bbh(x)dx+\delta_{f(p_{i+1})}\int_{I_i}\frac{x-p_{i}}{p_{i+1}-p_i}\bbh(x)dx\right\}\\
            &-\sum_{i=1}^N  \epsilon_i\frac{\delta_{f(p_{i+1})}-\delta_{f(p_i)}}{p_{i+1}-p_i} \int_{I_i} h(x) dx\\
        \end{split}
    \]
    \[
        \begin{split}
            \phantom{D\cL_0h}
            &=\cL_1 \bbh(x) dx-\sum_{i=1}^N \epsilon_i \delta_{f(p_{i+1})}\int_{I_i}\bbh(x)dx+\sum_{a\in\cA_0}\alpha_a\sign(f'(a))\delta_{f(a)}\\
            &+\sum_{i=1}^N\sum_{a\in\cA_0}\Id_{I_i}(a)\epsilon_i \left\{\frac{p_{i+1}-a+\alpha_a}{p_{i+1}-p_i}\delta_{f(p_i)}+\frac{a-p_{i}-\alpha_a}{p_{i+1}-p_i}\delta_{f(p_{i+1})}\right\} \in \textrm{SBV}.
        \end{split}
    \]
    Thus $\cL_0$ is well defined on SBV.

    The bound on the essential spectral radius follows from \eqref{eq:LY} and Hennion's Theorem \cite{Hennion93} (see also \cite[Appendix B]{DKL}).

    Next, assume that $\nu\in\sigma_{\SBV}(\cL_0)$ and $|\nu|>1$. Then $\nu$ must be point spectrum and there exists $h\in\SBV$ such that $\cL_0 h=\nu h$. But, differentiating, this would imply that the absolutely continuous part of $Dh$, call it $\bbh$, satisfies $\cL_1\bbh=\nu\bbh$ and, since the spectral radius of $\cL_1$ is one, this implies $\bbh=0$.

    Finally, assume that $f$ admits a Markov partition. Then we can assume, without loss of generality, that $\{I_i=(p_i,p_{i+1})\}_{i=1}^N$ is the Markov partition hence, $f(p_i)=p_j$ for some $j$. Let $\cA$ be the set of jumps of $h$, $\{\alpha_a\}_{a\in\cA}$ the size of the jumps and set $\cA_*=\cA\setminus \{p_i\}_{i}^N$. By the Markov property if $a\in\cA_*$ then $f^{-1}(a)\not \in \{p_i\}_{i}^N$. Hence, the above formula implies that if $\beta_n(a)=\{b\in\cA_*\;:\; f^n(b)=a\}$ then, for each $n\in\bN$,
    \[
        |\nu|^n |\alpha_a|\leq \sum_{b\in\beta_n(a)}|\alpha_b|\leq \|h\|_{\SBV}.
    \]
    which is possible only if $\alpha_p=0$. Hence $h$ can jump only at the boundaries of the partition and it is constant inside.
\end{proof}
The above theorem shows that, for Markov maps, the study of the eigenvalues larger than one is reduced, as in section \ref{sec:piecewise}, to the study of the finite dimensional matrix $B_0$ defined in \eqref{eq:defB}. However, in the general case identifying all the measures of maximal entropy requires some work that, in this generality, exceeds our scopes.
To give an idea of what can be done let us restrict ourselves to the simplest Markov example: full branch maps.

\subsection{Full branch monotone maps}\ \\
Let $\bM=\{f\in\bM_*\;:\; f(I_i)=(0,1)\}$. In this case, the previous computations show that $\cL_0 W^{1,1}\subset W^{1,1}$ and $\frac d{dx}\cL_0=\cL_1\frac {d}{dx}$. From now on we will thus work in $W^{1,1}$.
For this maps we want to investigate the measure of maximal entropy.
To start, note that, in this case,  $\Lambda\geq N$ since
\[
    N =\sum_{i=1}^N |f(I_i)|= \sum_{i=1}^N \int_{I_i} |f'(x)| \ dx \leq \Lambda \sum_{i=1}^N |I_i|=\Lambda,
\]
where the inequality is strict if $f'$ is not constant.

Next, let us recall some well known facts (see \cite{Buzzi10} for a review).
\begin{lem}[{\cite[Theorem 1]{MS80}}]\label{lem:mius}
    For $f\in\bM$ holds the variational principle
    \[
        \htop=   \ln N = \sup_{\mu\in\cM}h_\mu(f)
    \]
    where $\cM$ is the set of invariant measures of $f$ and $h_\mu(f)$ is the Kolmogorov-Sinai entropy.
\end{lem}

\begin{thm}\label{thm:gap}
    The operator $\cL_0$ when acting on $W^{1,1}$ has the spectral decomposition $\cL_0 h= N \cdot \mme(h)+Q(h)$ where $Q1=0$, $\mme(Q(h))=0$, for all $h\in W^{1,1}$, $\sigma_{W^{1,1}}(Q)\subset \{z\in\bC\;:\;|z|\leq 1\}$, and $\mme$ is a measure of maximal entropy.
\end{thm}
\begin{proof}
    We know that if $\nu\in \sigma(\cL_0)$ and $|\nu|>1$, then $\nu$ is point spectrum. That is, there exist $h\in W^{1,1}$ such that $\cL_0h=\nu h$.
    But then, differentiating, we have $\cL_1 h'=\nu h'$ where $h'\in L^1$.
    However, $\cL_1$ is a contraction on $L^1$, hence it must be either $|\nu|\leq 1$, contrary to the hypothesis, or $h'=0$. The latter implies that $h$ is constant, hence, we can always normalise it so that $h=1$.
    On the other hand $\cL_01(x)=\sum_{y\in f^{-1}(x)} 1 =N$. Hence $\nu=N$ and has geometric multiplicity one.
    Indeed, if the geometric multiplicity is not one, then there must exists $h\in W^{1,1}$ such that $\cL_0 h=N h+c$ for some constant $c$.
    But then, differentiating, $\cL_1 h'=N h'$, so $h$ must be constant again.

    In addition, note that for each $h\in W^{1,1}$ we have
    \[
        \begin{split}
            \|N^{-n}\cL_0^n h\|_{W^{1,1}}&\leq \int_0^1N^{-n}\cL_0^n |h|+\int_0^1 N^{-n}\cL_1^n|h'|\leq \|h\|_{\cC^0}+N^{-n}\|h\|_{W^{1,1}}\\
            &\leq (1+N^{-n})\|h\|_{W^{1,1}}.
        \end{split}
    \]
    Hence $N^{-n}\cL_0^n$ is uniformly bounded on $W^{1,1}$ and thus, by \cite[Lemma VIII.8.1]{DS58}, $N$ is semi-simple. It follows that the maximal eigenvalue is simple.

    Accordingly, we have the spectral decomposition $\cL_0 = N 1\otimes \mu+Q$ where $Q$ has spectral radius smaller or equal one, $Q1=0$, $\mu(Q h)=0$ for all $h\in W^{1,1}$, $\mu(1)=1$, and $\mu$ belongs to the dual of $W^{1,1}$. It remains to prove that $\mu$ is a measure and, indeed, a measure of maximal entropy $\mme$.

    Note that, for each $h\in W^{1,1}$,
    \[
        |\mu(h)|=\lim_{n\to\infty}\left|\int_0^1 N^{-n}\cL_0^nh\right|\leq \lim_{n\to\infty}{\|h\|}_\infty \int_0^1 N^{-n}\cL_0^n1={\|h\|}_\infty.
    \]
    Thus $\mu$ is a measure. In addition, for each $h\in \cC^1$ such that $h\geq 0$, we have
    \[
        \mu(h)=\lim_{n\to\infty}\int_0^1 N^{-n}\cL_0^{n}h\geq 0
    \]
    thus $\mu$ is a positive measure and, since it is normalized, it is a probability measure.
    Next, note that
    \[
        \mu(\cL_0 h)=\lim_{n\to\infty}\int_0^1 N^{-n}\cL_0^{n+1}h= N \lim_{n\to\infty}\int_0^1 N^{-n}\cL_0^{n}h= N\cdot \mu(h).
    \]
    It follows
    \[
        \mu(h\circ f)= N^{-1}\mu(\cL_0 h\circ f)= N^{-1}\mu( h \cL_01)=\mu(h).
    \]
    That is $\mu$ is an invariant measure. In addition, by the above considerations, $([0,1], f,\mu)$ is ergodic.

    The proof is concluded if we show $h_\mu(f)\geq \htop$.
    Let $\cP_n$ denote the $n^{\textrm{th}}$ canonical dynamical refinement of the partition $\cP$.
    Let $p\in\cP_n$  and $p_-,p_+\in\cP_{n}$ be the elements on the left and the right of $p$, respectively, if they exist. Let $J=p_-\cup p\cup p_+$. Let $h\in\cC^1(\bR,[0,1])$ be supported in $J$ and such that $h|_{p}=1$. Then $\cL_0^{n} h(x)\leq 3$ and
    \[
        \begin{split}
            \mu(p)\leq \mu(h)&=\lim_{m\to\infty} \int_0^1 N^{-m-n} \cL_0^{m+n} h\leq \lim_{m\to\infty} 3\int_0^1 N^{-m-n} \cL_0^{m} 1\\
            &= 3N^{-n}\mu(1)= 3N^{-n}.
        \end{split}
    \]
    Accordingly, calling $p_n(x)$ the element of $\cP_n$ which contains $x$, the Shannon-McMillan-Breiman Theorem (e.g., see \cite[Section 6.2, Theorem 2.3]{Petersen89}) states that for $\mu$ almost every point
    \[
        h_\mu(f)\geq     h_\mu(\cP, f)=\lim_{n\to\infty}-\frac 1n\ln\mu(p_n(x))\geq \lim_{n\to\infty} \ln N^{1-1/n}=\ln N,
    \]
    which concludes the proof by Lemma \ref{lem:mius}.
\end{proof}

\begin{rem}
    We do not know if $\mme$ is unique in this case, we have just constructed a measure of maximal entropy. Look at subsection \ref{sec:non_uniform} for a case where it is easy to prove that the measure of maximal entropy is unique.
\end{rem}

\begin{rem}
    The monotone interval maps and transfer operators studied in this section fit into the framework considered in \cite{BK90}. In \cite{BK90} the essential spectral radius (as operators acting on \(BV\)) is obtained and consequently  a spectral decomposition. Here we show that the spectral gap is large for the operator associated to the measure of maximal entropy.
\end{rem}

We have finally the announced mixing rate estimate
\begin{cor}
    For any $\nu>\frac 1{N}$ there exists $C_\nu>0$ such that, for each $h\in W^{1,1}$ and $\vf\in L^1(\mme)$
    \[
        \left| \int h \ \vf\circ f^n \ d\mme- \int h \ d\mme\int\vf\circ f^n \ d\mme\right|\leq C_\nu\nu^n{\|h\|}_{W^{1,1}} {\|\vf\|}_{L^1(\mme)}.
    \]
\end{cor}
\begin{proof}
    We start assuming that $\vf\in \cC^1$. Then, using Theorem \ref{thm:gap},
    \begin{multline*}
        \left| \int h \ \vf\circ f^n \ d\mme- \int h \ d\mme\int\vf \ d\mme\right|\\
        = \left|\lim_{m\to\infty}\int_0^1 N^{-m}(\cL_0^{m} h \ \vf\circ f^n)(x) \ dx  -\int h d\mme\int\vf \ d\mme\right|\\
        =\left|\lim_{m\to\infty} \int_0^1[ N^{-m+n}\cL_0^{m-n} \vf N^{-n}\cL_0^{n} h](x) \ dx-\int h d\mme\int\vf \ d\mme\right|\\
        =\left| \int \vf N^{-n}\cL_0^n h \ d\mme- \int h \ d\mme\int\vf \ d\mme\right|\\
        =\left| \int \vf N^{-n}Q^n h \ d\mme\right|
        \leq C_\nu\nu^n{\|h\|}_{W^{1,1}}\int|\vf| \ d\mme.
    \end{multline*}
    The corollary follows by a simple approximation argument.
\end{proof}

%%%%%%%%%%%%%%%%%%%%%%%%%%%%%%%%
\subsection{Non-uniformly expanding maps}\ \\ \label{sec:non_uniform}
Let $\cE\subset \bM$ the set of maps such that $f'\geq 1$, $f'=1$ at finitely many points and $\Lambda={\|f'\|}_\infty<\infty$.

This class of maps includes the well known Manneville--Pomeau map \cite{PM80, LSV99}.
\begin{rem}\label{rem:hilbert2}
    In \cite{CV13} non-uniformly expanding systems are studied and the existence of a spectral gap (and hence decay of correlations) is proven for a class of equilibrium states which includes the measure of maximal entropy.
    The approach in \cite{CV13} is based on Hilbert metrics and, although not stated explicitly, it  provides a poor estimate of the spectral gap (see Remark~\ref{rem:hilbert} for similar considerations) whereas our present approach  provides an explicit and close to optimal estimate.

    Here, we limit ourselves to the one dimensional case to present the idea in its simpler form. It is likely that similar results can be obtained for more general non-uniformly expanding maps, e.g., the higher dimensional examples treated in \cite{CV13}.
\end{rem}

Note that the maps in $\cE$ have a basic property.
\begin{lem}
    Any map $f\in \cE$ is expansive.
\end{lem}
\begin{proof}
    Let $\kappa=\min_{I\in \cP}|I|$. For each $\delta>0$ let $\cI_\delta=\{[a,b]\subset [0,1]\;:\; [a,b]\subset \bar{I}, I\in \cP \;;|b-a|\geq \delta\}$  and, for each $[a,b]\in \cI_\delta$, define $\vf(a,b):=\frac 1{|b-a|}\int_a^b f'(\xi)d\xi$.
    Note that, by hypothesis, $\vf(a,b)>1$, and since it depends continuously from $a,b$ (which vary in a compact set) there must be $\tau_\delta>1$ such
    $\vf(a,b)\geq\tau_\delta$.
    Accordingly,  $f^n(x)$ and $f^n(y)$ always belong to the same partition element we have $|f^n(x)-f^n(y)|\leq \kappa$  for all $n\in\bN$ which is possible only for $x=y$.
    On the other hand, if for some $f^n(x)$ and $f^n(y)$ belong to two different partition element, then either $|f^n(x)-f^n(y)|\geq \kappa$ or they belong to contiguous elements of $\cP$.
    In such a case it is easy to see that there exists $\delta$ such that either $|f^n(x)-f^n(y)|\geq \delta$ or $|f^{n+1}(x)-f^{n+1}(y)|\geq \delta$, hence the expansivity.
\end{proof}
The above fact allows to prove that the measure of maximal entropy is unique.
\begin{lem}
    For $f\in\cE$   the measure of maximal entropy $\mme$ is unique.
\end{lem}
\begin{proof}
    Since the map is expansive, there exists a map $\Phi:[0,1]\to \{1,\dots, d\}^\bN=:\Sigma$ which is well defined and invertible, apart from countably many points, that conjugates $f$ with the full shift $\sigma$.
    Hence, $\Phi$  induces a measurable isomorphism for each non-atomic measure.
    On the other hand for $(\Sigma,\sigma)$ holds the variational principle, hence the sup of the metric entropies is the topological entropy, which is $\ln N$, and there exists a unique measure of maximal entropy.
    Since atomic measures have zero entropy, and since the entropy is an affine function of the measures, it follows that the sup on the measure entropies is achieved on non-atomic measures.
    Thus, via the isomorphism $\Phi$ and since the entropy is an invariant for measure-preserving conjugacy, it follows that measure of maximal entropy for $f$ is unique.
\end{proof}

%%%%%%%%%%%%%%%%%%%%%%%%%%%%%%%%%%%%
\section{Hyperbolic maps}
\label{sec:hyp}
For hyperbolic, or partially hyperbolic maps the situation is less clear than in the expanding case and much more remains to be understood. Yet, the ghost of a general theory seems to be present. Let us start with the simplest possible case: linear maps.

In this case it is possible to study the problem using Fourier series (see \cite{Lind82}), however it is interesting to develop an alternative approach that  does not rely on the algebraic structure of the map and thus has the potential to be applicable in greater generality.

%%%%%%%%%%%%%%%%%%%%%%%%%%%%%%%%%
\subsection{Automorphisms of the torus}\ \\
Here we consider a linear map $f:\bT^n\to\bT^n$ defined by $f(x)=Ax\mod 1$ where $A\in SL(d,\bZ)$, i.e., a matrix with integer coefficient and $\det A=1$.
Let us call $E^u$ the unstable subspace, $E^s$ the stable one and $E^c$ the central one.

Note that, by hypothesis $f$ preserves the volume, thus the volume is the SRB measure. We are interested in its statistical properties, hence in the transfer operator
\[
    \cL h=h\circ f^{-1}.
\]

Next we introduce a norm. Let $\{v^s_i\}$, $\|v^s_i\|=1$, be a basis of $E^s$ and $\{v^u_i\}$, $\|v^u_i\|=1$, be a basis of $E^u$ and define, for each $h,\vf\in\cC^\infty$ and $p,q\in\bN_0$, $\partial^{s/u}_ih=\langle v^{s/u}_i,\nabla h\rangle$, and
\begin{equation}\label{eq:linear_norms}
    \begin{split}
        &|\vf|^s_q=\sup_{0\leq k\leq q}\sup_{i_1,\dots,i_k}\|\partial^s_{i_1}\cdots \partial^s_{i_k} \vf\|_\infty\\
        &\|h\|_{p,q}=\sum_{0\leq k\leq p}\sup_{i_1,\dots,i_k}\sup_{|\vf|^s_{k+q}\leq 1}\int_{\bT^n} \vf\partial^u_{i_1}\cdots \partial^u_{i_k}h.
    \end{split}
\end{equation}
We call $\cB^{p,q}$ the completion of $\cC^\infty$ with respect to the norms $\|\cdot\|_{p,q}$.

In the following we assume $E^u\neq \{0\}$. In addition, to simplify the exposition, we assume that $A$ has no Jordan blocks. The general case can be treated with a slight sophistication of the following arguments. We can thus choose the $v_i^u$ such that $Av_i^u=\lambda_iv^u_i$, with $\lambda_i\geq \lambda>1$. Also let $\lambda$ be such that $\|A|_{E^s}\|\leq \lambda^{-1}$.
\begin{rem}
    The above norms are inspired by \cite{BL20}. They are one of the many possible constructions of anisotropic Banach spaces adapted to hyperbolic maps or flows, see \cite{Baladi18} for an extensive discussion. Given the linear structure of the invariant foliations, the norms \eqref{eq:linear_norms} turn out to be especially convenient and simple to deal with, hence allowing a completely self-contained discussion. In the next section, we will use instead the norms defined in \cite{GLP13} in order to avoid having to redevelop the all theory (e.g. the Lasota-Yorke inequality) in the style of \cite{BL20}, which would certainly be possible.
\end{rem}
The following is the equivalent of \cite[Proposition 3.2]{BL20}.
\begin{prop}\label{prop:LY_auto}
    For each $p,q\in\bN_0, n\in\bN$, and $\nu\in (\lambda^{-\min\{p+1,q\}},1)$, there exist $A, B>0$ such that,
    \[
        \begin{split}
            &\|\cL h\|_{p,q}\leq \|h\|_{p,q}\\
            &\|\cL^n h\|_{p+1,q}\leq A \nu^{n}\|h\|_{p+1,q}+B\|h\|_{p,q+1}.
        \end{split}
    \]
\end{prop}
\begin{proof}
    Since $\langle v,\nabla (h\circ f^{-1})\rangle=\langle Df^{-1}v, \nabla h\rangle \circ f^{-1}$. We have
    \[
        \int_{\bT^n} \vf\partial^u_{i_1}\dots \partial^u_{i_k}\cL h=\prod_{j=1}^k\lambda_{i_j}^{-1}\int_{\bT^n} \vf\circ f \partial^u_{i_1}\cdots \partial^u_{i_k}h.
    \]
    Since $|\vf\circ f|^s_{k+q}\leq |\vf|^s_{k+q}$, the first inequality follows.

    Next, note that
    \[
        \|h\|_{p,q}=\sup_{i_1,\dots,i_p}\sup_{|\vf|^s_{p+q}\leq 1}\int_{\bT^n} \vf\partial^u_{i_1}\cdots \partial^u_{i_p}h+\|h\|_{p-1,q}.
    \]
    Thus, by the above computations,
    \[
        \|\cL^n h\|_{p,q}\leq \lambda^{-pn}\|h\|_{p,q}+\|\cL h\|_{p-1,q}.
    \]
    It thus suffices to consider the case $k<p$. If $|\vf|^s_{k+q}\leq 1$, then, for each $\ve>0$, let $\vf_\ve$ be such that $|\vf-\vf_\ve|^s_{k+q-1}\leq \ve$, $|\vf-\vf_\ve|^s_{k+q}\leq 2$ and $|\vf_\ve|^s_{k+q+1}\leq C\ve^{-1}$, for some fixed constant $C>2$.\footnote{ Such a function can be constructed by convolving with a mollifier in the space $E^s$.}
    \[
        \begin{split}
            \left|\int_{\bT^n} \vf\partial^u_{i_1}\cdots \partial^u_{i_k}\cL^n h\right|&\leq \prod_{j=1}^k\lambda_{i_j}^{-n}\left\{\left|\int_{\bT^n} (\vf-\vf_\ve)\circ f^n \partial^u_{i_1}\cdots \partial^u_{i_k}h\right|+C\ve^{-1}\|h\|_{k,q+1}\right\}\\
            &\leq  \prod_{j=1}^k\lambda_{i_j}^{-n}\left\{\max\{\ve, 2\lambda^{-(k+q)n}\}\|h\|_{k,q}+C\ve^{-1}\|h\|_{k,q+1}\right\}\\
            &\leq  2\lambda^{-(2k+q)n}\|h\|_{k,q}+\frac{C}2\lambda^{(k+q)n}\|h\|_{k,q+1}
        \end{split}
    \]
    where, in the last line, we have chosen $\ve=2\lambda^{-(k+q)n}$. Accordingly,
    \[
        \|\cL^n h\|_{p,q}\leq (\lambda^{-pn}+2\lambda^{-qn})\|h\|_{p,q}+C\lambda^{(q+p)n}\|h\|_{k,q+1}.
    \]
    Next, choose $n_0\in\bN$ such that $3\lambda^{-\min\{p,q\}n_0}\leq \nu^{n_0}$, write $n=kn_0+m$ with $m<n_0$ and iterate the above equation to obtain
    \[
        \|\cL^n h\|_{p,q}\leq \frac{3\lambda^{-\min\{p,q\}n_0}}{\nu^{n_0}}\nu^n\|h\|_{p,q}+\frac{C}{1-\nu}\lambda^{(q+p)n_0}\|h\|_{k,q+1},
    \]
    which proves the Proposition.
\end{proof}
\begin{rem} Note that Proposition \ref{prop:LY_auto} implies that the spectral radius of $\cL$ when acting on any space $\cB^{p,q}$ is bounded by one. On the other hand, since $\cL 1=1$, the spectral radius must be exactly one.
\end{rem}
The following is the equivalent of \cite[Lemma 4.1]{BL20}, although the proof follows a different path, easier in this particular case.
\begin{lem}\label{lem:compact}
    If $E^c=\{0\}$, then, for each $p,q\in\bN$, $\{h\in\cB^{p,q}\;:\; \|h\|_{p,q}\leq 1\}$ is relatively compact in $\cB^{p-1,q+1}$.
\end{lem}
\begin{proof}
    Let $d_s=\dim(E^s)$ and $d_u=\dim(E^u)$. By hypothesis $d=d_s+d_u$.

    Let $U:\bR^{d_s}\to\bR^{d_u}$ such that $\{{\boldsymbol v}^s(v)=(v,Uv)\}_{v\in\bR^{d_s}}=E^s$ and $V:\bR^{d_u}\to\bR^{d_s}$ such that $\{{\boldsymbol v}^u(v)=(Vv,v)\}_{v\in\bR^{d_u}}=E^u$.\footnote{ We can always choose coordinates in which this is possible.}

    Finally, consider mollifiers $j_\ve^{s/u}(x)=\ve^{-d_{s/u}}j^{s/u}(\ve^{-1}x)$, where  $j^{s/u}\in \cC^\infty(\bR^{d_{s/u}},\bR_+)$ such that $\supp j^{s/u}\subset \{\|x\|\leq 1\}$ and $\int_{\bR^{d_{s/u}}} j^{s/u}(x) dx=1$.
    Then, for each $|\vf|_{p+q}\leq 1$ and $h\in \cB^{p,q}$,
    \[
        \begin{split}
            &\int_{\bT^d}dx\, \vf \partial^{u}_{i_1}\cdots\partial^{u}_{i_{p-1}} h=\int_{\bT^d}dx \int_{\bR^{d_s}}dv\,\vf(x+{\boldsymbol v}^s(v))j_\ve^s(v)\partial^{u}_{i_1}\cdots\partial^{u}_{i_{p-1}} h(x)+\cO(\ve\|h\|_{p,q})\\
            &=\int_{\bT^d}dx \int_{\bR^{d_s}}dv\int_{\bR^{d_u}}dw\vf(x+{\boldsymbol v}^s(v)-{\boldsymbol v}^u(w))j_\ve^s(v)j_\ve^u(w)\partial^{u}_{i_1}\cdots\partial^{u}_{i_{p-1}} h(x)+\cO(\ve\|h\|_{p,q})\\
            &=:\int_{\bT^d}dx\; {{\boldsymbol \vf}_\ve}(x) \partial^{u}_{i_1}\cdots\partial^{u}_{i_{p-1}} h(x)+\cO(\ve\|h\|_{p,q}).
        \end{split}
    \]
    Note that $\|{{\boldsymbol \vf}_\ve}\|_{\cC^{p+q+1}}\leq C\ve^{-p-1}$ and hence for each $\ve$ there is a set $\{\phi_i\}_{i=1}^{N_\ve}\subset \cC^{p+q}$ such that, for all $\vf$ we have $\|{{\boldsymbol \vf}_\ve}-\phi_i\|_{\cC^{p+q}}\leq \ve$ for some $\phi_i$.  It follows that, for each $\ve$,
    \[
        \|h\|_{p-1,q+1}\leq \Const\ve\|h\|_{p,q}+\sup_{i\leq N_\ve} \left|\int_{\bT^d}\phi_i h\right|.
    \]
    From the above the wanted compactness follows by a standard diagonalization argument.
\end{proof}

We can now define the operators $\cD_i h=\partial^u_i h$. Then
\begin{equation}\label{eq:yet_another_commutation}
    \cD_i\cL h=\langle Df^{-1} v^u_i, \nabla h\circ f^{-1}\rangle=\lambda_i^{-1} \langle v^u_i, h\rangle\circ f^{-1}=\lambda_i^{-1} \cL\cD_i h.
\end{equation}

The usefulness of these operators rests in the following Lemma. This is the only place in which we use Fourier series, however the result follows essentially from the accessibility property although with a more cumbersome proof.
\begin{lem} \label{lem:homology}
    The $\cD_i$ are bounded operators from $\cB^{p-1,q+1}$ to $\cB^{p,q}$. In addition, if we assume that $A$ has no eigenvalues that are roths of unity, then if $h\in \cB^{p,q}$, $p\geq 1$ and $\cD_i h=0$, for all $i$, then  $h$ is constant.
\end{lem}
\begin{proof}
    The fact that the $\cD_j$ are bounded operators from $\cB^{p-1,q+1}$ to $\cB^{p,q}$ is a direct consequence of the definition of the norms in \eqref{eq:linear_norms} and integration by parts.

    Next,  Katznelson Lemma \cite[Lemma 3]{Katz71} (applied to $A^*$) implies that there exists $C_0>0$ such that, for each $k\in\bZ^n\setminus\{0\}$,
    \begin{equation}\label{eq:katz}
        \begin{split}
            &\dist (k,(E^u\oplus E^c)^\perp)\geq C_0 \|k\|^{-n},%\\
        \end{split}
    \end{equation}
    Let $\hat h_k$ be the Fourier coefficient of $h$. Suppose tha $\hat h_k\neq 0$ for some $k\neq 0$. Since, by hypothesis $0=\widehat{\cD_j h}_k= i\langle v^u_j,k\rangle \hat h_k$, we have $\langle v^s_j,k\rangle=0$ for all $j$. Thus $k\perp E^u$, contradicting \eqref{eq:katz}. Thus $h$ must be constant.
\end{proof}

We are now ready to draw our conclusions.
\begin{lem}\label{eq:lin_spectra}
    If $E^c=\{0\}$, then for each $\ve>0$ and $p,q$ large enough we have $\sigma_{\cB^{p,q}}(\cL)\subset\{1\}\cup\{z\in\bC\;:\; |z|\leq \ve\}$.
\end{lem}
\begin{proof}
    By Proposition \ref{prop:LY_auto} and Lemma \ref{lem:compact}, together with Hennion argument \cite{Hennion93} (or \cite[Appendix B]{DKL}) we have that the spectrum in the considered region is only point spectrum provided $\lambda^{-\min\{p,q\}}<\ve$. We thus require $p,q$ to be such that $\lambda^{-\min\{p,q\}}<\ve$.

    Next, suppose that $\cL h=\nu h$ with $|\nu|>\ve$, then, for all $j$, \eqref{eq:yet_another_commutation} implies
    \[
        \cL(\cD_j^qh)=\lambda_j^{q}\nu (\cD_j^q h).
    \]
    But since $|\lambda_j^{q}\nu|>1$ it cannot be an eigenvalue of $\cL$, thus it must be $\cD_j^qh=0$. But, since integrating by parts yields
    \[
        \int_{\bT^d} \cD_l^qh=0,
    \]
    for all $0<l\leq j$, Lemma \ref{lem:homology} implies that $h$ must be a constant, which, in turn, implies $\nu=1$.
\end{proof}
\begin{rem} Lemma \ref{eq:lin_spectra}, by the usual arguments, implies that $\cC^\infty$ observables have a super-exponential decay of correlations.
\end{rem}
\begin{rem}
    It seems reasonable to expect that a similar result should hold also if $E^c$ is not trivial, but the map is ergodic ($A$ does not have eigenvalues that are roots of unity).
    However, the proof of Lemma \ref{lem:compact} fails in this case. To overcome this problem it may be  necessary to use a different Banach space. Thus, at present, it is not clear how to apply this strategy to partially hyperbolic systems, even in the simplest case.
\end{rem}
The nonlinear case is much more subtle even in the Anosov setting. The obvious idea would be to consider an unstable vector field $w$ and the operator $\cD=\langle w, \nabla h\rangle$. Unfortunately, in general, unstable vector fields are only H\"older. Hence, it is not clear if $\cD$ is a well defined bounded operator from $\cB^{p,q}$ to $\cB^{p-1,q+1}$. To solve this problem one should probably use different banach spaces (for some appropriate version of such spaces such as the ones introduced in \cite{GL08}, see \cite{Ts18} for some recent progress along these lines).

Indeed, on the one hand $w$ is smooth along unstable manifolds, on the other hand in the stable direction is only H\"older so its derivatives must be regarded as distribution, like $h$, and multiplication of distributions is a rather touchy business. So the situation, although not hopeless, is rather unclear.

Such issue needs further thought. Here we limit ourselves to explore an interesting alternative: considering the external derivative $d$ as the appropriate differential operator.
This simple change of perspective yields interesting results since it seems to provide a connection with the topology of the manifold.
At least, this is the situation in the following where we discuss only the simplest case: two dimensional Anosov maps.
%%%%%%%%%%%%%%%%%%%%%%%%%%%%%%%%%%%
\subsection{ Anosov map on two dimensional manifolds}\ \\
Let $M$ be a smooth two dimensional compact and connected Riemannian manifold and $f\in \operatorname{Diff}^\infty(M, M)$,  be a transitive Anosov  map.
In other words, there exists $\lambda >1$ and two continuous strictly invariant  cone fields $\cC^s,\cC^u$ such that, for all $x\in M$,
\[
    \begin{split}
        &\|d_xf v\|\geq \lambda \|v\| \quad  \forall v\in \cC^u(x)\\
        &\|d_xf^{-1} v\|\geq \lambda \|v\| \quad  \forall v\in \cC^s(x).
    \end{split}
\]
\begin{rem}\label{rem:comment}
    According to the Franks-Newhouse Theorem \cite{Franks70, Newhouse70}, every Anosov diffeomorphism of a two-dimensional compact Riemannian manifold is topologically conjugate to a hyperbolic toral automorphism. Hence, in our case,  $M$ must be homeomorphic to $\bT^2$. Note however that in the following the smoothness of the map plays a fundamental role, hence one cannot in general reduce the discussion to the case
    $\bT^2=\bR^2\backslash\bZ^2$. It is thus convenient to argue considering $M$ a general two dimensional manifold. This has also the advantage to emphasise the possibility of a higher dimensional extension. Indeed, we will use the  Franks-Newhouse Theorem only at the end of the argument (Lemma \ref{lem:thisistheend}), to characterise the cohomology groups.
\end{rem}
In analogy with the previous sections, we will obtain results on the mixing properties of the measure of maximal entropy $\mme$.
\begin{thm}\label{thm:max_hyp}
    The exists $r\in\bN$, $C>0$ and $\kappa\in (0,1)$ such that for all $g,h\in \cC^\infty$ and $n\in\bN$ we have
    \[
        \left|\int_M g\circ f^n h d\mme-\int_M gd\mme\int_M h d\mme\right|\leq C{\|g\|}_{\cC^r}{\|h\|}_{\cC^r} e^{-\htop n}\kappa^n.
    \]
\end{thm}
This result is a corollary of the much more precise Theorem \ref{thm:maximal_entropy_gap} and it is proven in section \ref{sec:max_proof}.
To state Theorem \ref{thm:maximal_entropy_gap} we need to first introduce several objects.
\subsection{The operators}\ \\
The operator associated to the SRB measure is simply (e.g., see \cite{GL06})
\[
    \cL h(x)=(\det D_{f^{-1}(x)}f)^{-1}h\circ f^{-1}(x).
\]
However, in the present context the interesting object to study seems to be the action of forms, or rather {\em currents}.\footnote{ The idea that currents are a relevant object to study in the context of the statistical properties of dynamical systems goes back, at least, to \cite{RS75}. See, for example, \cite{Ruelle90} and \cite{BB05} for further use of \(k\)-forms in the dynamical systems context.}
Recall that the pullback on a differential form $\omega$ by  a map $g$ is defined as
\[
    {(g^*\omega)}_x( v_1, v_2)=\omega_{g(x)}(d_xg(v_1), d_xg(v_2)).
\]
If $g$ is a diffeomorphism we can define the pushforward as $g_*\omega=(g^{-1})^*\omega$. It is then natural to define the action of the dynamics on forms as the pushforward $f_*$.

Let $\omega_0$ be the Riemannian volume. Then any two form can be written as $\omega= h\omega_0$ for some function $h$. Then
\begin{equation}
    \label{eq:duality-zero-due}
    \begin{split}
        [f_*\omega(v_1,v_2)](x)&=h\circ f^{-1}(x)\omega_0((d_xf^{-1}(v_1), d_xf^{-1}(v_2))(x)\\
        &=h\circ f^{-1}(x)\det(D_{f^{-1}(x)}f)^{-1}\omega_0(v_1,v_2)(x)\\
        &=\left[\left(\cL h\cdot \omega_0\right)(v_1,v_2)\right](x).
    \end{split}
\end{equation}
That is, the operator $\cL$ is equivalent  to the pushforward on two forms.

Recall that
\begin{equation}
    \label{eq:1-0}
    d (f_*  h)=f_*dh.
\end{equation}
where, if $h$ is a zero form, then $f_*dh(x)=\left[D_xf^{-1}\right]^T (\nabla h)\circ f^{-1}(x)$.

The scalar product in $T^*M$ is canonically defined by using the canonical duality $\pi: T^*M\to T_*M$ defined by $\omega(v)=\langle \pi(\omega),v\rangle$, for all $v\in T_*M$. That is,
\begin{equation}\label{eq:scalare}
    \langle \omega_1,\omega_2\rangle=\langle \pi(\omega_1),\pi(\omega_2)\rangle=\omega_1(\pi(\omega_2)).
\end{equation}

For each $x\in M$ and $v_1, v_2, w_1, w_2\in T_x^*M$ we define
\begin{equation}
    \label{eq:scalarproduct}
    \langle v_1\wedge v_2,w_1\wedge w_\rangle=\det (\langle v_i,w_j\rangle).
\end{equation}
Assuming bilinearity, the above formula defines uniquely a scalar product among 2-forms. Also, we define a duality from $\ell$ to $2-\ell$ forms via (see \cite[Appendix A]{GLP13} for more details)
\begin{equation}
    \label{eq:hdgedef}
    \langle v,w\rangle \omega_0=(-1)^{\ell(2-\ell)}v\wedge *w=(-1)^{\ell(2-\ell)}w\wedge *v=*v\wedge w.
\end{equation}
Since such a formula must hold for all $\ell$-forms, the $(2-\ell)$-forms $*w, *v$ are uniquely defined. The operator ``$*$'' is the so called {\em Hodge operator}.

\subsection{The Banach spaces and the main result}\label{sec:banch}\ \\
The operators $f_*$ have been studied for flows in \cite{GLP13} using appropriate Banach spaces. We use the same notation and almost the same Banach spaces defined in \cite[Section 3]{GLP13}. However, since here we consider maps rather than flows, we do not have the requirement that the forms be null in the flow direction (see \cite[ equation (3.5)]{GLP13}).  Here we provide some more detail but we refer to \cite{GLP13} for the full story.

For any $r\in \mathbb{N}$, we assume that there exists $\delta_0>0$ such that, for each $\delta\in(0,\delta_0)$ and $\rho\in(0,4)$, there exists an atlas $\{(U_\alpha, \Theta_\alpha)\}_{\alpha\in\mathit{A}}$, where $\mathit{A}$ is a finite set, such that\footnote{ We use the notation
    $B_d(x,r)=\{y\in\bR^d\;:\; \|y-x\|<r\}$.}
\begin{equation}
    \begin{cases}
         & \Theta(U_\alpha)=B_2(0,30\delta\sqrt{1+\rho^2}),                                                                                                        \\
         & \cup_\alpha \Theta_\alpha^{-1}(B_2(0,2\delta))=M,                                                                                                       \\
         & \Vert (\Theta_\alpha)_*\Vert_\infty+\Vert (\Theta_\alpha^{-1})_*\Vert_\infty\leq 2; \quad \Vert \Theta_\alpha \circ \Theta_\beta^{-1}\Vert_{C^r}\leq 2.
    \end{cases}
\end{equation}
Fix $L_0>0$. For any $L > L_0$, let us define
\[
    \begin{split}
        \mathcal{F}_r(\rho,L):=\big\{F\in \cC^r(B_{1}(0,6\delta), \mathbb{R})\;:\;& F(0)=0;\\
        &\Vert DF\Vert_{\cC^0(B_{1}(0,6\delta))}\leq \rho; \Vert F\Vert_{\cC^r(B_{1}(0,6\delta))}\leq L\big\}.
    \end{split}
\]
Where the $\cC^r$ is defined as usual, e.g. see \cite[ equation (3.6)]{GLP13}.
For each $F\in \mathcal{F}_r(\rho, L)$, $x\in \mathbb{R}^2$, $\xi\in \mathbb{R}^{1}$, let $G_{x,F}(\xi):B_{1}(0,6\delta)\rightarrow \mathbb{R}^2$ be defined by $G_{x,F}(\xi):=x+(\xi,F(\xi))$. Let us also define $\tilde{\Sigma}(\rho,L):=\{G_{x,F}:x\in B_{1}(0,2\delta), F\in \mathcal{F}_r(\rho,L)\}$.
For each $\alpha\in \mathit{A}$ and $G\in \tilde{\Sigma}(\rho,L)$, we define the leaf
\[
    W_{\alpha, G}:=\{\Theta_\alpha^{-1}\circ G(\xi)\}_{\xi \in B_{1}(0,3\delta)}.
\]
For each $\alpha \in \mathit{A}$, $G\in \tilde{\Sigma}(\rho,L)$, note that $W_{\alpha, G}\subset \hat{U}_\alpha:=\Theta_\alpha^{-1}(B_d(0,6\delta\sqrt{1+\rho^2}))\subseteq U_\alpha$. Finally, we define $\Sigma_\alpha=\cup_{G\in \tilde{\Sigma}(\rho,L)} W_{\alpha,G}$. This is the set of  ``almost stable" leaves that we will use to define our norms.\\

Given a curve $W_{\alpha,G}\in\Sigma_\alpha$, we consider $\Gamma^{\ell, s}_c(\alpha, G)$ as the $\cC^s$ sections of the fiber bundle on $W_{\alpha,G}$ with fibers in $\wedge^\ell(T^*M)$, as defined in \cite[equation (3.8)]{GLP13}, equipped with the $\cC^s$ norm. \\

Following \cite[Section 3]{GLP13} let ${V}^s(\alpha, G)$ be the set of uniformly $\cC^s(U_{\alpha,G})$ vector fields, where $U_{\alpha,G}$ is any open set such that $U_\alpha\supset U_{\alpha,G}\supset W_{\alpha,G}$.

Let $\omega_{vol}$ be the volume form induced on $W_{\alpha, G}$ by the push-forward of Riemannian volume via the chart $\Theta_\alpha^{-1}$. Write $L_\nu$ for the Lie derivative along a vector field $\nu$. For all $\alpha\in \mathit{A}, G\in \Sigma_\alpha, g\in \Gamma^{l,0}_c(\alpha,G), \bar{v}^p=(v_1,\cdots, v_p)\in V^s(\alpha, G)^p$, let us define the functionals $J_{\alpha,G,g,\bar{\nu}^p}: C^p\to \bC$ by\footnote{ By $\langle\cdot,\cdot\rangle$ we mean the usual scalar product between forms.}
\begin{equation}\label{eq:norm-anosov}
    J_{\alpha,G,g,\bar{\nu}^p}(h)=\int_{W_{\alpha,G}}\langle g, L_{v_1}\cdots L_{v_p} h\rangle \omega_{vol}.
\end{equation}
Next, for all $p\in \mathbb{N}$, $q\in\mathbb{R}_+$, $p+q<r-1$, $l\in\{0,1,2\}$, let
\[ \mathbb{U}_{\rho,L,p,q,\ell}=\{J_{\alpha,G,g,\bar{\nu}^p}\vert \alpha \in A, G\in \Sigma_\alpha(\rho,L), g\in \Gamma^{\ell,p+q}_c,\nu_j\in V^{p+q},\] \[\Vert g\Vert_{\Gamma_c^{\ell,p+q}(\alpha,G)}\leq 1,\Vert \nu_j\Vert_{C^{p+q}(U_{\alpha,G})}\leq 1\}.\]
where, for $\nu \in V^s(\alpha,G)$, $\Vert \nu\Vert_{C^s(U_{\alpha,G})}=\sup_{\alpha,i}\Vert \langle \nu,e_{\alpha,i}\rangle \circ \Theta_\alpha^{-1}\Vert_{C^s(\Theta_\alpha(U_\alpha,G)}$.\\
For all $p\in \mathbb{N}$, $q\in \mathbb{R}_+$, $\ell\in\{0,1, 2\}$, we finally define the spaces $\cB^{p,q,\ell}$ as the closure of the $\cC^\infty$ $\ell$ forms with respect to the norm
\[
    \Vert h \Vert_{p,q,\ell}= \sup_{n\leq p} \sup_{J\in \mathbb{U}_{\rho,L,n,q,\ell}} J(h).
\]
Note that $\cB^{p,q,\ell}$ is contained in the space of $\ell$ ($p+q$ smooth) currents (see \cite{GLP13}).
\begin{rem}\label{rem:notation_der}
    Note that $(\Theta_\alpha)_*(L_vh)=L_{(\Theta_\alpha)_* v}(\Theta_\alpha)_*h$, thus, in coordinates, $L_vh$ will have the form $\sum_{i=1}^2\alpha_iL_{e_i} h=:\sum_{i=1}^2\alpha_i\partial_{x_i} h$, for some functions $\alpha_i$ and where  $\{e_i\}_{i=1}^2$ is the standard basis of $\bR^2$. If follows that restricting the vector fields in \eqref{eq:norm-anosov} to $\Theta_\alpha^*e_i$ yields an equivalent norm. Hence, to simplify notation, in the following we will use the notation $\partial^\beta h$ to designate the application of $|\beta|$ vector fields to $h$, where $\beta$ are the usual multiindexes used in PDE.
\end{rem}
The main result of this section consists in the following Theorem which provides a rather precise characterisation of the spectrum of the action on one forms, which is well known to be related to the measure of maximal entropy and thus plays the same role of the operators $\cL_0$ in the previous sections.

\begin{thm}
    \label{thm:maximal_entropy_gap}
    For each $\ve>0$, for $p,q$ large enough,
    \[
        \begin{split}
            %&\sigma_{\cB^{p,q,1}}(f_*)\cap \{z\in\bC\;:\; |z|>\kappa\}=\{e^{\htop}\} \\
            &\{e^{-\htop}, e^{\htop}\}\cup (\sigma_{\cB^{p+1,q-1,0}}(f_*)\setminus \{z\in\bC\;:\; |z|<\ve\})\subset \sigma_{\cB^{p,q,1}}(f_*)\\
            &\sigma_{\cB^{p,q,1}}(f_*) \subset  \{z\in\bC\;:\; |z|<\ve\}\cup\{e^{-\htop}, e^{\htop}\}\cup \sigma_{\cB^{p+1,q-1,0}}(f_*)\cup\sigma_{\cB^{p-1,q+1,0}}(\cL) ,
        \end{split}
    \]
    moreover $1\not \in \sigma_{\cB^{p,q,1}}(f_*)$. \\
    In addition, $\sigma_{\cB^{p+1,q-1,0}}(f_*)\setminus \{z\in\bC\;:\; |z|<\ve\}$ consists only of point spectrum and  there exists $\kappa\in (0,1)$ such that $\sigma_{\cB^{p+1,q-1,0}}(f_*)\subset\{1\}\cup\{z\in\bC\;:\; |z|<\kappa\}$. The same holds for $\sigma_{\cB^{p-1,q+1,0}}(\cL)$.
\end{thm}

\begin{rem} In fact, we conjecture that, for $\lambda^{\min\{p,q\}}> \ve^{-1}$,
    \[
        \sigma_{\cB^{p,q,1}}(f_*)\setminus  \{z\in\bC\;:\; |z|<\ve\}= \Big[\{e^{-\htop}, e^{\htop}\}\cup \sigma_{\cB^{p+1,q-1,0}}(f_*)\cup\sigma_{\cB^{p-1,q+1,0}}(\cL) \Big]\setminus\{1\},
    \]
    see Remark \ref{rem:conjecture}. This would be consistent with the fact that, by duality, the spectrum of $\cL$ equals the spectrum of $f^{-1}_*$ and that the spectra of $f_*$ on forms determine the Ruelle zeta function, see \cite{Ruelle76}, and the latter is described in term of periodic orbits, which are the same for $f$ and $f^{-1}$. Accordingly, one expects a symmetry between the spectra of $f_*$ and $f_*^{-1}$.\footnote{ Note that if $f_*$ acts on some Banach space, then here one considers $f_*^{-1}$ acting on its dual, so the relation is not obvious a priori. Indeed, one does not necessarily expect $f^{-1}_*$ to be a bounded operator when acting on a Banach space on which $f_*$ is bounded.}
\end{rem}
The next section is devoted to the proof of the above Theorem, while in Section \ref{sec:derham} we present a minimalistic discussion of cohomology in the spaces $\cB^{p,q,1}$ and Section \ref{sec:5comp} is devoted to comments on the implications of such a Theorem and a comparison with existing results.

\subsection{Proof of Theorem \ref{thm:maximal_entropy_gap}}\ \\
We start with some preliminary results establishing minimal information about Hodge duality and exterior differentials in our spaces of currents $\cB^{p,q,\ell}$.

\begin{lem}
    \label{lem:duality}
    The  Hodge duality map $\Phi h:=\star h=h\omega_0$, between zero forms and two forms, extends to a bounded isomorphism between $\cB^{p,q,0}$ and $\cB^{p,q,2}$ and $\Phi \cL=f_*\Phi$. In particular, $\sigma_{\cB^{p,q,2}}(f_*)=\sigma_{\cB^{p,q,0}}(\cL)$.
\end{lem}
\begin{proof}
    By equation \eqref{eq:hdgedef}, for each smooth zero form $h$, $\Phi h= h\omega_0$. Thus equation \eqref{eq:duality-zero-due} implies $\Phi \cL h=f_*\Phi h$ for each smooth zero form. The injectivity follows since $\cB^{p,q,0}, \cB^{p,q,2}$ are isomorphic to a subspace of the space of currents, see \cite[Lemma 3.10]{GLP13}, and the extension of $\Phi$ to the current is an isomorphism. The result then follows by proving that $\Phi$ is a bounded operator. For each multi-index $\alpha$, $|\alpha|=p$, smooth two form $\omega$ and zero form $h$ we have
    \[
        \left|\int_W\langle \omega, \partial^\alpha\Phi(h)\rangle\right| \le \sum_{\beta+\gamma=\alpha}\left|\int_W \langle \omega,\partial^\beta\omega_0\rangle\partial^\gamma h\right|\leq \Const {\|\omega\|}_{\cC^{q+p}(W)}\|h\|_{p,q,0}
    \]
    from which the claim follows.
\end{proof}

\begin{lem}
    \label{lem:d-action}
    The exterior derivative $d$ extends to a bounded operator $\cB^{p,q,\ell}\to\cB^{p-1,q+1,\ell+1}$.\footnote{ With a slight abuse of notation we will call such an extension $d$ as well.} If $h\in\cB^{p,q,0}$, $\psi\in\cC^\infty(M,\bR_+)$, with the interior of $\supp(\psi)$ connected, and $\psi dh=0$, then there exists $c\in\bC$ such that $\psi(h-c)=0$.\footnote{ This essentially says that closed  anisotropic zero currents are constant. Since for zero currents being closed and being harmonic is the same, this is a little piece of Hodge theory, all that is presently needed. Yet, it would be clearly useful to develop the Hodge theory in the context of anisotropic spaces.} Finally,  $d(\cB^{p,q,0})$ is closed in $\cB^{p-1,q+1,1}$.
\end{lem}
\begin{proof}
    If $h$ is an $\ell$ form, then, for each $\ell+1$ form $\omega$ and multi-index $\alpha$, $|\alpha|=p-1$, we have that there exists a constant $C_\ell>0$ such that
    \[
        \left|\int_W\langle \omega, \partial^\alpha dh\rangle\right|\leq C_\ell\|\omega\|_{\cC^{p+q}(W)}\|h\|_{p,q,\ell}
    \]
    from which it follows that $\|d h\|_{p-1,q+1,\ell+1}\leq C_\ell\|h\|_{p,q,\ell}$.

    Next, let $h,\psi$ be such that $\psi dh=0$. Note that $\cB^{p,q,0}$ is isomorphic to a subspace of the space of distributions $(\cC^{p+q})'$, see \cite[Lemma 3.10]{GLP13}. Let $K=\supp\psi$ and $U=\overset{\,\circ}{K}$, note that $U$ is connected by hypothesis. Thus for each smooth local function $\vf$, $\supp\vf\subset U$, and disintegration of $\omega$ along a smooth foliation $\{W_t\}\subset \Sigma$, we have\footnote{ Since the foliation is smooth the Jacobian $J$ of the disintegration is a smooth function and $\hat\vf=J\vf$. Note that $\frac{\hat \vf}{\psi}$ is a smooth function on $W_t$.}
    \begin{equation}\label{eq:image}
        \int_M \vf\partial_{x_i} h=\int dt \int_{W_t} \langle \frac{\vf}{\psi} d x_i, \psi dh\rangle =0.
    \end{equation}
    It follows that $\partial_{x_i} h=0$ as a distribution on $U$, hence $h=c$ on $U$, for some $c\in\bC$. That is $\psi(h-c)=0$ on $M$. From \cite[Lemma 3.10]{GLP13}, again, it follows that $\psi(h-c)=0$ as an element of $\cB^{p,q,0}$.

    To conclude the Lemma we want to prove that $d(\cB^{p,q,0})$ is closed in $\cB^{p-1,q+1,1}$. Let us suppose that $\omega_n\to\omega$, in $\cB^{p-1,q+1,1}$, with $\omega_n\in d(\cB^{p,q,0})$. That is, there exists $ \Xi_n\in \cB^{p,q,0}$ such that $\omega_n=d\Xi_n$. Let $\widehat\Xi_n=\Xi_n-\int_M\Xi_n$. Then, for each function $\vf$ supported in a chart $(U_\alpha,\Theta_\alpha)$ we can write
    \[
        \int \vf \widehat\Xi_n=\int_M\left( \vf-\int_M\vf\right)\Xi_n.
    \]
    Let $\bar x$ be such that $\int_M\vf=\vf(\bar x)$, then \footnote{To simplify notation we do not write explicitly the change of coordinates
        $\Xi_\alpha $.}
    \[
        \begin{split}
            \int \vf \widehat\Xi_n&=\int_{U_\alpha}dx\int_0^1dt \frac{d}{dt}\vf(\bar x+(x-\bar x)t) \Xi_n(x)\\
            &=-\sum_{i=1}^2\int_M \vf(\bar x+(x-\bar x)t){\langle (x_i-\bar x_i),\partial_{x_i}\Xi_n(x)\rangle}.
        \end{split}
    \]
    Hence, setting $\Psi_t(x)=-\sum_{i=1}^2\vf(\bar x+(x-\bar x)t)(x_i-\bar x_i)d x_{i}$, we have, recalling equation \ref{eq:scalare},
    \[
        \int \vf \widehat\Xi_n=\int_0^1 dt\int_M\langle \Psi_t, \omega_n\rangle.
    \]
    Arguing similarly for $\partial^\alpha\widehat\Xi_n$ it follows that $\widehat\Xi_n$ is a Cauchy sequence. Let $\Xi$ be the limit, then, by the continuity of $d$, $d\Xi=\omega$, hence $\omega\in d(\cB^{p,q,0})$.
\end{proof}

\subsubsection{\bfseries Spectral radius and essential spectral radius of $f_*$ and $\cL$}\ \\
The first step in the study of the operators $f_*,\cL$ is the following.
\begin{lem}
    \label{eq:lasota-yorke-hyp}
    The action of $f_*$ on $\ell$-forms extends to a linear bounded operator from $\cB^{p,q,\ell}$ to itself. With a slight abuse of notation we use $f_*$ for the action on each $\cB^{p,q,\ell}$. Then, there exists a constant $C_\sharp>0$ such that
    \[
        \begin{split}
            &{\|f_*^n h\|}_{0,q,0}\leq C_\sharp {\|h\|}_{0,q,0}\\
            &{\|f_*^n h\|}_{0,q,1}\leq C_\sharp e^{\htop n} {\|h\|}_{0,q,1}\\
            &{\|f_*^n h\|}_{p,q,0}\leq  C_\sharp \lambda^{-np} {\|h\|}_{p,q,0}+ C_\sharp {\|h\|}_{p-1,q+1,0}\\
            &{\|f_*^n h\|}_{p,q,1}\leq C_\sharp e^{\htop n} \lambda^{-np}{\|h\|}_{p,q,1}+C_\sharp  e^{\htop n}{\|h\|}_{p-1,q+1,1}\\
            &{\|\cL^n h\|}_{p,q,0}\leq C_\sharp \lambda^{-np}{\|h\|}_{p,q,0}+C_\sharp {\|h\|}_{p-1,q+1,0}.
        \end{split}
    \]
\end{lem}
\begin{proof}
    To start with let $h\in\cC^\infty(M,\bC)$ be a function, then
    \[
        \int_W\langle\vf, f_*^nh\rangle=\int_W\langle \vf, h\circ f^{-n}\rangle=\int_{f^{-n}W}\langle\vf\circ f^n, h \lambda^s_n\rangle
    \]
    where $\lambda^s_n(x)$ is the contraction of $f^n$ in the direction $T_*W$ at the point $x$. We can divide $f^{-n}W$ in a collection $\{W_i\}\subset \Sigma$. Let $\{\vartheta_i\}$ be a smooth partition of unity subordinated to  $\{W_i\}$. If $\lambda^s_{n,i}=\min_{x\in W_i}\lambda^s_n(x)$, then  the usual distortion arguments  implies, that for all $x\in W_i$, $\Const \lambda_{n,i}^s\leq\lambda^s_n(s)\leq \Const \lambda_{n,i}^s$, thus, integrating $|f^n(W_i)|=\int_{W_i}\lambda^s_{n,i} \geq \Const \lambda_{n,i}^s\delta$, where $\delta$ is the size of the manifolds (see the beginning of Section \ref{sec:banch}). In addition, for all $q\in \bN$, $\|\lambda^s_n(s)\|_{\cC^q(W_i)}\leq \Const \lambda_{n,i}^s$.  Accordingly,
    \[
        \begin{split}
            \left|\int_W\langle \vf, f_*^nh\rangle\right|&\leq \sum_i \left|\int_{W_i}\langle\vartheta_i \lambda^s_n\vf\circ f^n, h \rangle\right|\\
            &\leq\Const\sum_i |f^n(W_i)|\delta^{-1} {\|\vf\circ f^n\|}_{\cC^q(W_i)} {\|h\|}_{0,q,0}\\
            &\leq \Const {\|\vf\|}_{\cC^q(W)} {\|h\|}_{0,q,0}
        \end{split}
    \]
    which, by density, proves the first inequality of the Lemma.  Next, recalling \eqref{eq:scalare}, we have, for each $v=\pi(\omega), h$ a $\cC^\infty$ one form,
    \[
        \left|\int_W\langle\omega, f_*^nh\rangle\right|\leq \left| \int_W f_*^nh(v)\right| =\left|\int_W h_{f^{-n}(x)}(d_xf^{-n}v(x)) dx\right|.
    \]
    Setting $v_n(x)=d_{f^n(x)}f^{-n}v(f^{n}(x))$, by the usual distortion arguments we have ${\|v_n\|}_{\cC^{q}(W_i)}\leq \Const (\lambda^s_{n,i})^{-1} {\|v\|}_{\cC^{q}(W)}$, hence
    \[
        \begin{split}
            \left|\int_W\langle\omega, f_*^nh\rangle\right|&\leq \Const \sum_i \left|\int_{W_i}h(\vartheta_i \lambda^s_n v_n) \right|
            \leq\Const\sum_i |W_i|\delta^{-1}  {\|h\|}_{0,q,1}\\
            &\leq \Const e^{h_{\textit{top}}n} {\|\omega\|}_{\cC^q} {\|h\|}_{0,q,1}
        \end{split}
    \]
    where, in the last line, we have used $|f^{-n}W|\sim e^{\htop n}$, see \cite[Appendix D]{GLP13}. Taking the sup on $W$ and $\omega$ the second inequality of the Lemma follows.

    The next two inequalities are proven, similarly, as done in \cite[Lemma 4.7]{GLP13}, while the last follows by \cite[Lemma 2.2]{GL06} taking into account Lemma \ref{lem:duality}.
\end{proof}

We are now able to obtain a first information on the peripheral spectrum.
\begin{lem}
    \label{lem:spectrum1}
    For $p,q$ large enough, the spectra of $f_*$ on $\cB^{p,q,0}$ and on $\cB^{p,q,2}$ are contained in $\{1\}\cup\{z\in\bC\;:\; |z|< \kappa\}$ for some $\kappa<1$. The eigenvectors associated to the eigenvalue $1$ are the constant function $1$ and the measure $\srb$ respectively.
\end{lem}
\begin{proof}
    By  Lemma \ref{lem:duality} the action of $f_*$ on $\cB^{p,q,2}$ is conjugated to the action of $\cL$ on $\cB^{p,q,0}$, thus they have the same spectrum.
    But  \cite{GL06} implies that there exists $\kappa \in (0,1)$ such that $\sigma_{\cB^{p,q,0}}(\cL)\subset \{1\}\cup\{z\in\bC\;:\; |z|\leq \kappa\}$ and one is a simple eigenvalue. This proves the Lemma for $\cB^{p,q,2}$.

    Let us discuss $\cB^{p,q,0}$. Lemma \ref{eq:lasota-yorke-hyp}  and Hennion's Theorem \cite{Hennion93} (or see \cite[Appendix B]{DKL}) imply that the radius of the essential spectrum of  $f_*$ acting on $\cB^{p,q,0}$ is at most $\lambda^{-1}$ while the spectral radius is one, moreover the operator is power bounded. Accordingly, if $f_*$ has no eigenvalue on the unit circle apart from $1$ and $1$ is a simple eigenvalue, then there exists a $\kappa$ that satisfies the Lemma. Thus, we need only study eigenvalues of the form $e^{i\theta}$. Since the operator is power bounded, there cannot be a (non-trivial) Jordan block associated to such a maximal eigenvalue, hence their geometric and algebraic multiplicity coincide.  Hence, we have the spectral decomposition
    \[
        f_*=\sum_{j} e^{i\theta_j}\Pi_j+Q
    \]
    where $\theta_j\in\bR$, $\Pi_i\Pi_j=\delta_{ij}\Pi_j$, $\Pi_jQ=Q\Pi_j=0$ and $\|Q^n\|_{p,q,0}\leq \Const \kappa^n$.
    Suppose that $e^{i\theta}\in\sigma_{\cB^{p,q,0}}(f_*)$, $\theta\in\bR\setminus\{0\}$, then there exists $h\in\cB^{p,q,0}\setminus\{0\}$ such that $f_*h=e^{i\theta}h$, and, by the spectral decomposition, there exists $h_0\in\cC^\infty$ such that
    \[
        h=\lim_{n\to \infty}\frac 1n\sum_{k=0}^{n-1}e^{-i\theta k}f_*^kh_0.
    \]
    It follows that, for all $\vf\in\cC^\infty(M,\bC)$,\footnote{ Note that the integral on $M$ can be decomposed as an integral over elements of $\Sigma$, which are continuous functionals in the $\cB^{p,q}$ norms, hence we can exchange the limit with the integral.}
    \[
        \begin{split}
            \int_M \vf h\omega_0&=\lim_{n\to \infty}\frac 1n\sum_{k=0}^{n-1}e^{-i\theta k}\int_M \vf f_*^kh_0\omega_0 =\lim_{n\to\infty}\frac 1n\sum_{k=0}^{n-1} e^{-i\theta k}\int_M  \vf \cdot h_0\circ f^{-k}\omega_0.
        \end{split}
    \]
    But $f^{-1}$ is also a transitive Anosov diffeomorphism with its SRB measure, call it $\mu^-_{\textrm{SRB}} $, then
    \[
        \lim_{k\to\infty} \int_M  \vf \cdot h_0\circ f^{-k}\omega_0=\int_M \vf\omega_0\int_M  h\mu^-_{\textrm{SRB}}
    \]
    which implies
    \[
        \int_M \vf h\omega_0=0
    \]
    and since $\cB^{p,q,0}$ is a space of distributions, see \cite{GLP13}, it follows $h=0$ contrary to the hypothesis. We are left with the case $\theta=0$, that is the eigenvalue $1$. Since $1\circ f^{-1}=1$, one is an eigenvalue, we want to prove that it is simple. Let $f_*h=h$, and let $h_0$ as before,
    \[
        \begin{split}
            \left|\int_M \vf h\omega_0\right|&\leq \lim_{n\to\infty}\frac 1n\sum_{k=0}^{n-1} \int_M |\vf| \cdot |h_0|\circ f^{-k}\omega_0\leq \|\vf\|_{L^{1}(\omega_0)} \|h_0\|_{\cC^0}.
        \end{split}
    \]
    Thus $h\in L^\infty(\omega_0)$ (the dual of $L^1(\omega_0)$), and $h=h\circ f^{-1}$, $\omega_0$-almost surely. On the other hand, since $\omega_0$ is ergodic for the Anosov map, it follows that $h$ is almost surely constant. Thus $1$ is simple point spectrum for $f_*$ acting on $\cB^{p,q,0}$.
\end{proof}

%%%%%%%%%%%%%%%%%%%%%%%%%%%%%%%%%%%%%%%%%%%%%
\subsubsection{ \bfseries The peripheral spectrum of $f_*$ acting on $\cB^{p,q,1}$}\ \\
Here, we start a more in depth study of $\sigma_{\cB^{p,q,1}}(f_*)$.
\begin{lem}
    \label{lem:spectrume}
    The spectrum of $f_*$ on $\cB^{p,q,1}$ contains $e^{h_{\textit{top}}}$, which is also the spectral radius. In addition, the essential spectral radius is bounded by $ e^{h_{\textit{top}}} \lambda^{-\min\{p,q\}}$. The eigenvector associated to $e^{h_{\textit{top}}}$ is the Margulis measure and, together with the dual eigenvector defines the measure of maximal entropy.
\end{lem}
\begin{proof}
    The statement on the spectral radius and essential spectrum  follows from Lemma \ref{eq:lasota-yorke-hyp}  and Hennion's Theorem \cite{Hennion93} (see also \cite[Appendix B]{DKL}). Next, if $\nu$ is an eigenvalue of $f_*$,  $|\nu|=e^{h_{\textit{top}}}$, then, by Lemma \ref{eq:lasota-yorke-hyp}, $\nu^{-n}f_*$ is power bounded, hence it cannot be associated to a Jordan block. Let $h\in\cB^{p,q,1}$ be an eigenvalue, then,  by Lemma \ref{eq:lasota-yorke-hyp} again,
    $\|h\|_{p,q,1}\leq \Const \|h\|_{0,p+q,1}$.

    Next, let $\bar v^u(x)\in\cC(x)$ be a smooth normalised vector field. Then, for each $\omega\in\cC^{p+q}$, $W\in \Sigma$ and $n\in\bN$ let $v=\pi(\omega)$ and $v=w+v^u$ where $v$ belongs to the tangent space of $W$ and $v^u(x)=\alpha(f^{-n}(x))df^n_{f^{-n}(x)}\bar v^u(f^{-n}(x))$. Note that $\|\alpha\|_{\cC^{p+q}(f^{-n}(W))}\leq \lambda^{-n}\Const\|\omega\|_{\cC^{p+q}}$. Thus,
    \[
        \begin{split}
            \left|\int_W\langle\omega, h\rangle\right|&=\left|\int_W\langle\omega, \nu^{-n} f_*^nh\rangle\right| =\left|\nu^{-n}\int_W f_*^nh(v)\right|\\
            & \leq\left|\int_W  \nu^{-n}h_{f^{-n}(x)}(d_xf^{-n}w(x))dx\right|+\left|\int_W  \nu^{-n}\alpha(f^{-n}(x))h_{f^{-n}(x)}(\bar v^u(f^{-n}(x))dx\right|\\
            &\leq \Const \sum_i\int_{W_i} \vartheta_i(x)\nu^{-n}h_{x}(\bar w_i(x)) dx+\Const \lambda^{-n}\|h\|_{0,p+q,1} \|v\|_{\cC^{p+q}}
        \end{split}
    \]
    where $\bar w_i$ belongs to the tangent space of $W_i$ and $\|\bar w_i\|_{\cC^{p+q}}\leq \|w\|_{\cC^0}+\lambda^{-n}\|w\|_{\cC^{p+q}}$. Thus,
    \[
        \begin{split}
            \left|\int_W\langle\omega, h\rangle\right|&\leq \Const |f^{-n}W|\delta^{-1}e^{-h_{\textit{top}}n}\|h\|_{0,p+q,1}\|w\|_{\cC^0}+\Const \lambda^{-n}\|h\|_{0,p+q,1}\|v\|_{\cC^{p+q}} \\
            &\leq \Const \delta^{-1}\|h\|_{0,p+q,1}\|\omega\|_{\cC^0}+\Const \lambda^{-n}\|h\|_{0,p+q,1}\|\omega\|_{\cC^{p+q}}
        \end{split}
    \]
    where, in the last line, we have used the estimate on the growth of invariant manifolds, see \cite[Appendix C]{GLP13} for details.
    Taking the limit $n\to \infty$ and the sup in $W$ and $\omega$ yields
    \[
        \|h\|_{0,0,1}\leq \Const \|h\|_{p,q,1}.
    \]
    Next, let $v^s$ be the normalised stable direction. Then, setting $\bar\phi(x)=\ln |d_{f(x)} f^{-1} v^{s}(f(x))|$,
    \[
        d_xf^{-n}v^s(x)=e^{\sum_{k=0}^{n-1}\bar\phi\circ f^{-n+k}(x)}v^s(f^{-n}(x)).
    \]
    We can then define the transfer operator
    \begin{equation}
        \label{eq:topOp}
        \cL_*g(x)=g\circ f^{-1}(x)e^{\bar\phi(x)}.
    \end{equation}
    Defining the map $\Gamma:\cB^{0,0,1}\to \cM$, the space of signed measures, by $\Gamma(h)=h(v^s)\omega_0$, we have
    \[
        \int_M f_*(h)( g v^s)=\int_M g \cL_* \Gamma(h).
    \]
    In \cite{GL08} is proven that $\cL_*$ is has maximal eigenvalue $e^{-h_{\textit{top}}}$ and the associated eigenvector is the Margulis measure. This concludes the Lemma.
\end{proof}
\subsubsection{\bfseries Deeper in the spectrum of $f_*$ acting on $\cB^{p,q,1}$}\ \\
By Lemma \ref{lem:d-action} we can extend the de Rham cohomology to the currents in the spaces $\cB^{p,q,\ell}$.  In other words we can call {\em closed} the elements  $\omega\in\cB^{p,q,\ell}$ such that $d\omega=0$ and {\em exact} the ones for which it exists $\alpha\in \cB^{p+1,q-1,\ell-1 }$ such that $\omega=d\alpha$.

\begin{rem}\label{rem:c-cation}
    By equation \eqref{eq:1-0} and Lemma \ref{lem:d-action} it follows that $f_*$ sends closed currents into closed currents and exact currents into exact currents. Hence $f_*$ induces an action in cohomology (of the $\cB^{p,q,\ell}$ currents), let us call it $f_\sharp$.
\end{rem}
The next result shows that such a cohomology (let us call it {\em anisotropic cohomology}) is relevant to our problem.
\begin{lem}
    \label{lem:cosed-exact}
    If $\nu\in\sigma_{\cB^{p,q,1}}(f_*)$, and $\omega\in \cB^{p,q,1}$ are such that $f_* \omega=\nu \omega$ and $|\nu|>e^{\htop}\lambda^{-\min\{p,q\}}$,
    then either $\omega$ is not exact or $\nu\in \sigma_{\cB^{p+1,q-1,0}}(f_*)\setminus\{1\}$. Moreover, $\sigma_{\cB^{p+1,q-1,0}}(f_*)\setminus\{1\}\subset \sigma_{\cB^{p,q,1}}(f_*)$.
    If $\nu\in(\sigma_{\cB^{p,q,1}}(f_*)\setminus\sigma_{\cB^{p-1,q+1,0}}(\cL))\cup\{1\}$, then for each $\omega\in \cB^{p,q,1}$ such that $f_* \omega=\nu \omega$ we have $d\omega=0$.
\end{lem}
\begin{proof}
    To start with note that, by Lemma \ref{lem:spectrume}, $\nu$ must belong to the point spectrum.

    Let $\nu\in(\sigma_{\cB^{p,q,1}}(f_*)\setminus\sigma_{\cB^{p,q,0}}(f_*))\cup\{1\}$ and $\omega\in \cB^{p,q,1}$ such that $f_* \omega=\nu \omega$  and suppose that
    $\omega$ is exact. Thus, there exists $h\in \cB^{q-1,p+1,0}$ such that $dh=\omega$. This implies
    \[
        \nu d h=f_* dh=d f_* h.
    \]
    That is $d(f_* h-\nu h)=0$. It follows by Lemma \ref{lem:d-action} that $f_* h=\nu h+c$. By a change of variable it follows that the dual $(f_*)'$ of $f_*$ is given by the transfer operator $\cL_{f^{-1}}$ associated to the map $f^{-1}$. Since $f^{-1}$ is Anosov as well Lemmata \ref{eq:lasota-yorke-hyp} and \ref{lem:spectrum1} apply and the measure $\srb^{-}$ associated to $f^{-1}$ belongs to the dual of $\cB^{p,q,0}$.
    Since $f_* 1=1$ and the space $\bV=\{h\;:\; \int_M h d\srb^{-}=0\}$ is invariant for $f_*$, it is natural to write $h=\alpha+g$ with $\alpha\in\bC$ and $g\in\bV$.
    Then, we have
    \[
        c+\nu \alpha+\nu g=\alpha+f_*g.
    \]
    Applying $\srb^{-}$ to the above implies $c=\alpha(1-\nu)$, hence $\nu g=f_* g$. The only possibility is then $\nu=1$ but the associated eigenvector  would be $1\not \in \bV$, it follows $g=0$. But then $\omega=dh=d\alpha=0$. Hence, $\omega$ cannot be exact. The inclusion of the spectra is obvious.

    If $f_*\omega=\nu\omega$ and $d\omega=h\omega_0$, by Lemma \ref{lem:duality} we have $\cL h=\nu h$.
    Accordingly, either $\nu\in\sigma_{\cB^{p-1,q+1,0}}(\cL)$ or $h=0$, that is $d\omega=0$. On the other hand, if $\nu=1$, then $h\omega_0=\mu_{\textrm{SRB}}$.
    Hence, $d\omega=\srb$ and
    \[
        \int_M \srb=\int_M d\omega=0
    \]
    which is impossible since $\srb$ is a positive measure.
    Accordingly, it must be $d\omega=0$, that is, again, the form is closed.
\end{proof}
To conclude we need a theory of anisotropic de Rham cohomology, such a general theory goes a beyond our present scopes so we will develop only the minimal version needed here. This is contained in Section \ref{sec:derham}, and in particular in  Lemma \ref{lem:derham} which states that the anisotropic cohomology of one forms is isomorphic to standard de Rham cohomology. In particular, this implies that the vector space of the equivalence classes is finite dimensional, hence $f_\sharp$, defined in Remark \ref{rem:c-cation}, has only point spectrum, let us call $\Omega$ the spectrum of $f_\sharp$ when acting on one forms.

Next, we want to identify $\Omega$. As stated in Remark \ref{rem:comment} this is the only place where we use that our map is topologically conjugated to the linear model.
\begin{lem}\label{lem:thisistheend}
    We have $\Omega=\{e^{-\htop}, e^{\htop}\}$.
\end{lem}
\begin{proof}
    Lemma \ref{lem:derham} implies that the anisotropic de Rham cohomology for one forms is a topological invariant, hence so is $f_\sharp$. Since our map is conjugated to a linear model (see \cite[Section 2.6]{KH95}), $f_\sharp$  is conjugated to the action of the linear model on homology. The Lemma follows by a direct computation, see \cite[Section 3.2-e]{KH95} for details.
\end{proof}
The following Lemma concludes the proof of Theorem \ref{thm:maximal_entropy_gap}.
\begin{lem}
    \label{lem:ps-holo}
    For each $\ve>0$, if $p,q$ are large enough, we have
    \[
        \begin{split}
            &\Big[\Omega\cup \sigma_{\cB^{p+1,q-1,0}}(f_*)\setminus\Big(\{1\}\cup  \{z\in\bC\;:\; |z|<\ve\}\Big)\Big]\subset \sigma_{\cB^{p,q,1}}(f_*)\\
            &\sigma_{\cB^{p,q,1}}(f_*) \subset \Big[ \{z\in\bC\;:\; |z|<\ve\}\cup \Omega\cup \sigma_{\cB^{p+1,q-1,0}}(f_*)\cup\sigma_{\cB^{p-1,q+1,0}}(\cL) \Big]\setminus\{1\}.
        \end{split}
    \]
\end{lem}
\begin{proof}
    Lemma \ref{lem:spectrume} implies that if $p,q$ are large enough we have to worry only about point spectrum.

    Thus, if $\nu \in\sigma_{\cB^{p+1,q-1,0}}(f_*)$ then there exists $\theta\in \cB^{p+1,q-1,0}$ such that $f_*\theta=\nu\theta$. This implies that $f_*d\theta=\nu d\theta$ so either $d\theta=0$, but then by Lemma \ref{lem:d-action} we have $h$ constant and $\nu=1$, or $\nu\in\sigma_{\cB^{p,q,1}}(f_*)$.

    If $\nu\in\Omega$, the spectrum of $f_\sharp$ (defined in Remark \ref{rem:c-cation}), then it means that there exists $\omega\in\cB_0^{p,q,1}$ and $\psi\in\cB^{p+1,q-1,0}$ such that $f_*\omega=\nu\omega+ d\psi$, that is $f_\sharp [\omega]=\nu[\omega]$, where $[\omega]\neq 0$ is the equivalence class of $\omega$.  If $\nu \not \in \sigma_{\cB^{p+1,q-1,0}}(f_*)$, we can define $\theta=(\nu-f_*)^{-1}\psi$ and
    \[
        (\nu-f_*)d\theta=d\psi.
    \]
    But then $f_*(\omega+d\theta)=\nu(\omega+d\theta)$ which implies $\nu\in \sigma_{\cB_0^{p,q,1}}(f_*)$  unless $\omega+d\theta=0$. But the latter possibility would imply that $\omega$ is exact, that is $[\omega]=0$, contrary to the assumption. This proves the first inclusion of the Lemma.

    To prove the second inclusion note that if $f_*\omega=\nu\omega$, $\omega\in  \cB^{p,q,1}$ and $\nu\not \in \sigma_{\cB^{p-1,q+1,0}}(\cL)\setminus\{1\}$, then the last part of the Lemma \ref{lem:cosed-exact} implies $d\omega=0$. Then $f_\sharp[\omega]=\nu[\omega]$, thus either $\nu\in\Omega$ or $[\omega]=0$, i.e. $\omega$ is exact. But if $\nu\not \in \sigma_{\cB^{p+1,q-1,0}}(f_*)\setminus\{1\}$ the first part of Lemma \ref{lem:cosed-exact} implies that $\omega$ is not exact, hence $[\omega]\neq 0$.
    Since Lemma \ref{lem:thisistheend} implies $1\not \in\Omega$, the Lemma follows.
\end{proof}

\begin{rem} \label{rem:conjecture} It is conceivable that Lemma \ref{lem:ps-holo} could be upgraded to an equality. Indeed, suppose that for a two current $\int_M\omega=0$ implies that there exist a one current $\theta$ such that $\omega=d\theta$.\footnote{  This is equivalent to studying the cohomology for two forms.} Then if $f_*\omega=\nu\omega$, $\nu\neq 1$ and $\omega\not\equiv 0$, we have $\int_M\omega=0$ thus we can write $\omega=d\theta$ and $d(\nu\theta-f_*\theta)=0$. Thus $\nu\theta-f_*\theta=\psi$ with $d\psi=0$. Hence, if  $\nu\not\in \sigma_{\cB^{p,q,1}}(f_*)$, we have $\theta=(\nu-f_*)^{-1}\psi$. Since $d(z-f_*)^{-1}\psi$ is a meromorphic function and for large $z$ the von Neumann expansion implies the it is zero, we have $d\theta=0$, a contradiction. Hence the second inclusion of the  Lemma \ref{lem:ps-holo} would be an equality.
\end{rem}
%%%%%%%%%%%%%%%%%%%%%%%%%%%%%%%%%%%%%%%
\subsection{ \bfseries Application to the measure of maximal entropy}\label{sec:max_proof}
In this section we prove Theorem \ref{thm:max_hyp}.

Lemma \ref{lem:spectrume} implies that there exists $\ell_\star\in(\cB^{p,q,1})'$ and $h_\star\in\cB^{p,q,1}$ such that $f_* h_\star=e^{\htop}h_\star$ and
$\ell_\star(f_* \omega)=e^{\htop}\ell_\star(\omega) $, for all $\omega\in \cB^{p,q,1}$. In addition, $\ell_\star(\vf h_\star)=\mme(\vf)$. Lemmata \ref{lem:ps-holo} and \ref{lem:spectrum1} imply that the rest of the spectrum is contained in $\{z\in\bC\;:\; |z|<\kappa\}$ for some $\kappa\in (0,1)$. It follows that the spectral decomposition $f_*=e^{\htop}h_\star\otimes \ell_\star+\cQ$ with $\ell_\star\cQ=0$, $\cQ h_\star=0$, $\ell(h)=1$ and $\|\cQ^n\|_{p,q,1}\leq\Const \kappa^n$. Also note that the multiplication by a smooth function is a bounded operator. Thus
\[
    \begin{split}
        \int_M g\circ f^n h d\mme&=\ell_\star(g\circ f^n h h_\star)=e^{-n\htop}\ell_\star(f_*^n(g\circ f^n h h_\star))=e^{-n\htop}\ell_\star(gf_*^n( h h_\star))\\
        &=\ell_\star(g h_\star)\ell_*( h h_\star)+e^{-n\htop}\ell_\star(g\cQ^n( h h_\star)).
    \end{split}
\]
It follows that, for $r$ large enough,
\[
    \left|\int_M g\circ f^n h d\mme-\int_M g d\mme \int_M  h d\mme\right|\leq \Const\|g\|_{\cC^r}\|h\|_{\cC^r}e^{-n\htop}\kappa^n.
\]
%%%%%%%%%%%%%%%%%%%%%%%%%%%%%%%%%%%%%
\subsection{Anisotropic de Rham cohomology}\label{sec:derham}\ \\
While to develop a theory of anisotropic de Rham cohomology as well as the relative Hodge theory may certainly be of interest, in this section we will develop only the bare minimum necessary to our needs and we will keep the arguments as elementary as possible.

Without loss of generality we can, and will, assume that there exist  {\em good covers} $\{U^+_\alpha\}$, and $\{U_\alpha\}$ such that $U^+_\alpha\supset\overline{U}_\alpha$,  and a partition of unity  $\{\psi_\alpha\}$ subordinated to  $\{U_\alpha\}$. Also let $\{\psi^+_\alpha\}$ be such that $\supp(\psi^+_\alpha)\subset U^+_\alpha$ and $\psi^+_\alpha|_{U_\alpha}=1$.\footnote{ Recall that a good cover is a cover such that, for each collection $\cA$ of indexes, $\cap_{\alpha\in\cA} U_\alpha$ is contractible.}

\begin{lem}\label{lem:primitiva}
    If $h\in\cB^{p,q,1}$ is closed then, for each $\alpha$, there exists $H_\alpha, H_\alpha^+\in \cB^{p+1,q-1,0}$ such that $dH_\alpha=h\psi_\alpha +H^+_\alpha d\psi_\alpha$ and $H_\alpha=H^+_\alpha \psi_\alpha$.
\end{lem}
\begin{proof}
    For each $U_\alpha$ let us choose $x_\alpha\in \Theta_\alpha (U_\alpha^+\setminus \supp\psi^+_\alpha)$.
    We start assuming that $h$ is a smooth one form and we define, for all $x\in U_\alpha$,
    \begin{equation}\label{eq:primitiva}
        \bar H_\alpha (x)=\int_0^1\Theta_{\alpha*} h_{x_\alpha(1-t)+tx}(x-x_\alpha) dt.
    \end{equation}
    Also, for simplicity of notation, we confuse $h$ and $\Theta_{\alpha*}h=:\sum_i h_idx_i$ and set $\gamma_{\lin,x}(t)=x_\alpha(1-t)+tx$.
    Then
    \begin{equation}\label{eq:basic_der}
        \begin{split}
            \partial_{x_i}\bar H_\alpha(x)
            &= \sum_{k=1}^2\int_0^1 \left[t(\partial_{x_i}h_k)\circ \gamma_{\lin,x}(t) (x-x_\alpha)_k+h_i\circ \gamma_{\lin,x}(t)\right]\\
            &=\sum_{k=1}^2\int_0^1 t(\partial_{x_k}h_i)\circ \gamma_{\lin,x}(t) (x-x_\alpha)_k+h_i\circ \gamma_{\lin,x}(t)\\
            &\quad -\int_0^1 t\, dh(x-x_\alpha, e_i)\\
            &=\int_0^1 \frac{d}{dt}\left[t h_i\circ \gamma_{\lin,x}(t) \right]
            -\int_0^1 t\, d_{\gamma_{\lin,x}(t)}h(x-x_\alpha, e_i)\\
            &=h_i(x)-\int_0^1 t\, d_{\gamma_{\lin,x}(t)}h(x-x_\alpha, e_i).
        \end{split}
    \end{equation}
    Thus, if $h$ is a closed form, then we have $d\bar H_\alpha=h$.

    Next, let $\gamma\in \Sigma_\alpha$ and $\vf\in \cC^{q}(\gamma)$,  and set $H_\alpha=\psi_\alpha \bar H_\alpha$, $\vf_\alpha=\vf\psi_\alpha\circ \gamma$, then
    \[
        \begin{split}
            &\int_\gamma \vf\cdot   H_\alpha=\sum_{i=1}^2\int_a^b ds \int_0^1 dt  \vf_\alpha (s) \langle dx_i, \Theta_{\alpha,*} h\rangle(x_\alpha(1-t)+t\gamma(s))\cdot (\gamma(s)-x_\alpha)_i\\
            &=\sum_{i=1}^2 \int_0^1 dt\, t^{-1} \!\!\!\int_{ta}^{tb} ds \vf_\alpha (t^{-1}s) \langle dx_i,\Theta_{\alpha,*} h\rangle(x_\alpha(1-t)+t\gamma(t^{-1}s))\cdot (\gamma(t^{-1}s)-x_\alpha)_i .
        \end{split}
    \]
    If we define $\gamma_t(s)=x_\alpha(1-t)+t\gamma(t^{-1}s)$, then $\gamma'_t(s)=\gamma'(t^{-1}s)\in\cC^s$, and setting $\bar \vf_{\alpha,t}=\sum_{i=1}^2\vf_\alpha(t^{-1}s)(\gamma(t^{-1}s)-x_\alpha)_i dx_i$, we have, for some $c_{\alpha}\in (0,1)$,
    \begin{equation}\label{eq:close1}
        \int_\gamma \vf\cdot  H_\alpha=\int_{c_\alpha}^1 dt \, t^{-1} \int_{\gamma_t}   \langle\bar \vf_{\alpha,t}, h\rangle.
    \end{equation}
    Equation \eqref{eq:close1} implies that $H_\alpha$ is a continuous functional of $h$ hence it can be extended to all $h\in\cB^{0,q,1}$. By the same scheme we can define $H^+_\alpha=\psi^+_\alpha \bar H_\alpha$ when $h\in\cB^{0,q,1}$.
    Next, setting $x_{t,s}:=x_\alpha(1-t)+t\gamma(s)$ and using \eqref{eq:basic_der}, we have
    \begin{equation}\label{eq:derivative-H}
        \begin{split}
            \int_\gamma \vf\partial_{x_i}  H_\alpha=&\int_\gamma \vf \psi_\alpha \langle dx_i, h\rangle + \int_0^1 dt\, \int_{\gamma_t} \langle \vf_{\alpha,t},*dx_i\rangle *d h\\
            &+\int_\gamma \langle \vf dx_i,d \psi_\alpha \rangle  H^+_\alpha.
        \end{split}
    \end{equation}
    Hence, if $h$ is closed, $dH_\alpha= \psi_\alpha h +H^+_\alpha d\psi_\alpha$.
    If $h\in\cB^{1,q,1}$ is closed, then there exist smooth forms $h_n$ that converge to $h$. Moreover, by Lemmata \ref{lem:d-action}, \ref{lem:duality} it follows that $dh_n\to 0$ in $\cB^{0,q+1,2}$, hence equation \eqref{eq:derivative-H} implies
    \[
        \|H_{\alpha,n}-H_{\alpha,m}\|_{1, q-1,0}\leq \|h_n-h_m\|_{0,q,1}+\Const \|dh_n-dh_m\|_{0,q+1,2},
    \]
    thus $H_{\alpha,n}$ is a Cauchy sequence in $\cB^{1,q-1,1}$. Analogously, one can prove that $H^+_{\alpha,n}$ is Cauchy and, calling $H_\alpha$, $H^+_\alpha$ the limits, we have $dH_\alpha= \psi_\alpha h +H^+_\alpha d\psi_\alpha$.

    Similar arguments show that if $h\in \cB^{p,q,1}$ and closed then  $H_\alpha, H^+_\alpha\in\cB^{p+1,q-1,0}$ and $dH_\alpha= \psi_\alpha h +H^+_\alpha d\psi_\alpha$.
\end{proof}
\begin{lem}\label{lem:1-chain}
    There exist constants $c_{\alpha,\beta}\in\bC$ such that, for all $\alpha,\beta$,
    \[
        \psi_\alpha\psi_\beta [H^+_\alpha-H^+_\beta+c_{\alpha,\beta}]=0.
    \]
\end{lem}
\begin{proof}
    By Lemma \ref{lem:primitiva} follows
    \[
        \begin{split}
            d([H^+_\alpha-H^+_\beta]\psi_\alpha\psi_\beta)&=d(H_\alpha\psi_\beta-H_\beta\psi_\alpha)\\
            &=H^+_\alpha \psi_\beta d\psi_\alpha+ H_\alpha d\psi_\beta-H^+_\beta \psi_\alpha d\psi_\beta+ H_\beta d\psi_\alpha\\
            &=[H^+_\alpha-H^+_\beta]d(\psi_\alpha\psi_\beta).
        \end{split}
    \]
    This implies $\psi_\alpha\psi_\beta d[H^+_\alpha-H^+_\beta]=0$ and the Lemma follows  thanks to the last assertion of Lemma \ref{lem:d-action}.
\end{proof}
This fact allows to obtain our basic result.
\begin{lem}\label{lem:derham}
    The anisotropic de Rham cohomology for one forms is isomorphic to the standard  de Rham cohomology.
\end{lem}
\begin{proof}
    The first task is to understand when $h\in\cB^{p,q,1}$ is exact.
    Let $\bar c=(c_\alpha)\in \bC^{N}$, where $N=\sharp\{U_\alpha\}$, and define $H(\bar c)=\sum_\alpha (H^+_\alpha+c_\alpha)\psi_\alpha$. If $h$ is exact, then there exists $\theta\in\cB^{p+1,q-1,0}$ such that $d\theta=h$ but then\footnote{ Note that Lemma \ref{lem:primitiva} implies that $\psi_\alpha dH^+_\alpha=h\psi_\alpha$.}
    \[
        \psi_\alpha d(\theta -H^+_\alpha)=0.
    \]
    Then Lemma \ref{lem:d-action} implies that there exists $c_\alpha$ such that $\psi_\alpha (\theta -H^+_\alpha-c_\alpha)=0$, hence for such a collection of constants $\bar c=\{c_\alpha\}$ we have $\theta=H(\bar c)$. It follows $h$ is exact if and only if it is possible to choose $\bar c$ so that $dH(\bar c)=h$.

    To start with we have thus to compute
    \begin{equation}\label{eq:exact}
        dH(\bar c) =\sum_{\alpha} \psi_\alpha h+\sum_\alpha (H^+_\alpha+c_\alpha)d\psi_\alpha=h+\sum_{\alpha,\beta}(H^+_\alpha+c_\alpha)\psi_\beta d\psi_\alpha.
    \end{equation}
    Accordingly, if
    \begin{equation}\label{eq:constriant}
        (H^+_\alpha+c_\alpha -H^+_\beta-c_\beta)\psi_\beta d\psi_\alpha=0,
    \end{equation}
    then,
    \[
        \sum_{\alpha,\beta}(H^+_\alpha+c_\alpha)\psi_\beta d\psi_\alpha=\sum_{\alpha,\beta}(H^+_\beta+c_\beta)\psi_\beta d\psi_\alpha=\sum_{\beta}(H^+_\beta+c_\beta)\psi_\beta\, d\left(\sum_\alpha\psi_\alpha\right)=0,
    \]
    and, recalling equation \eqref{eq:exact}, $dH(\bar c)=h$. To conclude note that the problem is now reduced to the study of the  \v{C}ech cohomology $\check H^1(\cU,\bC)$ where $\cU=\{U_\alpha\}$. Indeed, a $1$-cochain $f$ is a $1$-cocycle iff for each 2-simplex $(U_{\alpha_0}, U_{\alpha_1}, U_{\alpha_2} )$ holds:\footnote{ Recall that $\{U_{\alpha_0}, \dots, U_{\alpha_q}\}$ is a $q$-simplex if $\cap_{i=0}^qU_{\alpha_i}\neq\emptyset$ while a $q$-cochain is a function from the $q$-simplex to $\bC$.}
    \begin{equation}\label{eq:cocycle}
        f(U_{\alpha_1},U_{\alpha_2})-f(U_{\alpha_0},U_{\alpha_2})+f(U_{\alpha_0}, U_{\alpha_1})=0
    \end{equation}
    while it is a coboundary if there exists a $0$-cochain $g$ such that for all $1$-simplex $(U_{\alpha_0},U_{\alpha_1})$ holds
    \begin{equation}\label{eq:coboundary}
        f(U_0,U_1)=g(U_0)-g(U_1).
    \end{equation}
    Accordingly, we can interpret the constants $\bar c=\{c_\alpha\}$ as $0$-cochain and the constants $\bar C=\{c_{\alpha,\beta}\}$, in Lemma \ref{lem:1-chain}, as a $1$-cochain. Then Lemma \ref{lem:1-chain} implies that $\bar C$ must be a $1$-cocycle. To see it, given any 2-simplex $\{U_{\alpha_0}, U_{\alpha_1}, U_{\alpha_2}\}$ consider any smooth function $\vf$ such that its support is strictly contained in $U_{\alpha_0}\cap U_{\alpha_1}\cap U_{\alpha_2}$, then, by Lemma \ref{lem:1-chain} and the definition of $\{\psi_\alpha\}$,
    \[
        \begin{split}
            0=&\int_M \vf \left[ H^+_{\alpha_1}-H^+_{\alpha_2}+c_{\alpha_1,\alpha_2}-H^+_{\alpha_0}+H^+_{\alpha_2}-c_{\alpha_0,\alpha_2}
            +H^+_{\alpha_0}-H^+_{\alpha_1}+c_{\alpha_0,\alpha_1}\right]\\
            =&\int_M \vf \left[ c_{\alpha_1,\alpha_2}-c_{\alpha_0,\alpha_2}+c_{\alpha_0,\alpha_1}\right]
        \end{split}
    \]
    which implies $c_{\alpha_1,\alpha_2}-c_{\alpha_0,\alpha_2}+c_{\alpha_0,\alpha_1}=0$ by the arbitrariness of $\vf$.

    On the other hand equation \eqref{eq:constriant} is satisfied iff $\bar C$ is a $1$-coboundary.  To see this, let $\{U_{\alpha_0}, U_{\alpha_1}\}$ be a
    1-simplex. We can assume w.l.o.g. that $\psi_{\alpha_0}d\psi_{\alpha_1}\neq 0$ otherwise $\psi_{\alpha_1}$ would be constant different from zero and one on $\supp(\psi_{\alpha_0})$. But then for each sufficiently small $\theta$ such that $\supp(\theta)\subset \supp(\psi_{\alpha_0})$ the set  $\{\tilde \psi_\alpha\}:=\{\psi_\alpha\}_{\alpha\not\in\{\alpha_0,\alpha_1\}}\cup\{\psi_{\alpha_1}-\theta, \psi_{\alpha_0}+\theta\}$ would still be a partition of unity subordinated to $\cU$ and one can choose $\theta$ such that $\tilde \psi_{\alpha_0}d\tilde \psi_{\alpha_1}\neq 0$.
    We can then find an open set $U\subset U_{\alpha_0}\cap U_{\alpha_1}$ such that $\psi_{\alpha_0}d\psi_{\alpha_1}\neq 0$ in $U$. Then, using equation \eqref{eq:constriant} multiplied by $\vf (\psi_{\alpha_0}d\psi_{\alpha_1})^{-1}$ and the statement of Lemma \ref{lem:1-chain} multiplies by $\vf(\psi_{\alpha_0}\psi_{\alpha_1})^{-1}$, for each $\vf$ supported in $U$ we have
    \[
        0=\int_M\vf \left[H^+_{\alpha_0}+c_{\alpha_0} -H^+_{\alpha_1}-c_{\alpha_1}-H^+_{\alpha_0}+H^+_{\alpha_1}-c_{\alpha_0,\alpha_1}\right]
        =\int_M\vf \left[c_{\alpha_0} -c_{\alpha_1}-c_{\alpha_0,\alpha_1}\right]
    \]
    which, by the arbitrariness of $\vf$ implies $c_{\alpha_0,\alpha_1}=c_{\alpha_0} -c_{\alpha_1}$.

    The above discussion implies that $h$ is exact if and only if $\bar C$ is a $1$-coboundary. This implies the $\cB^{p,q,1}$ cohomology is isomorphic to the \v{C}ech cohomology, which is isomorphic to the de Rham cohomology.
\end{proof}

%%%%%%%%%%%%%%%%%%%%%%%%%%%%%%%%%
\subsection{Conclusion and comparisons}\label{sec:5comp}\ \\
While the results for the simple case studied in Section~\ref{sec:piecewise} are fully satisfactory, the results in Section~\ref{sec:smoothexp}-\ref{sec:hyp} are still partial.
Indeed, we show that the preset approach yields rather sharp results for the operator associated to the measure of maximal entropy, but less information is obtained, e.g., for the operator associated to the SRB measure.
It is possible that considering the commutation of different operators with the transfer operator more information can be obtained, but this requires further work.

Also, in sections~\ref{sec:smoothexp} and \ref{sec:non-unifexp} we consider only one dimensional maps, yet the present approach seems amenable to extension to the higher dimensional setting.
In particular, the arguments of section~\ref{sec:non-unifexp} should allow to considerably improve \cite{CV13}, at least for small potentials.

In the case of two dimensional hyperbolic maps, presented in section \ref{sec:hyp}, our approach reproduces in a unified manned all the known results.
Theorems  \ref{thm:max_hyp} and \ref{thm:maximal_entropy_gap} are a refinement of \cite[Corollary 2.5]{Baladi20}, which contains slightly stronger results than \cite{Forni20a}.
In addition, for the application to toral parabolic flows, we can obtain the exact equivalent of \cite[Corollary 2.3]{Baladi20} which is sharper than the corresponding results in \cite{Forni20a}. Indeed, if $h_t$ is the unit speed flow along the stable manifold of an Anosov map $f$ then our results yield (see \cite{GL19} for details)
\[
    \left|\int_0^T g\circ h_t (x)dt-T\mu_{\operatorname{top}}(g)\right|\leq \Const \|g\|_\infty
\]
which implies that the ergodic average either grows linearly, or $g$ is a cocycle.
(See also \cite{Carrand20} for a very recent and short proof of a logarithmic bound in a more general setting.)

We have thus seen that the present approach both reproduces the results in  \cite{Baladi20}, and enlightens  the connection with the action in cohomology  (already present, in some form, in \cite{Forni20a,Forni20b}).

In conclusion, the present strategy unifies and refines the existing results in all the cases we have presented. In addition, it appears amenable to further generalisation. In particular, it seems possible to extend it to the higher dimensional case.

Another promising direction would be to apply it to Anosov flows where some hints of the relevance of some type of cohomology already exists (e.g. see \cite{Ts18}). Along the same lines, it is reasonable that our ideas can yield relevant results if applied to pseudo-Anosov and partially hyperbolic maps.

\appendix
\section{Proof of Theorem~\ref{thm:large-gap1}}\label{app:one}
Before proving Theorem \ref{thm:large-gap1} we need a few preliminary lemmata.

In this case it is convenient to define $\psi(g)(x)=\int_0^x g(y)dy$ and
\[
    \cL_+ g=\cL_2 g+\cL_1(D_{f}\cdot \psi(g)).
\]
Note that $\cL_+$ is a positive operator: if $g\geq 0$, then
\begin{equation}\label{eq:cL+positive}
    \cL_+g\geq \cL_1(D_{f} \cdot \psi(g))\geq 0
\end{equation}
(here we use the assumption that $D_{f}\geq 0$).
This facilitates the study of its spectrum. There is an obvious connection with the operator we are interested in:
\begin{equation}\label{eq:conenction}
    \cL_{\star}g = \cL_{+} g- (\cL_1 D_{f})\cdot \int_0^1(1-y)g(y)dy,
\end{equation}
that is $\cL_\star$ is a rank one perturbation of $\cL_+$.

Before proceeding further we need some information on $\cL_+$.
Recall that, as defined in the Theorem, $\mu_*=\frac {1}{f'(1)}$.

\begin{lem} \label{lem:spec_one}
    The spectral radius of $\cL_+$, acting on $L^1$, is $\mu_*$. Moreover, $\mu_*$ is an eigenvalue of $\cL_+'$ (the dual operator to $\cL_+$), acting on $L^\infty$ with eigenvector given by the constant function one.
\end{lem}
\begin{proof}
    Note that, for all $g\in L^1$,
    \begin{equation}\label{eq:average}
        \begin{split}
            \int_0^1 1\cdot \cL_+g(y)dy&=\int_0^1\left[\frac{g(y)}{f'(y)} +\left(\frac 1{f'(y)}\right)'\psi(g)(y)\right] dy=\int_{0}^{1} \left(\frac {\psi(g)}{f'}\right)'\!\!\!(y) \;dy\\
            &=\frac 1{f'(1)}\psi(g)(1)-\frac 1{f'(0)}\psi(g)(0)=\frac 1{f'(1)}\int_0^11 \cdot g(y) \ dy.
        \end{split}
    \end{equation}
    Hence, $\frac 1{f'(1)}$ is an eigenvalue of the dual of $\cL_+$ and hence it belongs to the spectrum of $\cL_+$.
    The lemma follows since
    \begin{equation}
        \label{eq:L1}
        \begin{split}
            \int_0^1\left|\cL_+g(y)\right|dy&\leq \int_0^1\cL_+|g|(y)dy=\frac 1{f'(1)}\int_0^1 |g(y)| \ dy.
        \end{split}
    \end{equation}
\end{proof}
\noindent Note that the above Lemma implies that the space $\bV_0=\{h\in L^1\;|\; \int_0^1 h=0\}$ is invariant under $\cL_+$. However, this does not give much information on the spectrum. To learn more it is convenient to study the operator $\cL_+$ acting on $W^{1,1}$.
\begin{lem}\label{lem:W11}
    For all $g\in W^{1,1}$ we have
    \[
        \begin{split}
            & \|\cL_+g\|_{L^1}\leq\mu_* \|g\|_{L^1}\\
            &  \|\cL_+g\|_{W^{1,1}}\leq  \mu_*^{2}\|g\|_{W^{1,1}}+(3{\|D_{f}\|}_\infty+{\|D_{f}'\|}_{L^1}+{\|D_{f}^2\|}_{L^1}+\mu_*)\|g\|_{L^1}.
        \end{split}
    \]
\end{lem}
\begin{proof}
    The first inequality follows from \eqref{eq:L1}. Next, for each $g\in W^{1,1}$, using again \eqref{eq:derivL}, we have
    \begin{equation}\label{eq:oneder}
        (\cL_+g)'=\cL_3 g'+3\cL_2 D_{f}g+\cL_2D_{f}'\psi(g)+\cL_1 D_{f}^2\psi(g).
    \end{equation}
    Thus  (note that \(D_{f}\geq0\) implies that \(f''\leq 0\) and so \(f'(0)\geq f(x) \geq f'(1)\)), using also the estimates of Lemma~\ref{lem:Phi-est},
    \[
        \|(\cL_+g)'\|_{L^1}\leq  \mu_*^{2}{\|g'\|}_{L^1}+(3{\|D_{f}\|}_\infty+{\|D_{f}'\|}_{L^1}+{\|D_{f}^2\|}_{L^1}){\|g\|}_{L^1}.
    \]
    The Lemma follows using again \eqref{eq:L1}.
\end{proof}
\begin{lem}\label{lem:gap}
    $\mu_*$ is a simple eigenvalue of $\cL_+$.
    Moreover, \(h_{+}\), the eigenvector associated to \(\mu_{*}\), is strictly positive, i.e., \(h_{+}>0\).
    In addition, there exists $\mu_1<\mu_*$ such that $\sigma_{W^{1,1}}(\cL_+)\subset \{\mu_*\}\cup\{z\in\bC\;:\; |z|\leq \mu_1\}$.
\end{lem}
\begin{proof}
    Lemma \ref{lem:W11} and \cite{Hennion93} (see also \cite[Appendix B]{DKL}) imply that the essential spectrum of $\cL_+$ is contained in a disk of size $\mu_*^2$. Thus, since $\mu_*$ is an eigenvalue of $\cL_+'$, $\mu_*$ must be an eigenvalue of $\cL_+$. Moreover, Lemma \ref{lem:W11} implies that $\{\mu_*^n\cL_+^n\}$ is uniformly bounded when acting on $W^{1,1}$, hence by \cite[Lemma VIII.8.1]{DS58} $\mu_*$ is a semi-simple eigenvalue (no Jordan blocks).

    Let $h_+\in W^{1,1}\setminus\{0\}\subset\cC^0$ be a corresponding eigenvector.
    Next, suppose that $\cL_+ g=\mu_*e^{i\theta}g$, then $\mu_*|g|\leq \cL_+|g|$, but then
    \[
        \int_0^1 \cL_+|g|-\mu_*|g|=0
    \]
    thus $\mu_*|g|= \cL_+|g|$. Accordingly, we can assume that $h_+\geq 0$. But then it must be $h_+>0$. Indeed, if there exists $\bar x$ such that $h_+(\bar x)=0$, then, calling $y$ the maximal element in $f^{-1}(\bar x)$,
    \[
        0=\mu_*h_+(\bar x)\geq \frac1{f'(y)}\left(\frac 1{f'}\right)'\!\!\!(y)\int_0^y h_+.
    \]
    Hence $h_+(x)=0$ for all $x\leq y$. Iterating the argument we have that $h_+(x)=0$ for all $x<1$, and, by continuity, $h_+\equiv 0$, contrary to the assumption. Accordingly, if there exists another $h$ such that $\cL_+ h=\mu_* h$, then it cannot be zero anywhere otherwise $|h|$, which is also an eigenvector, would be identically zero. But then there exists $\alpha\in\bR$ such that $\alpha h_+- |h|$ has a zero and hence, since $\alpha h_+- |h|$ also is an eigenvalue, $h=\alpha h_+$. Hence, $\mu_*$ is simple (the corresponding eigenspace has dimension one).

    Therefore, if $e^{i\theta}\mu_* g= \cL_+ g$ then there must exist $\vartheta\in \cC^0$ such that $g=e^{i\vartheta} h_+$. It follows
    \[
        \begin{split}
            0=&\mu_* h_+-e^{-i\theta-i\vartheta}\cL_+(e^{i\vartheta}h_+) =\cL_+ h_+-e^{-i\theta-i\vartheta}\cL_+(e^{i\vartheta}h_+)\\
            &=\cL_2\left[ 1-e^{-i\theta-i\vartheta\circ f+i\vartheta}\right] h_+ +\cL_1D_{f}\left[\psi(g)- e^{-i\theta-i\vartheta\circ f}\psi(e^{i\vartheta} h_+)\right].
        \end{split}
    \]
    Taking the real part and integrating yields
    \[
        0= \int_0^1\frac{1-\cos[\theta+\vartheta\circ f-\vartheta]}{f'}h_+ +\int_0^1\!\!\!\!dx D_{f}(x)\int_0^x \!\!\!\!dy \left[1- \cos[\theta+\vartheta\circ f(x)-\vartheta(y)\right]h_+(y).
    \]
    Since both terms are positive, the only possibility is $\theta+\vartheta\circ f(x)-\vartheta(y)=k\pi$. This implies that $\vartheta$ is constant and hence $g$ is proportional to $h_+$, hence it must be $\theta=0$. This proves that $\mu_*$ is the only peripheral eigenvalue and the spectral gap.
\end{proof}
We can now conclude our argument.
\begin{proof}[{\bf Proof of Theorem~\ref{thm:large-gap1}}]
    Equation \eqref{eq:conenction} and Lemma \ref{lem:W11} imply the bound on the essential spectral radius.

    Since  \(f'\) is continuous on \([0,1]\) we know that
    \[
        \int_0^1|(\tfrac 1{f'(y)})'|  dy
        =  \frac{1}{f'(1)} -   \frac {1}{f'(0)}=:\Delta.
    \]
    This means that the first statement of the theorem follows from Theorem~\ref{thm:simple-bound} where
    \[
        \tau=\lambda^{-1} +\int_0^1\left|\left(\tfrac 1{f'(y)}\right)'\right|  dy
        =
        \tfrac 2{f'(1)}-\tfrac 1{f'(0)}=\Delta+\mu_*.
    \]

    It remains to show the absence of eigenvalues in the sets $A_0,\ldots,A_4$.
    By equation \eqref{eq:conenction} we have that if $\cL_\star g=zg$ then
    \begin{equation}\label{eq:projected}
        (z-\cL_+)g=-\cL_1 D_{f}\int_0^1(1-y)g(y) \ dy.
    \end{equation}
    Recall  $\mu_1<\mu_*$  from Lemma~\ref{lem:gap}.
    Let  $|z|>\mu_1$ and suppose, for sake of contradiction, that the right hand side of the above equation is zero. This would mean, by Lemma \ref{lem:gap}, that $g=h_+$ and $z=\mu_*$. However this would then imply that the integral on the right was strictly positive and, since $D_{f}\not\equiv 0$, would contradict the assumption.
    This means that the  right hand side of the above equation cannot be zero.
    Moreover, $z\neq \mu_*$ since, if it were,
    $
        \int_0^1(\mu_*-\cL_+)g=0
    $
    by Lemma \ref{lem:spec_one}, and this would imply that the right hand side of \eqref{eq:projected} is zero, contrary to what we have seen.

    It follows that, possibly after a normalization, for $|z|>\mu_1$ we can write $g=(z-\cL_+)^{-1}\cL_1 D_{f}$, and substituting in \eqref{eq:projected} we have
    \[
        (z-\cL_+)^{-1}\cL_1 D_{f} (x)=-(z-\cL_+)^{-1}\cL_1 D_{f}(x)\int_0^1(1-y) (z-\cL_+)^{-1}\cL_1 D_{f}(y)dy.
    \]
    Accordingly, if we define
    \begin{equation}
        \label{eq:defXi}
        \Xi(z):=1+\int_0^1(1-y) (z-\cL_+)^{-1}\cL_1 D_{f}(y)dy
    \end{equation}
    we have that $z$ is an eigenvalue of \(\cL_\star\) if and only if $\Xi(z)=0$.
    In the following we will repeatedly use the following facts: By Lemma \ref{lem:spec_one}, the spectral radius of $\cL_+$ is $\mu_*$, and, for each $\vf\in L^1$, $\int_0^1\cL_+\vf=\mu_*\int_0^1\vf$, while, by definition, $\int_0^1\cL_1\vf=\int_0^1\vf$. Also, by \eqref{eq:cL+positive}, $\cL_+$ is a positive operator and so is, obviously, $\cL_1$. Moreover we set \(\int_0^1 D_{f} = \Delta\).

    Since by Lemma \ref{lem:spec_one} we have $\|\cL_+\|_{L^1}\leq \mu_*$,  for $|z|\geq \mu_*$ we can do the trivial estimate
    \begin{equation}\label{eq:integral}
        \begin{split}
            &\left | \int_0^1(1-y) (z-\cL_+)^{-1}\cL_1 D_{f}(y)dy\right|\leq \sum_{n=0}^\infty\int_0^1(1-y) |z|^{-n-1}\cL_+^{n}\cL_1 D_{f}(y)dy\\
            &\leq \sum_{n=0}^\infty\int_0^1 |z|^{-n-1}\cL_+^{n}\cL_1 D_{f}(y)dy=\frac{\Delta}{|z|-\mu_*}.
        \end{split}
    \end{equation}
    The above implies that $\Xi(z)\neq 0$ for all $|z|>\tau$, which we know already. Hence, to gain further informations we have to analyse \eqref{eq:integral} more in depth.

    For $z >\mu_*$,
    \[
        \begin{split}
            \int_0^1(1-y) (z-\cL_+)^{-1}\cL_1 D_{f}(y)dy&= \sum_{n=0}^\infty \int_0^1(1-y) z^{-n-1}\cL_+^n\cL_1 D_{f}(y)dy> 0.
        \end{split}
    \]
    which implies $\Xi(z)>1$, hence non zero. Moreover, by Lemma \ref{lem:gap} we get the spectral representation $\cL_+ h=\mu_*h_+\int_0^1 h+Qh$ where for all $\mu\in(\mu_1,\mu_*)$ there exists $C_\mu$ such that $\|Q^n\|_{W^{1,1}}\leq C_\mu \mu^n$. Thus, for $z \in (\mu, \mu_*)$,
    \[
        \begin{split}
            \int_0^1(1-y) (z-\cL_+)^{-1}\cL_1 D_{f}(y)dy=& \int_0^1(1-y) (z-\mu_*)^{-1}h_+(y) dy \left(\int_0^1\cL_1 D_{f}\right)\\
            &+\sum_{n=0}^\infty \int_0^1(1-y) z^{-n-1}Q^n\cL_1 D_{f}(y)dy.
        \end{split}
    \]
    Thus, for some $C>0$,
    \[
        \left|\Xi(z)-1- (z-\mu_*)^{-1}\left[\frac{1}{f'(1)}-\frac{1}{f'(0)}\right]\int_0^1(1-y) h_+\right|\leq C_\mu C(|z|-\mu)^{-1}
    \]
    Hence, there exists $\mu_2\in (\mu_1,\mu_*)$ such that $\Xi(z)=0$ has no solution for $z>\mu_2$.\footnote{ With some further work one could estimate $\mu_2$, but we believe the above suffices to show how to proceed.} This establishes the fact that the set $A_0$ does not belong to the spectrum.

    To study non positive $z$, note that $\overline{\Xi(z)}=\Xi(\bar z)$  and that \(( z-\cL_+)^{-1} - (\bar z-\cL_+)^{-1} = (\bar z - z)( z-\cL_+)^{-1}(\bar z-\cL_+)^{-1}\). Hence, if $z=a+ib$, we have
    \begin{equation}\label{eq:complex0}
        \begin{split}
            \Im(\Xi(z))&=ib\int_0^1(1-y) (\bar z-\cL_+)^{-1}(z-\cL_+)^{-1}\cL_1 D_{f}(y)dy\\
            &= ib\int_0^1(1-y) ([a-\cL_+]^2+b^2)^{-1}\cL_1 D_{f}(y)dy\\
            \Re(\Xi(z))&=1+\int_0^1(1-y)(a-\cL_+) ([a-\cL_+]^2+b^2)^{-1}\cL_1 D_{f}(y)dy.
        \end{split}
    \end{equation}
    Note that if $a> \mu_*$, then, by Lemma \ref{lem:spec_one} and \eqref{eq:cL+positive}, $[a-\cL_+]^{-1}=\sum_{n=0}^\infty a^{-n-1}\cL_+^n$ is a positive operator. Hence, if $b\neq 0$, $\Xi(z)=0$ implies
    \begin{equation}\label{eq:uff}
        \begin{split}
            0&=1+\int_0^1(1-y) (\Id+[a-\cL_+]^{-2}b^2)^{-1}[a-\cL_+]^{-1}\cL_1 D_{f}(y)dy\\
            &=1+\int_0^1(1-y) [a-\cL_+]^{-1}\cL_1 D_{f}(y)dy\\
            &\phantom{=}
            -\int_0^1(1-y) (\Id+[a-\cL_+]^{-2}b^2)^{-1}[a-\cL_+]^{-3}b^2\cL_1 D_{f}(y)dy.
        \end{split}
    \end{equation}
    Next we want to estimate the two integrals on the right hand side of \eqref{eq:uff}. Let us start with the first integral, assuming $a>\mu_*$,\footnote{ Note that in this computation we are sacrificing optimality of the result to the simplicity of the formulae.}
    \[
        \begin{split}
            \int_0^1(1-y) [a-\cL_+]^{-1}\cL_1 D_{f}(y)dy&\geq a^{-1}\int_0^1(1-y) \cL_1 D_{f}(y)dy\\
            &=a^{-1}\sum_{i=1}^N\int_{p_i}^{p_{i+1}}(1-f(y))\left(\frac1{f'(y)}\right)'dy\\
            &=a^{-1}\left(1-\sum_{i=1}^N\frac1{f'(p_i)}\right)=:a^{-1}\Gamma.
        \end{split}
    \]
    To estimate the second integral note that, for  each $g\geq 0$ and $|b|< a-\mu_*$, we can write
    \begin{equation}\label{eq:cancellation2}
        \begin{split}
            (\Id+[a-\cL_+]^{-2}b^2)^{-1}g&=\sum_{n=0}^\infty (-1)^{n}[a-\cL_+]^{-2n}b^{2n}g\\
            &\leq \sum_{n=0}^\infty [a-\cL_+]^{-4n}b^{4n}g.
        \end{split}
    \end{equation}
    Hence
    \[
        \begin{split}
            &\int_0^1(1-y) (\Id+[a-\cL_+]^{-2}b^2)^{-1}[a-\cL_+]^{-3}b^2\cL_1 D_{f}(y)dy\leq\\
            &\leq \sum_{n=0}^\infty\int_0^1 [a-\cL_+]^{-4n-3}b^{4n+2}\cL_1 D_{f}(y)dy\\
            &=\sum_{n=0}^\infty [a-\mu_*]^{-4n-3}b^{4n+2}\Delta=\frac{(a-\mu_*)b^2\Delta}{(a-\mu_*)^4-b^4}.
        \end{split}
    \]
    The above implies that $\Xi(z)=0$ has no solution if $a>\mu_*$, $|b|< a-\mu_*$ and
    \[
        1+a^{-1}\Gamma>\frac{b^2(a-\mu_*)\Delta}{(a-\mu_*)^4-b^4}
    \]
    which is implied by $a>\mu_*$ and
    \begin{equation}\label{eq:complex1}
        b^2 <\frac{(a-\mu_*)\Delta}{2(1+a^{-1}\Gamma)}\left[\sqrt{1+4\frac{(1+a^{-1}\Gamma)^2(a-\mu_*)^2}{\Delta^2}}-1\right].
    \end{equation}
    This establishes that $A_1$ is disjoint from the spectrum (apart from $1$).

    If  $a<-\mu_*$,  we cannot easily use positivity arguments since
    \[
        \Re(\Xi(z))=1-\int_0^1(1-y)(|a|+\cL_+) ([|a|+\cL_+]^2+b^2)^{-1}\cL_1 D_{f}(y)dy.
    \]
    We are thus left with the cruder estimate,
    $|(|a|+\cL_+)^{-1}g|\leq \sum_{n=0}^\infty |a|^{-n-1}\cL_+^n|g|=(|a|-\cL_+)^{-1}|g|$. Hence, for $|b|\leq |a|-\mu_*$ and recalling the computation for $b=0$, we have

    \[
        \begin{split}
            &\int_0^1(1-y)(|a|+\cL_+) ([|a|+\cL_+]^2+b^2)^{-1}\cL_1 D_{f}(y)dy\leq \int_0^1(1-y)(|a|+\cL_+)^{-1} \cL_1 D_{f}(y)dy\\
            &+\sum_{n=1}^\infty\int_0^1b^{2n}[|a|-\cL_+]^{-2n-1}\cL_1 D_{f}(y)dy.
        \end{split}
    \]
    The above decomposition allows to still use some positivity argument for the first term after the inequality. Indeed,
    \[
        \begin{split}
            \int_0^1(1-y) (|a|+\cL_+)^{-1}\cL_1 D_{f}(y)dy&= \sum_{n=0}^\infty |a|^{-1-n}(-1)^{n}\int_0^1(1-y) \cL_+^{n}\cL_1 D_{f}(y)dy\\
            &\leq \sum_{n=0}^\infty |a|^{-1-2n}\int_0^1 \cL_+^{2n}\cL_1 D_{f}(y)dy\\
            &\leq \sum_{n=0}^\infty |a|^{-1-2n}\mu_*^{2n}\Delta=\frac{|a|\Delta}{a^2-\mu_*^2}.
        \end{split}
    \]
    Accordingly,
    \[
        \begin{split}
            \int_0^1(1-y)(|a|+\cL_+) ([|a|+\cL_+]^2+b^2)^{-1}\cL_1 D_{f}(y)dy\leq &\frac{|a|\Delta}{a^2-\mu_*^2}\\
            &+\frac{\Delta b^2}{(|a|-\mu_*)[(|a|-\mu_*)^2-b^2]}.
        \end{split}
    \]
    which implies that $\Xi(x)=0$ has no solutions if $a<-\mu_*$ and
    \begin{equation}\label{eq:negative_bound}
        b^2< \frac{(|a|-\mu_*)^2(a^2-\mu_*^2-\Delta|a|)}{a^2-\mu_*^2+\Delta|a|}.
    \end{equation}
    This establishes that $A_2$ is disjoint from the spectrum.

    On the other hand, for $b\neq 0$, the equations \eqref{eq:complex0} and $\Xi(z)=0$ imply also
    \begin{equation}\label{eq:complex}
        0=1-\int_0^1(1-y)\cL_+ ([a-\cL_+]^2+b^2)^{-1}\cL_1 D_{f}(y)dy.
    \end{equation}

    If $|z|^2>2|a|\mu_*$, then
    \[
        \begin{split}
            &\int_0^1(1-y)\cL_+ (|z|^2-2a\cL_++\cL_+^2)^{-1}\cL_1 D_{f}(y)dy= \\
            &=\int_0^1(1-y)\cL_+ (\Id+(|z|^2-2a\cL_+)^{-1}\cL_+^2)^{-1}(|z|^2-2a\cL_+)^{-1}\cL_1 D_{f}(y)dy.
        \end{split}
    \]
    Note that if $a\geq 0$, then $(|z|^2-2a\cL_+)^{-1}$ is a positive operator, hence
    \[
        \begin{split}
            &\int_0^1(1-y)\cL_+ ([a-\cL_+]^2+b^2)^{-1}\cL_1 D_{f}(y)dy=\\
            &=\sum_{n=0}^\infty(-1)^n\int_0^1(1-y)\cL_+ (|z|^2-2a\cL_+)^{-n-1}\cL_+^{2n}\cL_1 D_{f}(y)dy\\
            &\leq\sum_{n=0}^\infty\int_0^1\cL_+ (|z|^2-2a\cL_+)^{-2n-1}\cL_+^{4n}\cL_1 D_{f}(y)dy\\
            &=\sum_{n=0}^\infty \mu_* (|z|^2-2a\mu_*)^{-2n-1}\mu_*^{4n}\Delta=\frac{\mu_*\Delta(|z|^2-2a\mu_*)}{(|z|^2-2a\mu_*)^2-\mu_*^4}.
        \end{split}
    \]
    It follows that equation \eqref{eq:complex} has no solution if $a>0$, $|z-\mu_*|>\mu_*$ and
    \[
        \mu_*\Delta(|z-\mu_*|^2-\mu_*^2)\leq(|z-\mu_*|^2-\mu_*^2)^2-\mu_*^4,
    \]
    that is $a>0$ and
    \begin{equation}\label{eq:complex2.0}
        |z-\mu_*|^2>\mu_*^2+\mu_*\frac{\Delta+\sqrt{4\mu_*^2+\Delta^2}}2.
    \end{equation}
    This establishes that $A_3$ is disjoint from the spectrum.

    If $a<0$, again it is harder to take advantage of the positivity, hence, for simplicity, we content ourselves with a cruder estiamte
    \[
        \begin{split}
            &\int_0^1(1-y)\cL_+ (|z|^2-2a\cL_++\cL_+^2)^{-1}\cL_1 D_{f}(y)dy= \\
            &= \sum_{n=0}^\infty |z|^{-2-2n}\int_0^1(1-y)\cL_+ (2a\cL_+-\cL_+^2)^{n}\cL_1 D_{f}(y)dy\\
            &\leq \sum_{n=0}^\infty |z|^{-2-2n}\int_0^1\cL_+ (\cL_+^2+2|a|\cL_+)^{n}\cL_1 D_{f}(y)dy\\
            &=\sum_{n=0}^\infty |z|^{-2-2n}\mu_*(\mu_*^2+2|a|\mu_*)^{n}\Delta=\frac{\mu_*\Delta}{|z|^2-\mu_*^2-2|a|\mu_*}.
        \end{split}
    \]
    It follows that equation \eqref{eq:complex} has no solution if
    \begin{equation}\label{eq:complex2}
        b^2>\mu_*\Delta+\mu_*^2+2|a|\mu_*-a^2.
    \end{equation}
    Finally, this establishes that $A_4$ is disjoint from the spectrum.

    Collecting \eqref{eq:complex1}, \eqref{eq:negative_bound}, \eqref{eq:complex2.0} and \eqref{eq:complex2} completes the proof of Theorem~\ref{thm:large-gap1}.
\end{proof}
\begin{rem} The previous proof it is not optimal. We refrain form walking this road further as our purpose was only to show that the study of the spectrum can be reduced to the study of a concrete function \eqref{eq:defXi}, which plays the role of a determinant but may be easier to study than the dynamical determinant.
\end{rem}

\end{document}